\documentclass[11pt, a4paper,leqno]{amsart}
\usepackage{amsmath,amsthm,amscd,amssymb,amsfonts, amsbsy}
\usepackage{latexsym}
\usepackage{txfonts}
\usepackage{exscale}
\usepackage{graphicx}
\usepackage{bbm}
\usepackage{enumerate}
\usepackage[utf8]{inputenc}

\usepackage[colorlinks,citecolor=red,pagebackref,hypertexnames=false]{hyperref}

\usepackage{caption}
\usepackage{color}

\let\oldtocsection=\tocsection
\let\oldtocsubsection=\tocsubsection
\let\oldtocsubsubsection=\tocsubsubsection
\renewcommand{\tocsection}[2]{\hspace{0em}\oldtocsection{#1}{#2}}
\renewcommand{\tocsubsection}[2]{\hspace{2em} \oldtocsubsection{#1}{\small{#2}}}
\renewcommand{\tocsubsubsection}[2]{\hspace{4em}\oldtocsubsubsection{#1}{\scriptsize{#2}}}

\calclayout
\allowdisplaybreaks

\theoremstyle{plain}
\newtheorem{theorem}[equation]{Theorem}
\newtheorem{lemma}[equation]{Lemma}
\newtheorem{corollary}[equation]{Corollary}
\newtheorem{proposition}[equation]{Proposition}

\theoremstyle{definition}
\newtheorem{definition}[equation]{Definition}

\theoremstyle{remark}
\newtheorem{remark}[equation]{Remark}

\numberwithin{equation}{section}

\newcommand{\RR}{{\mathbb{R}}}
\newcommand{\NN}{{\mathbb{N}}}

\newcommand{\eps}{\varepsilon}

\newcommand{\dist}{\operatorname{dist}}

\newcommand{\heart}{\varheartsuit}
\newcommand{\dv}{\operatorname{div}}

\newcommand{\re}{\mathbb{R}}
\newcommand{\rn}{\mathbb{R}^n}
\newcommand{\ren}{\mathbb{R}^n}

\newcommand{\ree}{\mathbb{R}^{n+1}}

\newcommand{\cA}{\mathcal{A}}
\newcommand{\dd}{\mathbb{D}}

\newcommand{\F}{\mathcal{F}}
\newcommand{\cH}{\mathcal{H}}
\newcommand{\cB}{\mathcal{B}}

\newcommand{\sbf}{{\bf S}}

\newcommand{\Qin}{{Q_{\text{in}}}}
\newcommand{\Qout}{{Q_{\text{out}}}}

\newcommand{\pom}{\partial\Omega}

\newcommand{\hm}{\omega}

\renewcommand{\P}{\mathcal{P}}

\newcommand{\bfpsi}{\mbf{\Psi}}

\newcommand{\xbf}{\mathbf{X}}
\newcommand{\mbf}[1]{{\mathbf{#1}}}

\renewcommand{\emptyset}{\text{\textup{\O}}}

\DeclareMathOperator{\diam}{diam}
\DeclareMathOperator{\osc}{osc}

\DeclareMathOperator*{\esssup}{ess\,sup}

\def\div{\mathop{\operatorname{div}}\nolimits}
\def\Lip{\mathop{\operatorname{Lip}}\nolimits}
\def\BMO{\mathop{\operatorname{BMO_P}}\nolimits}

\newcommand{\vertiii}[1]{{\left\vert\kern-0.15ex\left\vert\kern-0.15ex\left\vert #1
		\right\vert\kern-0.15ex\right\vert\kern-0.15ex\right\vert}}

\def\Xint#1{\mathchoice
{\XXint\displaystyle\textstyle{#1}}%
{\XXint\textstyle\scriptstyle{#1}}%
{\XXint\scriptstyle\scriptscriptstyle{#1}}%
{\XXint\scriptscriptstyle%
\scriptscriptstyle{#1}}%
\!\int}
\def\XXint#1#2#3{{\setbox0=\hbox{$#1{#2#3}{%
\int}$ }
\vcenter{\hbox{$#2#3$ }}\kern-.6\wd0}}
\def\barint{\,\Xint -} 
\def\bariint{\barint_{} \kern-.4em \barint}
\def\bariiint{\bariint_{} \kern-.4em \barint}
\renewcommand{\iint}{\int_{}\kern-.34em \int} 
\renewcommand{\iiint}{\iint_{}\kern-.34em \int} 

\renewcommand{\d}{\, \mathrm{d}}

\title[Variable coefficient FBP for $L^p$-solvability]{A variable coefficient Free Boundary Problem for $L^p$-solvability of parabolic Dirichlet problems in graph domains} 
\author[S. Bortz, S. Ferris, P. Hidalgo-Palencia, S. Hofmann]
{Simon Bortz, Sandra Ferris, Pablo Hidalgo-Palencia, Steve Hofmann}

\address{Simon Bortz \& Sandra Ferris, Department of Mathematics
\\
University of Alabama
\\
Tuscaloosa, AL, 35487, USA}
\email{sbortz@ua.edu}

\address{Pablo Hidalgo-Palencia\\
	Instituto de Ciencias Matemáticas CSIC-UAM-UC3M-UCM\\
	Con\-se\-jo Superior de Investigaciones Científicas\\
	E-28049 Ma\-drid, Spain; \& Departamento de Análisis Matemático y Matemática Aplicada
	\\
	Facultad de Matemáticas
	\\
	Universidad Complutense de Madrid
	\\
	E-28040 Madrid, Spain
}
\email{pablo.hidalgo@icmat.es}

\address{Steve Hofmann\\ 
	Department of Mathematics, University of Missouri, Columbia, MO 65211, USA
}
\email{hofmanns@missouri.edu}
\date{\today}

\subjclass[2010]{}
\thanks{S.B. was  supported by the Simons foundation grant ``Travel support for Mathematicians'' (grant number 959861). P.H.-P. is supported by the grant CEX2019-000904-S-20-3, funded by MCIN/AEI/ 10.13039/501100011033, and acknowledges financial support from MCIN/AEI/ 10.13039/501100011033 grants CEX2019-000904-S and PID2019-107914GB-I00. S.H. was supported by NSF grant DMS-2349846. This project was carried out while the authors were visiting University of Alabama and the authors express their gratitude to UA. Roll Tide!}

\begin{document}
\allowdisplaybreaks

\begin{abstract}
We investigate variable coefficient analogs of  \cite{BHMN1}. In particular, we show that if $\Omega$ is the region above the graph of a Lip(1,1/2) (parabolic Lipschitz) function and $L$ is a parabolic operator in divergence form
\[L = \partial_t - \div A \nabla\]
with $A$ satisfying an $L^1$ Carleson condition on its spatial and time derivatives, then the $L^p$-solvability of the Dirichlet problem for $L$ and $L^*$ implies that the graph function has a half-order time derivative in BMO. Equivalently, the graph is parabolic uniformly rectifiable.

In the case of $A$ symmetric, we only require that the Dirichlet problem for $L$ is solvable, which requires us to adapt a clever integration by parts argument by Lewis and Nystr\"om \cite{LN07}.  A feature of the present work is that we must overcome the lack of translation invariance in our equation, which is a fundamental tool in similar works, including \cite{BHMN1}. 
\end{abstract}

\maketitle 

\tableofcontents

\section{Introduction}
In this paper we extend the results of \cite{BHMN1} to the setting of variable coefficient operators satisfying a Carleson condition on their oscillation. In particular, we consider parabolic operators in divergence form
\begin{equation} \label{L.eq}
	\mathcal{L}f := \partial_t f - \div_X(A(X,t) \nabla_X f),\qquad  \mathcal{L}^* f := -\partial_t f - \div_X(A^*(X,t) \nabla_X f),
\end{equation}
acting in $\mathbb{R}^{n+1} = \{(X,t): X \in \mathbb{R}^n, t \in \mathbb{R}\}$, where $A$ is a real $n \times n$
matrix with transpose $A^*$, satisfying the point-wise ellipticity bounds for 
some $\Lambda \ge 1$
\begin{equation}\label{ellip.eq}
\|A\|_{L^\infty} \le \Lambda, \quad \text{and} \quad
\langle A(X,t) \xi, \xi\rangle \ge \Lambda^{-1} |\xi|^2,
 \qquad \forall \; \xi \in \mathbb{R}^n, \; \text{a.e.} \, (X,t) \in 
 \mathbb{R}^{n+1}.
\end{equation}

Our results concern how the $L^p$-solvability of the Dirichlet problem for $\mathcal{L}$ and $\mathcal{L}^*$ imposes additional regularity on the boundary of certain domains. In particular we prove the following theorems.
\begin{theorem}\label{main1.thrm}
Let $n \geq 2$. Suppose that $\Omega \subseteq \RR^{n+1}$ 
is the region above the graph of a $\Lip(1,1/2)$ function $\psi$ (see Definition \ref{Lipnot}), and $A$ is a \textbf{symmetric} elliptic matrix satisfying \eqref{ellip.eq} and the $L^1$-Carleson oscillation condition in Definition \ref{smoothL1carl.def}. Assume further that there exists some $1 < p < \infty$ such that the $L^p$ Dirichlet problem is solvable for $\mathcal{L}$ (see Definition~\ref{Lpsolv.def}), where
\[\mathcal{L}f := \partial_t f - \div_X(A(X,t) \nabla_X f).\]
Then $D_t^{1/2}\psi \in \BMO$ (the parabolic $\mathrm{BMO}$ space, see Section \ref{sec:BMO}), with norm controlled uniformly depending only on the constants appearing in the hypotheses. Equivalently, the $L^p$ solvability of the Dirichlet problem for $\mathcal{L}$ implies that the graph of $\psi$ is parabolic uniformly rectifiable.
\end{theorem}

\begin{theorem}\label{main2.thrm}
Let $n \geq 2$. Suppose that $\Omega \subseteq \RR^{n+1}$ 
is the region above the graph of
a $\Lip(1,1/2)$ function $\psi$ (see Definition \ref{Lipnot}), and $A$ is a 
(not necessarily symmetric) elliptic matrix satisfying \eqref{ellip.eq} and the $L^1$-Carleson oscillation condition in Definition \ref{smoothL1carl.def}. Assume further that there exists some $1 < p < \infty$ such that the $L^p$ Dirichlet problem is solvable for both $\mathcal{L}$ and $\mathcal{L}^*$ (see Definition~\ref{Lpsolv.def}), where
\[\mathcal{L}f := \partial_t f - \div_X(A(X,t) \nabla_X f),\qquad  \mathcal{L}^* f := -\partial_t f - \div_X(A^*(X,t) \nabla_X f).\]
Then $D_t^{1/2}\psi \in \BMO$ (the parabolic $\mathrm{BMO}$ space, see Section \ref{sec:BMO}), with norm controlled uniformly depending only on the constants appearing in the hypotheses. Equivalently, the $L^p$ solvability of the Dirichlet problem for $\mathcal{L}$ and $\mathcal{L}^*$ implies that the graph of $\psi$ is parabolic uniformly rectifiable. 
\end{theorem}
The converse to Theorem \ref{main2.thrm} (hence also to
Theorem \ref{main1.thrm}) is known: 
assuming that $\Omega$ is the region above  
the graph of a $\Lip(1,1/2)$ function $\psi$ with $D_t^{1/2}\psi \in \BMO$,
and given that $A$ is an elliptic matrix satisfying \eqref{ellip.eq}, and also
that the $L^1$ Carleson oscillation condition holds, 
one obtains that there exists $p \in (1,\infty)$ 
such that the $L^p$ Dirichlet problem is solvable for 
$\mathcal{L}$ and $\mathcal{L}^*$. In fact, this converse is known under a somewhat 
weaker hypothesis on the coefficients, namely an 
$L^2$ Carleson measure condition on $\nabla A$; see \cite{RN}, which adapts
the methods of \cite{KKPT} and \cite{KP} to the
parabolic setting.  One may give an easier proof, using more recently developed techniques,
by applying the criterion in \cite{DPP17}\footnote{which in turn is a parabolic version of \cite{KKiPT}}, and using the local square function estimate in
 \cite[Lemma A.2 (i)]{HL2}\footnote{which in turn is a parabolic version of analogous elliptic bounds in \cite{KP}} (plus a pullback argument, using a parabolic version of the Dahlberg-Kenig-Stein mapping).  A slightly weaker version of the result (i.e.,
 with a more restrictive Carleson measure condition on the coefficients) 
 appeared earlier in \cite{HL-Mem}. Under the weaker $L^2$ Carleson measure 
 assumption, at present
we lack certain tools used in Theorems \ref{main1.thrm} 
and \ref{main2.thrm},
in particular, the local square function estimates of Lemma \ref{mainsfest.lem}, for the time derivative and second spatial partial derivatives of the Green function.  Thus we require the stronger $L^1$ condition in Definition \ref{smoothL1carl.def}. 
In fact, under the  $L^2$ Carleson measure 
assumption, it remains an open problem to establish
even the elliptic analogue of such square function estimates (we note that a variant of these estimates, in the elliptic setting, has been obtained in work of Feneuil and Li \cite{FL}, but it is not clear whether the estimates of \cite{FL} would suffice for our purposes in the present paper).

It is known that the solvability of the $L^p$ Dirichlet problem for $\mathcal{L}$ (for some $p > 1$) is equivalent 
(in the presence of the doubling property of parabolic measure)
to a certain scale-invariant mutual 
absolute continuity condition (of parabolic measure with respect to parabolic surface measure)
known as the parabolic (Muckenhoupt) $A_\infty$ 
condition\footnote{ In the absence of the doubling property, 
$L^p$ solvability is equivalent to the ``weak-$A_\infty$" condition; see, e.g., \cite{GH1}.}. 
This condition has been extensively studied in its connection with the boundedness of classical objects in harmonic analysis (maximal functions and SIOs) since the early 70's \cite{Muck}. In 1977, Dahlberg \cite{DahlbergRH2} resolved a major open conjecture concerning the absolute continuity of {\it harmonic} measure in Lipschitz domains by (implicitly) showing that the harmonic measure is an $A_\infty$ weight. It was also shown that the $A_\infty$ condition gives $L^p$ non-tangential maximal function estimates for 
solutions with $L^p$ data, for some $p<\infty$. 
This led R. Hunt to conjecture that if the lateral boundary $\Sigma$
of a parabolic domain is given by the graph of a $\Lip(1,1/2)$ function, 
then {\it caloric} measure should be mutually absolutely continuous with 
respect to the parabolic surface measure, but this conjecture was ultimately disproven by Kaufman and Wu \cite{KW}. 
In \cite{LewSil}, J. Lewis and his student J. Silver pioneered the search 
for the ``correct" regularity condition on $\Sigma$ that would allow
solvability of the parabolic
$L^p$ Dirichlet
problem in domains with time-varying rough (lateral) boundary. In particular, the work \cite{LewSil} showed that  
if $\Omega \subseteq \mathbb{R}^2$ was the region above a  
the graph of a $\Lip(1/2)$ function $\psi$, 
then a (stronger) condition, similar to (but slightly stronger than) $D_t^{1/2}\psi \in \BMO$, is sufficient for the mutual absolute continuity of caloric measure and (parabolic) surface measure.

The condition 
that $D_t^{1/2}\psi \in \BMO$ 
was seen to be natural, from the standpoint of singular 
integrals\footnote{It is also possible to recover Murray's result from the $T1$ theorem \cite{DavidJourne}.},
in light of the work of Murray \cite{MurThesis}.
Motivated by these considerations, 
Lewis and Murray \cite{LM} showed 
that the additional assumption
$D_t^{1/2}\psi \in \BMO$, is sufficient for the $L^p$ solvability of the 
Dirichlet problem for the {\it heat equation} above the graph of $\psi$ (for some $p > 1$). 
Led by the interesting papers of Lewis, Murray and 
Silver, the fourth named author began 
to work on parabolic singular integrals \cite{HofParaSio}, 
and together with Lewis \cite{HL96} showed that if 
$D_t^{1/2}\psi$ has sufficiently small BMO norm,
then one has $L^2$ solvability of the Dirichlet,
Neumann, and Regularity problems, via the method of layer potentials, 
for the heat equation in the domain above the 
graph of $\psi$. Later works 
extended these results beyond the graphical setting and/or to the variable coefficient setting \cite{HL-Mem,HL2,DPP17,NS,BHHLN-DJ,GH1,GH2}, yielding in particular the aforementioned converse of Theorems \ref{main1.thrm} and \ref{main2.thrm}.

The work \cite{BHMN1} shows that the results in \cite{LM} are sharp, that is, if $\Omega$ is the region above the graph of a $\Lip(1,1/2)$ function $\psi$, the solvability of the $L^p$ Dirichlet problem for some $p > 1$ is {\em equivalent} to $D_t^{1/2}\psi \in \BMO$.  Our work here concerns an extension of \cite{BHMN1}, using ideas from \cite{LN07} or \cite{HMT}. In particular, we use an integration by parts scheme
in order to establish an essential ingredient in \cite{BHMN1}, namely,
a local square function estimate for 
\[(|\partial_t G|^2 + |\nabla^2 G|^2) \dist(\cdot, \partial \Omega).\]   
To obtain this estimate, we use a different scheme depending on whether we are proving Theorem \ref{main1.thrm} or \ref{main2.thrm}. 
In the case of Theorem \ref{main1.thrm}, we adapt the argument in \cite{LN07} by using some elementary linear algebra.
In the case of Theorem \ref{main2.thrm}, we use an argument analogous to that of \cite{HMT} (but 
significantly more delicate in the parabolic setting).

In contrast to the situation in the constant coefficient case treated in
 \cite{BHMN1}, it is not clear whether the level sets of the Green function 
 locally stratify the region near $\partial \Omega$ graphically, 
in the sense that the sets $E_\lambda = \{Y: G(X,Y) = \lambda\}$ are 
 `nice' graphs for small $\lambda$, and 
 $\cup_{\lambda \in (0,\lambda_0)}E_\lambda$
is a sufficiently `fat' strip along the boundary. This local stratification by graphical level sets is a crucial ingredient in \cite{BHMN1}, 
and the way we circumvent this obstacle casts 
new light on \cite{BHMN1}, which may be useful in other settings. 
In particular, our strategy involves using the estimates on the time and second spatial partial derivatives of $G$ to prove that it is ``often" the case the $\partial_{x_0} G \gtrsim 1$ (here the domain is given by $\Omega := \{(x_0, x', t): x_0 > \psi(x,t)\}$), which, in turn, allows us to construct graphs in certain ``sawtooth regions". Then we can refine the estimates in \cite{BHMN1} to prove a certain ``corona" decomposition of the graph of $\psi$, by graphs of {\it regular} $\Lip(1,1/2)$ functions, that is, Lip(1,1/2) 
functions that have a half-order time derivative in  
(parabolic) $\BMO$. In \cite{BHHLN-Corona}, 
it is shown that such a corona decomposition 
implies that the graph of $\psi$ is parabolic uniformly rectifiable, which in turn, in the graph setting, is equivalent to $D_t^{1/2}\psi \in \BMO$ (since $\psi$ is $\Lip(1,1/2)$ by hypothesis).

We mention again that it would be interesting to know whether our results hold under 
the weaker $L^2$-Carleson oscillation condition. In the elliptic setting this is accomplished 
in \cite{HMMTZ}, by using both
compactness methods and the technique of extrapolation of Carleson measures; 
however, in the context of \cite{HMMTZ} the authors merely need to show an exterior corkscrew condition. Indeed, it is known that, in the elliptic setting, 
if a set satisfies the two-sided corkscrew condition then it must be uniformly rectifiable, but this is not the case in the parabolic setting\footnote{Otherwise every $\Lip(1,1/2)$ function would have a half-order time derivative in (parabolic) BMO, which is not true, see \cite{BHHLN-BP}.}. Moreover, it is known that weak flatness conditions (e.g. WHSA \cite{HLMN}, BWGL \cite{DS2}) that characterize (elliptic) uniformly rectifiable sets fail to characterize parabolic uniform rectifiability. 
It would be interesting if a more direct, ``constructive" (no compactness) argument, 
via a square function estimate, could be used to recover the result in \cite{HMMTZ}, 
even in the elliptic setting, as this might provide a path to prove the results here under the $L^2$-Carleson oscillation condition. Perhaps the ideas in \cite{FL} could serve as a first step on that path.

\begin{remark}\label{remark-smoothL1carl}
To conclude we make a remark for the careful reader, who may notice that in
the $L^1$-Carleson oscillation condition in Definition \ref{smoothL1carl.def}, we assume that the matrix function $A$ is locally smooth. In fact, at least in the setting of Theorem~\ref{main1.thrm}, this is merely qualitative and inessential,
as one can assume a weaker oscillation condition and smooth the coefficients, maintaining the $A_\infty$ property for the parabolic measure of the smoothed operator. This is known for the case that $A$ is symmetric \cite{N97} and it could be that a straightforward modification of \cite{N97} would give the non-symmetric case. These topics are discussed in Appendix \ref{extendrmks.sect}.
\end{remark}

\section{Preliminaries}

\subsection{\for{toc}{\small}Notation}

\begin{itemize}
	\item Throughout the paper, $n \ge 2$ is a natural number and
	we let $d:=n+1$ denote the  natural parabolic homogeneous dimension of space-time $\ren$:
	\[
	\rn = \big\{ \mbf{x}=(x,t)  \in \re^{n-1}\times\re\big\}.
	\]
	
	\item The ambient space we work in is $\ree:=\re\times \re^{n-1}\times \re$,
	\[
	\ree = \big\{\xbf= (X,t)=(x_0,x,t) \in \re\times \re^{n-1}\times\re\big\}.
	\]
	Here we have distinguished the last coordinate as the time coordinate and the first spatial coordinate as the graph coordinate. 
	
	\item To help the reader identify the nature of points used in the paper, we use the notation employed above, which we here describe in detail. We use lower case letters
	(e.g. $x$, $y$, $z$) to denote spatial points in $\re^{n-1}$, and capital letters (e.g $X = (x_0, x)$, $Y=(y_0, y)$, $Z = (z_0, z)$), to denote spatial points in $\re^{n}=\re\times\re^{n-1}$.
	We also use boldface capital letters (e.g. $\mbf{X}=(X,t)$, $\mbf{Y}=(Y,s)$, $\mbf{Z}=(Z,\tau)$) to denote points in space-time
	$\ree$, and boldface lowercase letters (e.g. $\mbf{x}=(x,t)$, $\mbf{y}=(y,s)$, $\mbf{z}=(z,\tau)$) to denote points in space-time $\RR^n$. In accordance with this notation, given $\mbf{X}=(X,t)\in \ree$ (resp. $ \mbf{x}=(x,t)\in\ren$) we use the notation $t(\mbf{X})$ (resp. $t(\mbf{x})$) to denote its time component, that is,  $t(\mbf{X}) = t$ (resp. $t(\mbf{x}) = t$). 
	
	\item We also denote 1-dimensional integrals by $\int$, integrals in space-time $\RR^n$ by $\iint$, and integrals in space-time $\RR^{n+1}$ by $\iiint$. Averages are denoted, in each case, by $\fint$, $\bariint$ and $\bariiint$. 
	
	\item We denote the parabolic length by
	\begin{align*}
		\|\mbf{X}\|&=\|(X,t)\|:=|X|+|t|^{\frac12}, \quad \mbf{X}=(X,t)\in\ree\,,
		\\ \|\mbf{x}\|&=\|(x,t)\|:=|x|+|t|^{\frac12}, \quad \mbf{x}=(x,t)\in\ren\,,
	\end{align*}
	and accordingly, all distances will be measured with respect to the parabolic metric
	\[
	\dist(\mbf{X},\mbf{Y}):=\|\mbf{X}-\mbf{Y}\|:=|X-Y|+|t-s|^{\frac12},
	\qquad \mbf{X}=(X,t),\ \mbf{Y}=(Y,s)\in\ree,
	\]
	\[
	\dist(\mbf{x},\mbf{y}):=\|\mbf{x}-\mbf{y}\|:=|x-y|+|t-s|^{\frac12},
	\qquad \mbf{x}=(x,t),\ \mbf{y}=(y,s)\in\rn.
	\]
	
	\item It is sometimes convenient to use a different (smooth) parabolic distance. Given $\mbf{x}=(x,t)\in\ren\setminus\{0\}$,  
	we let $\vertiii{\mbf{x}}$ 
	be defined as the unique positive solution of the equation
	\begin{equation*}
		\frac{|x|}{\vertiii{\mbf{x}}^2}+\frac{t^2}{\vertiii{\mbf{x}}^4}=1, 
	\end{equation*}
	In particular, $\vertiii{\mbf{x}}\approx \|\mbf{x}\|$ 
	with implicit constants depending 
	only on $n$.

	\item Given $\mbf{x}=(x,t)\in\ren$ and $R>0$, we define the parabolic
	\textbf{cube in $\ren$} as
	\[
	Q_R(\mbf{x})
	:=
	\big\{
	\mbf{y}=(y,s)\in\ren: |y_i-x_i|<R, \, 1\leq i\leq n-1,\,\ |t-s|< R^2
	\big\}.
	\]
	
	\item Given $\mbf{x}=(x,t)\in\ren$ and $R>0$, 
	we denote a closed parabolic \textbf{cube in $\ree$} by
	\begin{equation}\label{cuba}\mathcal{J}_{R}(\mbf{X}) = \mathcal{J}_{R}(x_0, x,t):=[x_0-R, x_0+R]\times \overline{Q_R(\mbf{x})}\,.
	\end{equation}
	For $\mathcal{J} = \mathcal{J}_{R}(\mbf{X})$, we let $\ell(\mathcal{J}) := \ell(\mathcal{J}_{R}(\mbf{X}))= 2R$
	denote the parabolic side length of $\mathcal{J}$.
	
	\item We let $\dd$ denote the standard grid of parabolic dyadic cubes on $\mathbb{R}^n$, and $\ell(Q)$ denote the side length of $Q \in \dd$. Also, for a given $Q_0 \in \dd$, we denote $\dd(Q_0) := \{ Q \in \dd : Q \subseteq Q_0 \}$.
\end{itemize}

\subsection{\for{toc}{\small} Parabolic Hausdorff measure}
Given $\eta \geq 0$, we let $\cH^\eta$ denote
 standard $\eta$-dimensional Hausdorff measure.
  We also define a {parabolic} Hausdorff measure of {homogeneous}
  dimension $\eta$, denoted
  $\cH_{\text p}^\eta$, in the same way that one defines standard Hausdorff measure, but instead using coverings
  by {parabolic} cubes. I.e., for $\delta>0$, and for $E\subseteq \mathbb R^{n+1}$, we set
  \[ \cH_{\text{p},\delta}^\eta(E):= \inf \sum_k \diam(E_k)^\eta\,,
  \]
  where the infimum runs over all countable such coverings of $E$, $\{E_k\}_k$, with $\diam(E_k)\leq \delta$ for all $k$. Of course, the diameter is measured in the parabolic metric.  We then define
  \[
  \cH_{\text p}^\eta (E) := \lim_{\delta\to 0^+} \cH_{\text{p},\delta}^\eta(E)\,.
  \]
As for classical Hausdorff measure, $ \cH_{\text{p}}^\eta$ is a Borel regular measure.
We refer the reader to \cite[Chapter 2]{EG} for a discussion of the basic properties of standard
Hausdorff measure, which adapt readily to treat $  \cH_{\text{p}}^\eta$.
In particular, one obtains a measure equivalent to   $\cH_{\text{p}}^\eta$ if one defines
$\cH_{\text{p},\delta}^\eta$ in terms of coverings by arbitrary sets of parabolic diameter at most $\delta$, rather than
cubes.

\subsection{\for{toc}{\small}Lip(1,1/2) graph domains}\label{Lipnot}  A function $\psi:\mathbb R^{n-1}\times\mathbb R\to \mathbb R$ is Lip(1,1/2) with constant $C$, if
\begin{align}\label{1.1}
|\psi(x,t) - \psi(y,s)| \le C(|x - y| + |t - s|^{1/2}) = C\|(x,t)-(y,s)\|, \quad \forall (x,t), (y,s) \in \rn.
\end{align}
We define $\|\psi\|_{\Lip(1,1/2)}$ to be the infimum of all constants $C$ as in \eqref{1.1}.  If we set
\begin{eqnarray}\label{1.1++}
\Sigma:=
\big\{(\psi(x,t), x, t): (x,t)\in\ren\big\}
=
\big\{(\psi(\mbf{x}), \mbf{x}): \mbf{x}\in\ren\big\} =:
\big\{\mbf{\Psi}(\mbf{x}): \mbf{x}\in\ren\big\},
\end{eqnarray}
 then we say that $\Sigma$ is a Lip(1,1/2) graph. The set $\Sigma$ is the boundary of the domain
\begin{eqnarray}\label{1.1+++}
\Omega:=\big\{\mbf{X}=(x_0,x,t)\in\ree: x_0>\psi(x,t)\big\},
\end{eqnarray}
We refer to $\Omega\subseteq\mathbb R^{n+1}$  as an
(unbounded) \textbf{$\Lip(1,1/2)$ graph
domain} with constant $\|\psi\|_{\Lip(1,1/2)}$. Given the closed set $\Sigma \subseteq \mathbb R^{n+1}$
  of  homogeneous dimension $\text{dim}_{\cH_{\text{p}}}(\Sigma)=n+1$, we
define  a surface measure on $\Sigma$ as the restriction of $\cH_{\text{p}}^{n+1}$ to $\Sigma$, i.e.,
\begin{equation}\label{sigdef}
\sigma := \sigma_\Sigma:=  \cH_{\text{p}}^{n+1}|_\Sigma\,.
\end{equation}

We remark that for Lip(1,1/2) graphs, $\sigma$ as defined in \eqref{sigdef} is equivalent to $\d\sigma^{\bf s}:= \d\sigma_t \,\d t$, where $\d\sigma_t$ is 
standard $(n-1)$-dimensional Hausdorff measure $\cH^{n-1}$, 
restricted to the cross section 
$\Sigma_t:= \{x: (x,t) \in \Sigma\}$.  We refer the reader to 
\cite[Remark 2.8 and Appendix B]{BHHLN-Corona} for details. 

\subsection{\for{toc}{\small}Surface cubes and reference points}\label{scuberef} We define several geometric objects related for $Lip(1,1/2)$ graph domains, see Figure \ref{boxescs.fig}. We let
\begin{eqnarray}\label{1.1+}
M_0 := 2 + \|\psi\|_{\Lip(1,1/2)}.
\end{eqnarray}
For every $\mbf{X}=(x_0,x,t) \in\ree$ and $R>0$, we introduce \textbf{vertically elongated open ``cubes"} 
\begin{equation*}
 I_R(\mbf{X}):=(x_0-3M_0\sqrt{n}R, x_0+3M_0\sqrt{n}R)\times Q_R(\mbf{x})
\end{equation*}
and set
$$\Delta_R(\mbf{X}):=I_R(\mbf{X})\cap\Sigma.$$ We will refer to $\Delta_R(\mbf{X})$ as a surface box or cube of size $R>0$ and centered at $\mbf{X}$. Unless otherwise specified, we implicitly assume that the center
$\mbf{X}=(x_0,\mbf{x})=(x_0,x,t)$, of any surface box 
$\Delta_R(\mbf{X})$, is in $\Sigma$, that is, $x_0=\psi(x,t)$.

Note the crude estimate
\begin{equation*}
\mbf{Y} \in  I_R(\mbf{X}) \implies \|\mbf{Y} - \mbf{X}\| \le 5M_0\sqrt{n} R,
\end{equation*}
and that by construction,
\begin{equation} \label{surface-box:graph}
	\Delta_R(\mbf{X})
= \bfpsi\big(Q_R(\mbf{x})\big) =
\big\{ \bfpsi(\mbf{y}): \mbf{y}\in Q_R(\mbf{x})\big\},
\qquad
\forall\,\mbf{X}=(x_0,\mbf{x})\in\Sigma,
\end{equation}
where we recall that $\mbf{y}\mapsto \bfpsi(\mbf{y}):=(\psi(\mbf{y}),\mbf{y})$
is the graph parametrization of $\Sigma$ (see \eqref{1.1++}).
Indeed, if $\mbf{y}\in Q_R(\mbf{x})$, 
then $\|\mbf{x} - \mbf{y}\| \le (\sqrt{n}R + R)$ and hence
\[|\psi(\mbf{x}) - \psi(\mbf{y})| \le M_0 (\sqrt{n}R + R) \le 2\sqrt{n}M_0 R.\]
We also note, by the same reasoning, that if $R > 0$, $\mbf{X}=(x_0,\mbf{x})\in\Sigma$ and $\mbf{y}\in Q_R(\mbf{x})$, then
\[\psi(\mbf{y}) + s \in I_R(\mbf{X}) \quad \forall s \in (-M_0\sqrt{n}R, M_0\sqrt{n}R).\]
Given $\tau>0$, we define the parabolic dilation
$\tau\Delta_R(\mbf{X}):=\Delta_{\tau R}(\mbf{X})$.

We introduce time forward and time backwards \textbf{corkscrew points} relative to $\Delta_R(\mbf{X})$,
\begin{equation} \label{CSpm}
\cA^\pm_R(\mbf{X})  := \left(x_0+ 2 M_0R, x, t\pm 2 R^2\right),   \quad  \mbf{X}=(x_0,x,t)= (x_0,\mbf{x})=(\psi(\mbf{x}),\mbf{x}) \in \Sigma.
\end{equation}
Note that for $\mbf{X}\in\Sigma$,
\begin{equation*}
\cA^\pm_R(\mbf{X}) \in \Omega_{2R}(\mbf{X}) := I_{2R}(\mbf{X}) \cap \Omega,
\end{equation*}
and that $$\dist\left(\cA^\pm_R(\mbf{X}), \partial \Omega_{2R}\right)\approx \dist\left(\cA^\pm_R(\mbf{X}), \partial \Omega_{3R}\right) \approx R,$$
where the implicit constants depend only on $n$ and $\|\psi\|_{\Lip(1,1/2)}$. Furthermore,
\[|\psi(x_0, x, t\pm 2 R^2) - \psi(x_0, x, t)| \le \sqrt{2}\|\psi\|_{\Lip(1,1/2)}R,\]
and hence
\begin{equation*}
\dist(\cA^\pm_R(\mbf{X}), \partial \Omega) \ge 2M_0R - \sqrt{2}\|\psi\|_{\Lip(1,1/2)}R \ge 2R.
\end{equation*}

\begin{figure}
	\centering
	\includegraphics[width=.65\textwidth]{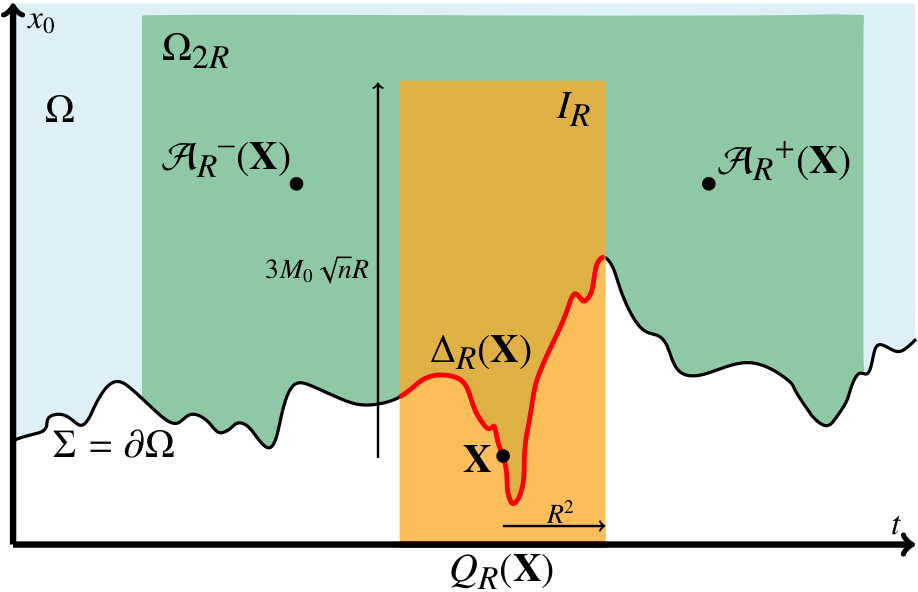}
	\caption{
		A (not-to-scale) depiction of elongated cubes and corkscrews associated to a surface ball $\Delta_R(\mbf{X}) = \mbf{\Psi}(Q_R(\mbf{X}))$. We always think that time flows from left to right, and $x_0$ is the vertical direction. One should imagine the other $n-1$ spatial directions as being perpendicular to both $x_0$ and $t$, but we will not reflect those in our 2D pictures for simplicity. 
	}
	\label{boxescs.fig}
\end{figure}

We frequently use (sometimes without mention) the following fact: the distance to the Lipschitz graph is, up to constants, attained in the vertical direction (in which we parametrize the graph).
\begin{lemma}\label{distgraphlem.lem} Assume that $\mbf{X}=(x_0,\mbf{x})\in \Omega$. Then
\[M_0^{-1} (x_0 - \psi(\mbf{x})) \le \dist(\mbf{X}, \Sigma) \le x_0 - \psi(\mbf{x}).\]
\end{lemma}
\begin{proof}
Let $L := \|\psi\|_{\Lip(1,1/2)}$. Find $\mbf{Y} = (\psi(\mbf{y}), \mbf{y}) \in \Sigma$ be such that $\dist(\mbf{X}, \Sigma) = \|\mbf{Y} - \mbf{X}\|$. Then
\begin{equation*}
	\| (\psi(\mbf{x}), \mbf{x}) - \mbf{Y} \|
	= 
	| \psi(\mbf{x}) - \psi(\mbf{y}) | + \| \mbf{x} - \mbf{y} \|
	\leq 
	(L+1) \| \mbf{x} - \mbf{y} \|
	\leq 
	(L+1) \| \mbf{X} - \mbf{Y} \|.
\end{equation*}
Therefore, using the (reverse) triangle inequality we obtain
\begin{equation*}
	\| \mbf{X} - \mbf{Y} \|
	\geq 
	\| \mbf{X} - (\psi(\mbf{x}), \mbf{x}) \| - \| (\psi(\mbf{x}), \mbf{x}) - \mbf{Y} \|
	\geq 
	(x_0 - \psi(\mbf{x})) - (L+1) \| \mbf{X} - \mbf{Y} \|,
\end{equation*}
which yields the desired result after rearranging the terms.

\end{proof}

\subsection{\for{toc}{\small}(Parabolic) BMO and fractional integral operators} \label{sec:BMO} Given a locally integrable function $f:\ren\to \re$, we say that $f\in\BMO(\ren)$, the \textbf{parabolic BMO-space}, if
\[
\|f\|_{\BMO(\ren)}
:=
\sup_{Q\subseteq\re^n}
\bariint_{Q} |f(x,t)-f_Q|\,\d x\d t<\infty,
\]
where the supremum runs over all parabolic cubes $Q=Q_R(\mbf{y})$, with $\mbf{y}\in\ren$ and $R>0$. Here $f_Q$ denotes the average of $f$ on $Q$.

We introduce the \textbf{fractional integral operator} of parabolic order 1 on $\ren$, via Fourier transform,
\begin{equation*}
	(\mathrm{I_P} \psi)^{\,\widehat{}}\, (\xi,\tau):= \vertiii{(\xi,\tau)}^{-1}\,\widehat{\psi}(\xi,\tau),
	\qquad
	(\xi,\tau)\in\re^n.
\end{equation*}
Then, we can represent it by a kernel $V$ as
\begin{equation}\label{def-IP}
\mathrm{I_P}\psi(\mbf{x})=\iint_{\re^n} V(\mbf{x}-\mbf{y})\,\psi(\mbf{y})\,\d\mbf{y},
\qquad
\mbf{x}\in\re^n,
\end{equation}
where, recalling that $n+1$ is the homogeneous dimension of parabolic $\rn$, it additionally holds
\begin{equation}\label{vstdbd.eq}
0\le V(\mbf{y})\lesssim \|\mbf{y}\|^{-n}.
\end{equation}

Using $\mathrm{I_P}$, following
\cite{FR}, we introduce a (parabolic) \textbf{half-order time derivative} as
\[
\mathcal{D}_t \psi(\mbf{x})
:=
\partial_t \circ\mathrm{I_P} \psi(\mbf{x})
=
\iint_{\re^n} \partial_t\big(V(\mbf{x}-\mbf{y})\big)\,\psi(\mbf{y})\,\d\mbf{y},
\qquad
\mbf{x}=(x,t)\in\re^n.
\]
This operator should be viewed as a principal value operator, or one should consider $\partial_t$ in the weak sense. Note that
the Fourier symbol for $\mathcal{D}_t $ is $2\pi i\tau/\vertiii{(\xi,\tau)}$.

Another half-order time derivative, $D_{1/2}^t$, can be introduced via the multiplier $|\tau|^{1/2}$, or by
\begin{eqnarray*} 
 D_{1/2}^t  \psi (\mbf{x})=D_{1/2}^t  \psi (x,t):=c \int_{ \mathbb R }
\, \frac{ \psi ( x, s ) - \psi ( x, t ) }{ | s - t |^{3/2} } \, \d s,
\end{eqnarray*} for properly chosen $c$.

\subsection{\for{toc}{\small}Regular Lip(1,1/2) graph domains}\label{Lipnotreg} Let $\psi:\mathbb R^{n-1}\times\mathbb R\to \mathbb R$ be a $\Lip(1,1/2)$ function (see \eqref{1.1}).  Such a function is said to be a {\bf regular Lip(1,1/2)} function if, in addition,
 \begin{eqnarray*} 
 \mathcal{D}_t \psi\in \BMO(\ren).
\end{eqnarray*}
If $\psi$ is a {regular} $\Lip(1,1/2)$ function, then we say that
$\Sigma$ as in \eqref{1.1++} is a {\em regular} $\Lip(1,1/2)$ graph, and  that $\Omega\subseteq\mathbb R^{n+1}$ as in \eqref{1.1+++} is an
(unbounded) regular $\Lip(1,1/2)$ graph
domain. In both cases the regularity is determined by $\|\psi\|_{\Lip(1,1/2)}$ and $\|\mathcal{D}_t \psi\|_{\BMO(\ren)}$. In \cite{HL96} it is proved that
\begin{equation}\label{eqnormequiv}
\|\psi\|_{\text{R-Lip}}:=
\|\mathcal{D}_t \psi\|_{\BMO(\ren)}+\|\nabla_x\psi\|_{\infty}\approx \|D_{1/2}^t \psi\|_{\BMO(\ren)}+\|\nabla_x\psi\|_{\infty}\,,
\end{equation}
where $\mathcal{D}_t, D_{1/2}^t$ were defined in subsection~\ref{sec:BMO}. In particular, that a function is a {regular} $\Lip(1,1/2)$ function can be equivalently
formulated using $D_{1/2}^t$ instead of $\mathcal{D}_t$, but the latter 
will be considerably more convenient for us to work with in this paper.

\begin{remark}  One can prove that, in general, the class of regular Lip(1,1/2) functions is strictly contained in Lip(1,1/2), i.e., there are examples of
functions $\psi$ which are  Lip(1,1/2) but not regular Lip(1,1/2), see \cite{LewSil},
\cite{KW}.  Moreover, it follows from the arguments of Strichartz
\cite{Stz} that for a regular Lip(1,1/2) function $\psi$, the assumption that 
$\psi$ is Lip(1/2) in the 
time variable is redundant: it 
follows from the finiteness of the R-Lip norm in \eqref{eqnormequiv}.
\end{remark}

\begin{remark}\label{PUR} One can prove, in the context of Lip(1,1/2) graphs, that the graph being regular Lip(1,1/2) is equivalent to the graph being parabolic uniformly rectifiable.
The notion of parabolic uniform rectifiability was introduced in
\cite{HLN1,HLN2}\footnote{based of course on the work of David and Semmes in the elliptic setting
\cite{DS1,DS2}}, in connection with the study of parabolic versions of the free boundary results of Kenig and Toro \cite{KT1,KT2,KT3}\footnote{See also \cite{Eng}, for an improved version of the result in \cite{HLN2}.}.  The results in \cite{HLN2,Eng}
are in some sense ``small constant" versions of the results presented here; we refer the interested reader to the introduction of \cite{BHMN1} for a discussion of the latter point,
and to \cite{HLN1,HLN2} (or for that matter to
\cite{BHHLN-CME,BHHLN-Corona}) for the precise definition of parabolic uniform rectifiability.
For our purposes, we simply remark that for the graph of a Lip(1,1/2) function $\psi$, the definition of parabolic uniform rectifiability reduces to \eqref{Carlestforh.eq} below, 
with $\beta$ defined as in \eqref{betadef}, and that 
for such $\psi$ and $\beta$, \eqref{Carlestforh.eq} is equivalent to the {\em regular} Lip(1,1/2) condition
(see \cite{HLN2} for details). To be precise, set
\begin{equation}\label{betadef}
\beta(r,x,t) := \inf_{L} \left[\bariint_{Q_r(x,t)} \left(\frac{\psi(y,s) - L(y)}{r} \right)^2 \,\d\sigma(y,s)\right]^{1/2}, \quad (r,x,t) \in \ree_+ ,
\end{equation}
where the infimum is taken over all affine functions $L$ of $y$ only, and let
\[\d\nu:= \d\nu_\psi := \beta^2(r,x,t) \frac{\, \d r\, \d x\, \d t}{r}.\]
If $\psi(x,t)$ is Lipschitz in the space variable
$x$, uniformly in $t$, then
 the condition that $\mathcal{D}_t \psi\in \BMO(\ren)$ is 
equivalent to saying that $\d\nu$ is a Carleson measure on $\ree_+$, i.e., there exists a finite constant $\|\nu\|$ such that for all $(z,\tau)\in\mathbb R^n$ and for all $R>0$, we have the uniform bound
\begin{equation}\label{Carlestforh.eq}
\int_0^R\bariint_{Q_R(z,\tau)}\beta^2(r,x,t) \frac{\, \d r\, \d x\, \d t}{r}\leq \|\nu\| <\infty.
\end{equation}

\end{remark}

\subsection{\for{toc}{\small}Dyadic cubes and a criterion for a Lip(1,1/2) graph to be regular}
Recall that $\dd$ denotes the standard grid of parabolic dyadic cubes on $\mathbb{R}^n$, and $\ell(Q)$ denotes the side length of any $Q \in \dd$.

\begin{definition}[(Semi-)coherent stopping times]
A collection of dyadic cubes $\sbf \subseteq \mathbb{D}$ is called {\bf semi-coherent stopping time} if there exists a maximal cube $Q(\sbf)$ (with respect to containment) in $\sbf$ and for any $Q', Q, Q^* \in \dd$ we have  
\[ \sbf \ni Q' \subseteq Q \subseteq Q^* \in \sbf 
\quad \implies \quad 
Q \in \sbf. \]
A stopping time regime is called {\bf coherent} if it is semi-coherent and whenever $Q \in \sbf$, then either all or none of its children are in $\sbf$.
\end{definition}

We are going to prove Theorem \ref{main1.thrm} (and Theorem~\ref{main2.thrm}) using the following ``Corona decomposition" characterization of regular $\Lip(1,1/2)$ functions.  In fact, it is a characterization, more generally, of parabolic uniformly rectifiable (PUR) sets, but as noted above, for a Lip(1,1/2) graph, PUR is equivalent to regularity.
The use of the Corona decomposition is a novelty of the present work,
as compared to \cite{BHMN1}, and as noted in the introduction, we use this 
approach in order to overcome the fact that in the case of non-constant coefficients, we are able to obtain only a weaker version of stratification by the level sets of the Green function, namely, we obtain an appropriate stratification only in a Corona sense, rather than in a ``big pieces" sense as in \cite{BHMN1}.
\begin{lemma}[{\cite[Theorem 1.2]{BHHLN-Corona}}]\label{coronaenough.lem}
Suppose that $\psi$ is a $\Lip(1,1/2)$ function, with associated graph $\Sigma$. Then $\psi$ is a regular $\Lip(1,1/2)$ function if and only if the following holds.  There exist constants $C_1, b_1, b_2 > 1$ such that we can form the disjoint decomposition $\mathbb{D}=\mathcal{G} \sqcup \mathcal{B}$ so that:
\begin{itemize}
\item $\mathcal{G}$ is further divided into semi-coherent stopping time regimes $\{\sbf\}$,
\item the bad cubes $\mathcal{B}$ and the maximal cubes $\{Q(\sbf)\}$ pack, that is,
\begin{equation*}
			\sum_{Q' \in \mathcal{B} \cap \mathbb{D}(Q)} |Q'|
			+ \sum_{\mbf{S} : Q(\mbf{S}) \subseteq Q} |Q(\mbf{S})|
			\leq 
			C(C_1) |Q|, 
			\quad 
			\forall Q \in \mathbb{D},
		\end{equation*}
\item for each $\sbf$, there exists a  \textbf{regular} $\Lip(1, 1/2)$ graph $\Sigma_{\mbf{S}}$ such that 
		\begin{equation*} 
			\sup_{\xbf \in \mbf{\Psi}(2Q)} \dist(\xbf, \Sigma_{\mbf{S}})
			\leq 
			C_1 \diam(Q), \quad \forall Q \in \sbf.
		\end{equation*} 
\end{itemize}
\end{lemma}

For future reference, we state the following lemma, which we intend to employ later:
 it will allow us to work below the scale of a given cube in $\mathbb{D}$, since the passage to the whole $\mathbb{D}$ is of a very abstract nature.

\begin{lemma}[{\cite[Section 7]{DS1}}]\label{dsinitialcor.lem}
There exists a collection of disjoint cubes $\mathcal{C} \subseteq \dd$ such that
$$\Sigma = \bigcup\limits_{ Q' \in \mathcal{C}} \mbf{\Psi}(Q'),$$ and such that the collection
\[\widetilde{\cB}: = \big\{Q \in \mathbb{D}: Q \not\subseteq Q', \forall Q' \in \mathcal{C}\big\}\]
satisfies the packing condition
\[\sum_{{Q \in \widetilde{\cB}: Q \subseteq R}} \sigma(\mbf{\Psi}(Q)) \le C\sigma(\mbf{\Psi}(R)), \quad \forall R \in \mathbb{D}.\] 
Here the constant $C$ depends only on $n$ and $\|\psi\|_{\Lip(1,1/2)}$.
\end{lemma}
For a proof of the lemma, we 
refer the reader to the discussion in \cite[Section 7, p. 38]{DS1}; note that the same argument works in the parabolic setting treated here.

\subsection{\for{toc}{\small}Whitney regions and sawtooths}\label{whit.sec}

For any dyadic cube $Q \in \dd$ and $K > 4$, we define the \textbf{Whitney regions} $U_Q(K)$ as 
\begin{equation}\label{whitdef.eq}
U_Q(K) :=  \big{\{}(x_0, x,t) \in \ree: \; (x,t) \in 10^{5}Q, \;\; x_0 - \psi(x,t) \in (K^{-1}\ell(Q), K\ell(Q))\big{\}}.
\end{equation}
By construction, for any $K > 4$, $\Omega = \cup_{Q \in \dd} U_Q(K)$. 

We also define a few fattened versions of these:
\begin{equation*}
U_Q^*(K) := \big{\{}(x_0, x,t) \in \ree: \; (x,t) \in 10^{6}Q, \;\; x_0 - \psi(x,t) \in ((2K)^{-1}\ell(Q), 2K\ell(Q))\big{\}},
\end{equation*}
\begin{equation*}
U_Q^{**}(K) := \big{\{}(x_0, x,t) \in \ree: \; (x,t) \in 10^{7}Q, \;\; x_0 - \psi(x,t) \in ((4K)^{-1}\ell(Q), 4K\ell(Q))\big{\}},
\end{equation*}
and 
\begin{equation*}
U_Q^{***}(K) := \bigcup_{Q' \in \mathfrak{A}_Q} U_{Q'}^{**},
\quad \text{where} \quad 
\mathfrak{A}_Q := \big{\{}Q' \in \mathbb{D}: \; U_{Q'}^{**} \cap U_{Q}^{**} \neq \emptyset \big{\}}.
\end{equation*}
Note that it  holds that (see Lemma~\ref{distgraphlem.lem})
\begin{equation*}
\dist(\xbf, \partial\Omega) \approx_K \ell(Q) \approx_K \dist(\xbf, \mbf{\Psi}(Q)), \qquad \; \forall \xbf \in U_Q^{***}(K)
\end{equation*}
(and by extension any $\xbf$ in $U_Q(K)$, $U_Q^*(K)$ or $U_Q^{**}(K)$).

Moreover, it is easy to show that, for any $Q \in \dd$, it holds (the constants are in no way optimal) 
\begin{equation} \label{eq:U_Q_in_balls}
	U_Q^{**} \subseteq Q_{10^{10}KM_0\diam(Q)}(\psi(\mbf{x}_Q), \mbf{x}_Q), \quad \text{ and } \quad 
U_Q^{***} \subseteq Q_{10^{13}K^3M_0\diam(Q)}(\psi(\mbf{x}_Q), \mbf{x}_Q).
\end{equation}

\begin{definition}[Sawtooth regions]\label{sawtoothdef.def}
Given a collection\footnote{Here we do not insist on $\sbf$ being a stopping time regime. We will use these objects with $\sbf$ a doubly truncated stopping time regime, which will destroy the property that $\sbf$ has a unique maximal cube.} of cubes $\sbf \subseteq \dd$ we define the following (dyadic) \textbf{``sawtooth regions"}:
\[\Omega_{\sbf} := \Omega_\sbf(K) = \bigcup_{Q \in \sbf} U_Q(K), \qquad \Omega^{*}_{\sbf} := \Omega^{*}_\sbf(K) = \bigcup_{Q \in \sbf} U_Q^{*}(K), \]
\[\quad \Omega_{\sbf}^{**} := \Omega_\sbf^{**}(K) = \bigcup_{Q \in \sbf} U_Q^{**}(K), \qquad \Omega_{\sbf}^{***} := \Omega_\sbf^{***}(K) = \bigcup_{Q \in \sbf} U_Q^{***}(K).\]
\end{definition}

See a (not-to-scale) depiction of the constructions in Figure~\ref{fig:U_Q}.

\begin{figure}
	\centering 
	\includegraphics[width=.6\textwidth]{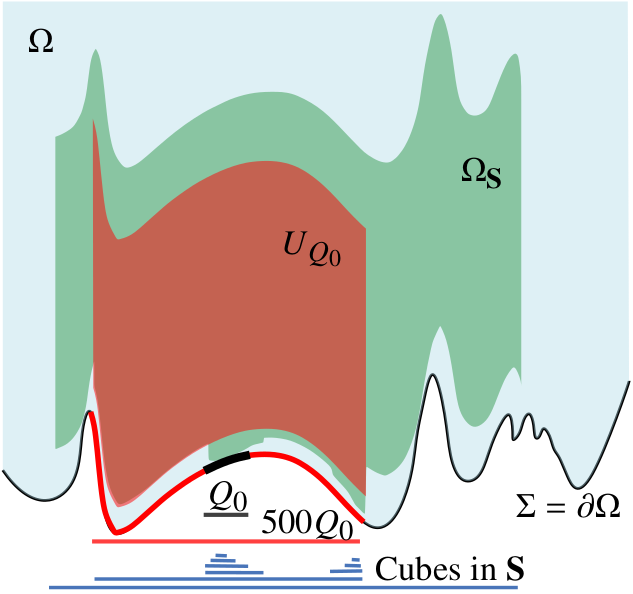}
	\caption{A sketch of how a sawtooth $\Omega_{\sbf}$ is comprised of several (fattened) Whitney regions $U_{Q_0}$. If the cubes in $\sbf$ get smaller, the sawtooth region $\Omega_\sbf$ gets closer to the boundary.}
	\label{fig:U_Q}
\end{figure}

\subsection{\for{toc}{\small}Convention concerning constants}\label{coo} We refer to $n$, $\Lambda$, $\| \psi \|_{\Lip(1, 1/2)}$ (see \eqref{1.1}), and the constants $C_\ast$ and $q  > 1$ appearing in Definition \ref{defainfty} below, as the {\it structural constants}. For all constants $A,B > 0$, the notation $A\lesssim B$  means, unless otherwise stated, that $A/B$ is bounded from above by a positive constant depending at most on the structural constants; $A\gtrsim B$ of course means $B\lesssim A$. We write $A\approx B$  if $A\lesssim B$ and  $B\lesssim A$, while for a given constant $\eta$,
$A\lesssim_\eta B$  means, unless otherwise stated, that the implicit constant depends at most on the structural constants and 
$\eta$.

\section{Boundary estimates in $\Lip(1,1/2)$ domains}\label{bbest} 
In the sequel, we will assume that $A$ is an elliptic $n \times n$ matrix (see \eqref{ellip.eq}) and $\Omega$ is the region above a $\Lip(1,1/2)$ graph as in \eqref{1.1+++}, with boundary $\Sigma$.

The Dirichlet problem, parabolic measure, and the boundary behavior of non-negative solutions,  for the heat equation but also for more general linear uniformly parabolic equations with space and time dependent coefficients,  have been studied intensively in Lipschitz cylinders and in Lip(1,1/2) domains over the years, see, e.g., \cite{FGS,FGS_BH,FS,FSY,LM,N97}.
Results include 
Carleson type estimates, the relation between the associated parabolic measure and the Green function, the backward in time Harnack inequality, the doubling property of parabolic measure, boundary Harnack principles (local and global) and  H\"older continuity up to the boundary of quotients of non-negative solutions vanishing on the lateral boundary.  
All estimates stated in this section are known, and also apply for solutions to the adjoint equation subject to the appropriate changes (typically just exchanging $\cA_R^+$ with $\cA_R^-$) induced by the change of variables $t \mapsto -t$.

\subsection{\for{toc}{\small}\label{ss3.1}Green's functions and parabolic measure}

We refer the reader to \cite{N97} for the results stated in this subsection and the next (Subsections \ref{ss3.1} and \ref{sec:PDEestimates}), 
in the setting of a 
Lip(1,1/2) domain, or also \cite{DK, GH1} in very general frameworks.
In particular, for the parabolic equations that we consider in this paper, 
the Dirichlet problem with bounded and continuous data has a unique solution in 
$\Omega$.
Given $\mbf{Y}\in \Omega$ we let $G(\cdot)=G ( \cdot, \mbf{Y})$
 denote \textbf{Green's function} for $\mathcal{L}$ in $ \Omega$ with pole at $ \mbf{Y}$, i.e.
\begin{eqnarray*}
(\partial_t-\div_X A \nabla_X)\, G (\mbf{X}, \mbf{Y})   = \delta_{\mbf{Y}} (  \mbf{X} ) \;\; \mbox{ for }
\mbf{X} \in \Omega, \qquad  \mbox{ and } \quad G(\cdot, \mbf{Y}) \equiv 0 \;\; \mbox{ on }
\Sigma.
\end{eqnarray*}
Moreover, we note that $ G
( \mbf{Y}, \cdot ) $ is the Green's function for the adjoint equation with pole at $ \mbf{Y}$, i.e., 
 \begin{eqnarray*}
 	(-\partial_t - \div_X A^* \nabla_X)\, G (\mbf{Y},\mbf{X})  = \delta_{\mbf{Y}} (  \mbf{X} ) \;\; \mbox{ for }
\mbf{X} \in \Omega, \qquad  \mbox{ and } \quad  G(\mbf{Y}, \cdot) \equiv 0  \;\; \mbox{ on }
\Sigma.
\end{eqnarray*}
We  let $ \omega^{\mbf{Y}}(\cdot)$ and $\widetilde{\omega}^{\mbf{Y}}( \cdot) $ be the \textbf{parabolic and adjoint parabolic measures},
with pole at $\mbf{Y}\in\Omega$, associated to the equation for $\mathcal{L}$ and the adjoint one for $\mathcal{L}^*$, in $ \Omega$.
Then, given $\mbf{Y} \in \Omega$, if we extend our Green's function $G$ by 0 outside $\Omega$, the following identities hold for every $ \phi \in
C_0^\infty ( \mathbb R^{ n + 1 } \setminus \{  \mbf{Y} \} ) $:
 \begin{align*}
 \iiint \left( A^\ast \nabla_X G( \mbf{Y}, \mbf{X})\cdot \nabla\phi  - G( \mbf{Y}, \mbf{X})\, \partial_t\phi \right) \d \mbf{X} &= \iint\phi  \, \d \omega^{\mbf{Y}},\notag\\
 \iiint \big( A \nabla_X G( \mbf{Y}, \mbf{X})\cdot \nabla\phi  + G( \mbf{Y}, \mbf{X})\, \partial_t\phi \big) \, \d \mbf{X} &= \iint \phi  \, \d \widetilde{\omega}^{\mbf{Y}}.
\end{align*}

\subsection{\for{toc}{\small}PDE estimates} \label{sec:PDEestimates}
As in the preceding subsection, we refer to \cite{N97} for the results stated in 
Lemmas \ref{HCatBdry.lem}, \ref{carlest.lem}, \ref{strongharnack4gf.lem}, \ref{CFMS},  and  \ref{calisdoubpre.lem} below (we also mention \cite{FGS,FGS_BH,FS,FSY,LM} where one may find analogous results specifically for the heat equation, or for variable coefficient parabolic equations in Lipschitz cylinders). In particular, the proofs of the backwards Harnack inequality 
\ref{strongharnack4gf.lem} and the doubling property of parabolic measure \ref{calisdoubpre.lem} may be found in
\cite{FS,FSY}; in \cite{N97}, the author simply observes that the proofs in \cite{FS,FSY}
may be carried over to the Lip(1,1/2) setting, {\em mutatis mutandis}.

Note that all implicit constants in this subsection depend only on $n$, $\Lambda$ and $\|\psi\|_{\Lip(1,1/2)}$.

\begin{lemma}[Boundary Hölder continuity]\label{HCatBdry.lem}
Let $\mbf{X}\in\Sigma$ and $R>0$. Assume that
	$0 \le u\in C(\overline{I_{2R}(\mbf{X})\cap\Omega})$  satisfies
	$\mathcal{L} u=0$ in $I_{2R}(\mbf{X})\cap\Omega$, with
	$u=0$ in $\Delta_{2R}(\mbf{X})$.
Then there exists $\alpha \in (0, 1/2)$ such that
\[u(\mbf{Y}) \lesssim \left(\frac{\dist(\mbf{Y}, \Sigma)}{R}\right)^{\alpha} \sup_{\mbf{Z} \in I_{2R}(\mbf{X})\cap\Omega} u(\mbf{Z}),
\qquad 
\mbf{Y} \in  I_{R}(\mbf{X}) \cap \Omega.\]
\end{lemma}

\begin{lemma}[Carleson's estimate]\label{carlest.lem} Let $\mbf{X}$, $R$, $u$, and $\alpha$ be as in the statement of
Lemma \ref{HCatBdry.lem}. Then
\begin{equation*}
u(\mbf{Y}) \lesssim u(\cA^+_R(\mbf{X})),
\qquad 
\mbf{Y} \in  I_{R}(\mbf{X}) \cap \Omega,
\end{equation*}
where $\cA^+_R(\mbf{X})$, is the time-forward corkscrew point,
defined in \eqref{CSpm}, relative to the ``surface ball" $\Delta_R(\mbf{X})$.
In particular,
\begin{equation}\label{carlhcest.eq}
u(\mbf{Y}) \lesssim \left(\frac{\dist(\mbf{Y}, \Sigma)}{R}\right)^{\alpha} u(\cA^+_R(\mbf{X})),
\qquad 
\mbf{Y} \in  I_{R/2}(\mbf{X}) \cap \Omega.
\end{equation}
\end{lemma}

\begin{corollary}[Bourgain's estimate] \label{Bourgain} There is a uniform constant $c>0$ such that for all 
$\mbf{X} =(\psi(\mbf{x}),\mbf{x})\in\Sigma$, and every $R>0$,
\[
\hm^{\cA^+_R(\mbf{X})}\big(\Delta_R(\mbf{X})\big) \geq c\,.
\]
\end{corollary}

\begin{proof}[Sketch of proof]
First apply Lemma \ref{HCatBdry.lem} to the solution 
$u(\mbf{Y}):= 1- \hm^{\mbf{Y}}\big(\Delta_R(\mbf{X})\big)$, with $\mbf{Y}$ sufficiently close to the boundary, and then use Harnack's inequality to connect to $\cA^+_R(\mbf{X})$.  We omit the standard details.
\end{proof}

\begin{remark}
It is well-known that, in fact, 
the conclusion of Corollary \ref{Bourgain} is actually equivalent to that of Lemma
\ref{HCatBdry.lem}; see, e.g., \cite[Lemma 2.5]{GH1} for a proof
of the converse implication in a much more general setting.
\end{remark}

For $\mbf{X}= (x_0, x, t) \in \Sigma$,  $r > 0$, and
 $\kappa =\kappa(n,M_0)$, a sufficiently large constant to be fixed, we introduce the \textbf{space-time parabolas} by
\begin{equation} \label{paraboladef.eq}
	\P^\pm_{\kappa, r}(\mbf{X}) := \big\{(y_0,y,s) \in \Omega: \;\; |(x_0,x) - (y_0,y)| \le \kappa \, |t-s|^{1/2}, \; \; \pm(s-t)\ge 16r^{2}\big\}.
\end{equation}
Recall that $M_0 = 2 + \|\psi\|_{\Lip(1,1/2)}$ (see \eqref{1.1+}). The parabola $\P^+_{\kappa, r}$ is the forward in time parabola and $\P^-_{\kappa, r}$ is the backward in time parabola. Note that
if $\mbf{Y} \in \P^\pm_{\kappa, r}(\mbf{X})$, then $\mbf{Y} \in \P^\pm_{\kappa, r'}(\mbf{X})$ for all $r' \in (0,r)$. Concerning the ``aperture'' $\kappa$,
we may take this constant as large as we like, but we will choose
\begin{equation}\label{kappadef}\kappa:=  40M_0\!\sqrt{n}.
\end{equation}

\begin{lemma}[Backwards/Strong Harnack inequality]\label{strongharnack4gf.lem} If $\mbf{X} \in \Sigma$ and $\mbf{Y} \in \P^+_{\kappa, R}(\mbf{X})$, then
\[G(\mbf{Y}, \cA_R^-(\mbf{X})) \approx  G(\mbf{Y}, \cA_R^+(\mbf{X})).\]
Similarly, if $\mbf{X} \in \Sigma$ and  $\mbf{Y} \in \P^-_{\kappa, R}(\mbf{X})$,  then
\[  G(\cA_R^+(\mbf{X}),\mbf{Y})\approx G(\cA_R^-(\mbf{X}),\mbf{Y}).\]
\end{lemma}

\begin{lemma}[``CFMS estimate'' ]\label{CFMS}
If $\mbf{X} \in \Sigma$ and $\mbf{Y} \in \P^+_{\kappa, R}(\mbf{X})$, then
\[R^n \, G(\mbf{Y}, \cA_R^+(\mbf{X})) \approx \omega^{\mbf{Y}}(\Delta_{R}(\mbf{X})) \approx R^n \, G(\mbf{Y}, \cA_R^-(\mbf{X})).\]
Similarly, if $\mbf{X} \in \Sigma$ and $\mbf{Y} \in \P^-_{\kappa, R}(\mbf{X})$, then
\[R^n \, G(\cA_R^-(\mbf{X}),\mbf{Y}) \approx \widetilde{\omega}^{\mbf{Y}}(\Delta_{R}(\mbf{X})) \approx R^n\, G(\cA_R^+(\mbf{X}),\mbf{Y}).\]
\end{lemma}

\begin{lemma}[Doubling property]\label{calisdoubpre.lem}
If $\mbf{X} \in \Sigma$ and $\mbf{Y} \in \P^+_{\kappa, R}(\mbf{X})$, then
\[\omega^{\mbf{Y}}(\Delta_{R}(\mbf{X})) \approx \omega^{\mbf{Y}}(\Delta_{R/2}(\mbf{X})).\]
Similarly, if $\mbf{X} \in \Sigma$ and $\mbf{Y} \in \P^-_{\kappa, R}(\mbf{X})$, then
\[\widetilde{\omega}^{\mbf{Y}}(\Delta_{R}(\mbf{X})) \approx \widetilde{\omega}^{\mbf{Y}}(\Delta_{R/2}(\mbf{X})).\]
\end{lemma}

Actually, we will use the following corollary of Lemma \ref{calisdoubpre.lem}.

\begin{lemma}[{\cite[Lemma 3.12]{BHMN1}}]\label{calisdoub.lem}
Let $\kappa$ be as in \eqref{kappadef} and consider  $\mbf{X} \in \Sigma$ and $r > 0$. Then,
\[\cA^+_{4r}(\mbf{X}) \in \P^+_{\kappa, 2\rho}(\mbf{Z})\]
for all $\mbf{Z} \in \Sigma$ and $\rho > 0$ such that $\Delta_{2\rho}(\mbf{Z}) \subseteq \Delta_r(\mbf{X})$. In particular, for such $\mbf{Z}$ and $\rho$, it holds 
\[\omega^{\cA^+_{4r}(\mbf{X})}(\Delta_{2\rho}(\mbf{Z})) \approx \omega^{\cA^+_{4r}(\mbf{X})} (\Delta_{\rho}(\mbf{Z})).\]
\end{lemma}

\section{The assumptions in Theorem \ref{main1.thrm} and Theorem \ref{main2.thrm}}

Let us first define what we mean by $L^p$ solvability of the Dirichlet problem, the assumption on our main theorems. Actually, we are not going to use it directly, so we refer the interested reader to \cite{GH1} for further details.
\begin{definition}[$L^p$ solvability of the Dirichlet problem] \label{Lpsolv.def}
	Let $\Omega \subseteq \ree$ be the region above a $\Lip(1, 1/2)$ function, and $A$ an elliptic matrix satisfying \eqref{ellip.eq}. We say that the $L^p$ Dirichlet problem for $\mathcal{L}$, as in \eqref{L.eq}, is solvable if there exists some $p < \infty$ such that the following holds: for any $f \in L^p(\Sigma)$, the solution $u$ of the Dirichlet boundary value problem $\mathcal{L}u = 0$ in $\Omega$, $u = f$ on $\Sigma$ (understood in the sense of parabolic non-tangential convergence), provided by the parabolic measure, satisfies the estimate $\| N_* u \|_{L^p(\Sigma)} \lesssim \|f\|_{L^p(\Sigma)}$ (with constants uniform on $f$). Here we denote by $N_*u$ the nontangential maximal function of $u$ (for the precise definition, see \cite[Section 4]{GH1}).
\end{definition}

As commented above, we will not use this property directly, but rather another condition which is easier to deal with, and we define next.

\begin{definition}($A_\infty$ property) \label{defainfty}  
Recall that the parabolic and adjoint parabolic measures satisfy appropriate doubling
properties, by virtue of Lemma \ref{calisdoubpre.lem}.  
We therefore say that $\omega$ is in $A_\infty$ 
(or that its density is a parabolic $A_\infty$ weight) 
if there exist $C_\ast \geq 1$ and $q  > 1$ such that 
the following holds: if $\mbf{X} \in \Sigma$, $r > 0$, and $\mbf{Y} \in \P^+_{\kappa, 2r}(\mbf{X})$, then $\omega^{\mbf{Y}} \ll \sigma$ on $\Delta_{2r}(\mbf{X})$ and $k^{\mbf{Y}} := {\d\omega^{\mbf{Y}}}/{\d\sigma}$ satisfies the reverse-H\"older inequality
\begin{equation}\label{rhq}
\left( \bariint_{\Delta_{r}(\mbf{X})} (k^\mbf{Y})^q \, \d\sigma \right)^{1/q} \le C_\ast \bariint_{\Delta_{r}(\mbf{X})} k^\mbf{Y} \, \d\sigma.
\end{equation} 

Similarly, we say $\widetilde{\omega}$
 is in $A_\infty$ (or that its density is a parabolic $A_\infty$ weight) if there
exist $C_\ast \geq 1$ and $q  > 1$ such that the following holds. If $\mbf{X} \in \Sigma$, $r > 0$, and $\mbf{Y} \in \P^-_{\kappa, 2r}(\mbf{X})$, then $\widetilde{\omega}^{\mbf{Y}} \ll \sigma$ on $\Delta_{2r}(\mbf{X})$ and $\tilde{k}^{\mbf{Y}} := {\d\omega^{\mbf{Y}}}/{\d\sigma}$ satisfies the reverse-H\"older inequality
\begin{equation*}
\left( \bariint_{\Delta_{r}(\mbf{X})} (\tilde{k}^\mbf{Y})^q \, \d\sigma \right)^{1/q} \le C_\ast \bariint_{\Delta_{r}(\mbf{X})} \tilde{k}^\mbf{Y} \, \d\sigma.
\end{equation*}
\end{definition}

\begin{remark}\label{whnweusrhq.rmk}
Consider $\mbf{X}^0 \in \Sigma$ and $R_\ast > 0$, Using  Lemma \ref{calisdoub.lem} we have that
\[\cA^+_{4R_\ast}(\mbf{X}^0) \in \P^+_{\kappa, 2\rho}(\mbf{Z}),\]
for all $\mbf{Z} \in \Sigma$ and $\rho > 0$ such that $\Delta_{2\rho}(\mbf{Z}) \subseteq \Delta_{R_\ast}(\mbf{X}^0)$. Thus  \eqref{rhq} is valid with $\mbf{Y} = \cA^+_{4R_\ast}(\mbf{X}^0)$, i.e.,
\[\left( \bariint_{\Delta_{r}(\mbf{Z})} (k^{\cA^+_{4R_\ast}(\mbf{X}^0)}) ^q \, \d\sigma \right)^{1/q}\le C_\ast \bariint_{\Delta_{r}(\mbf{Z})} k^{\cA^+_{4R_\ast}(\mbf{X}^0)} \, \d\sigma,
\qquad \mbf{Z} \in \Delta_{R_*/2}(\mbf{X}^0), r < R_*/2.\]
\end{remark}

As anticipated above, we can actually use the handier $A_\infty$ condition instead of the $L^p$ solvability of the Dirichlet problem. Therefore, from here on, we will just use the $A_\infty$ condition as hypothesis for Theorems~\ref{main1.thrm} and \ref{main2.thrm}. This is justified by the following result.
\begin{theorem}\label{th4.5}
Let $\Omega \subseteq \ree$ be the region above a $\Lip(1, 1/2)$ function, and $A$ an elliptic matrix satisfying \eqref{ellip.eq}. Then, the parabolic measure associated to the operator $\mathcal{L}$ with coefficients $A$ (as in \eqref{L.eq}) is in $A_\infty$ (with exponent $q > 1$) if and only if the $L^{q'}$ Dirichlet problem for $\mathcal{L}$ is solvable in $\Omega$ (where $1/q + 1/q' = 1$, so concretely $q' < \infty$).
\end{theorem}

The equivalence stated in Theorem \ref{th4.5} is well known, but see, e.g.,
\cite{LM} and \cite{N97}, and also
\cite[Theorem 2.10]{GH1} for a proof in a much more general setting.

Now we give the definition of our extra condition on the matrix of coefficients $A$ of our PDE.

\begin{definition}[Smoothness and $L^1$-Carleson oscillation condition]\label{smoothL1carl.def}
Let $A$ be an $n \times n$ elliptic matrix on $\mathbb{R}^{n+1}$ as in \eqref{ellip.eq}. We say $A$ satisfies an \textbf{$L^1$-Carleson oscillation condition} if $A$ is $C^\infty(\Omega)$ and there exists a constant $C_A > 0$ such that for every $(X,t) \in \Omega$ and $k = 1,2$ it holds
\begin{equation}\label{smoothnessonA.eq}
\dist((X,t),\partial\Omega)^k \, |\nabla^k_X A(X,t)| 
+ \dist((X,t),\partial\Omega)^{2k}\, |\partial_t^k A(X, t)| 
+ \dist((X,t),\partial\Omega)^3 \, |\partial_t \nabla_X A(X, t)|
\le C_A, 
\end{equation} 
and $\mu$ is a Carleson measure relative to $\Omega$ with constant bounded by $C_A$, where 
\begin{equation} \label{eq:carleson_measure}
	\d\mu(X,t) := \big( \, |\nabla_X A(X,t)| \, + \, |\partial_t A(X,t)| \, \dist((X,t),\partial\Omega) \big)  \d X \d t.
\end{equation}  
Here $\mu$ being a Carleson measure relative to $\Omega$ means for every $\xbf \in \partial \Omega$ and $\rho > 0$ it holds
\begin{equation} \label{carlesonmeasure.eq}
	\mu(\mathcal{J}_\rho(\xbf) \cap \Omega)\le C_A\, \rho^{n+1}.
\end{equation}
\end{definition}

Note that if 
\[\d\tilde{\mu}(X,t) := \big( \, |\nabla_X A(X,t)|^2\dist((X,t),\partial\Omega)  + |\partial_t A(X,t)|^2\dist((X,t),\partial\Omega)^3 \big) \d X \d t,\]
then it still holds, simply by using \eqref{smoothnessonA.eq} and the Carleson condition on $\mu$,
\begin{equation} \label{L2Carleson.eq}
	\tilde{\mu}(\mathcal{J}_\rho(\xbf) \cap \Omega) \lesssim \rho^{n+1}, 
\qquad \xbf \in \partial \Omega, \; \rho > 0.
\end{equation}
This means that, as is well known, our $L^1$ Carleson condition is stronger than the usual $L^2$ Carleson condition frequently considered in the literature. 
As remarked in the Introduction, the essential obstacle to working with 
the weaker $L^2$ Carleson condition is that at present, we are unable to prove
the local square function estimates of Lemma \ref{mainsfest.lem} 
without the $L^1$ hypothesis.

We recall that the smoothness assumption in \eqref{smoothnessonA.eq}
can be relaxed, at least in the context of Theorem \ref{main1.thrm}:
see Remark \ref{remark-smoothL1carl},
and Appendix \ref{extendrmks.sect}.

\section{Step 1: Setting up initial stopping time regimes and the preliminary estimates}

In the remaining part of the paper, we prove Theorems \ref{main1.thrm} and \ref{main2.thrm}. For that, we will follow a sequence of steps, some of which are similar to those in \cite{BHMN1}. Nevertheless, there will be, at some points, substantial differences to be able to deal with our variable coefficient matrices.

In this first step, we start by exploiting (in Subsection~\ref{sec:corona_A_infty}) the $A_\infty$ assumption for $\omega$ (and $\widetilde \omega$ if proving Theorem~\ref{main2.thrm}) to create ample (sawtooth) regions where $\omega$ is ``comparable'' to $\sigma$. Later, in Subsection~\ref{sec:green_whitney}, we will analyze the behavior of the Green function. Here we will find the most challenging difficulty as compared to \cite{BHMN1}: in our setting, we are only able to show that the Green function grows in the vertical direction $x_0$ (which is essential to run the forthcoming arguments) only in Whitney regions associated to the ``$A_\infty$ sawtooths'', rather than in whole strips surrounding the boundary as in \cite{BHMN1}. This will have a strong influence in the rest of the paper, because it forces us to use coronizations through the whole argument.

\begin{remark}
In the sequel, all constants (explicit and implicit) may depend on the \textit{structural constants} $n, \Lambda, \|\psi\|_{\Lip(1, 1/2)}$, and the $A_\infty$ constants of $\omega$ (namely, $C_*$ and $q$);  and also the $A_\infty$ constants of $\widetilde \omega$ in the case of Theorem~\ref{main2.thrm}. 
\end{remark}

For the time being, we will keep track of dependencies on the constant $K$ defining the Whitney regions (see Subsection~\ref{whit.sec}), but later we will choose it depending only on the structural constants.

\subsection{\for{toc}{\small}Control of the parabolic measure in ``ample'' stopping-time regimes} \label{sec:corona_A_infty}

The following lemma will be the primary tool to produce our initial corona decomposition (one for the parabolic measure). We remind the reader that $\mbf{\Psi}(Q)$ is the ``surface ball" on the graph $\Sigma$, obtained by ``lifting" the parabolic cube $Q\subset \rn$ to the graph;
see \eqref{surface-box:graph}.

We now present two Lemmata that will be useful in the proof of
Theorem \ref{main1.thrm} (i.e., in the case that $A$ is symmetric).  
For the general case considered
in Theorem \ref{main2.thrm}, we will require  
 Lemma \ref{pmcorona2.eq} below.

\begin{lemma}\label{compsaw.lem}
Given $Q_0 = Q_R( {\bf x}_{Q_0}) \in \dd$ (i.e. it has center $\mbf{x}_{Q_0}$ and `radius' $R$), we define
\begin{equation}\label{CSfortop.eq}
\qquad \xbf^\pm_{Q_0}:= \cA^\pm_{M_\star R}(\psi(\mbf{x}_{Q_0}), \mbf{x}_{Q_0}), 
\qquad \text{where } M_\star := 10^{100}M_0^2n.
\end{equation}
Then, under the hypotheses of Theorem \ref{main1.thrm}, 
there exists a constant $M' \geq 1$ 
and a collection $\mathcal{F} = \{Q_j\}_{j}$ of pairwise disjoint subcubes of ${Q_0}$ satisfying 
\begin{equation}\label{ampcntct.eq}
\bigg|\bigcup_{Q_j \in \mathcal{F}} Q_j\bigg| \le \frac14 |Q_0|,
\end{equation}
and also
\begin{equation}  \label{sawestm1.eq}
\frac{1}{M'} \frac{\omega^{\xbf^+_{Q_0}}(\mbf{\Psi}(Q_0))}{\sigma(\mbf{\Psi}(Q_0))} 
\le 
\frac{\omega^{\xbf^+_{Q_0}}(\mbf{\Psi}(Q))}{ \sigma(\mbf{\Psi}(Q))} 
\le 
M' \frac{\omega^{\xbf^+_{Q_0}}(\mbf{\Psi}(Q_0))}{\sigma(\mbf{\Psi}(Q_0))}, 
\quad \forall Q \in \mathbb{D}(Q_0), Q \not\subseteq Q_j \ \forall Q_j \in \mathcal{F}.
\end{equation}
\end{lemma}

The proof of Lemma \ref{compsaw.lem} is a fairly routine consequence of
the $A_\infty$ property of $\hm$, see \cite[Lemma 4.9]{BHMN1}. Now, given $Q_0$ as in Lemma \ref{compsaw.lem} we may define the semi-coherent stopping time regime by
\begin{equation}\label{saw4pm.eq}
\sbf^* := \sbf^*_{Q_0} := \dd_{\mathcal{F}, Q_0} := \{Q \in \dd(Q_0): Q \not \subseteq Q_j, \ \forall Q_j \in \mathcal{F}\}.
\end{equation}
Note that \eqref{ampcntct.eq} gives the following estimate on the ``contact region" for the discrete sawtooth
\begin{equation}\label{ampcntct2.eq}
|E_{Q_0}| \ge (3/4) |Q_0|, \quad \text{where} \quad E_{Q_0}:= Q_0 \setminus \cup_{Q_j \in \mathcal{F}} Q_j.
\end{equation}

Now, producing our ``corona decomposition for the parabolic measure'' is a standard argument: it is just a matter of iterating Lemma~\ref{compsaw.lem}. 
\begin{lemma}\label{pmcorona.eq}
Suppose that the hypotheses of Theorem \ref{main1.thrm} hold. Then, there exist constants $C', M' > 1$ such that following holds. We can decompose the dyadic grid as
\[\dd = \mathcal{G}^* \sqcup \mathcal{B}^*\]
where the `good cubes' $\mathcal{G}^*$ are further decomposed into disjoint semi-coherent stopping time regimes $\{\sbf^*\}$ (i.e. $\mathcal{G}^* = \sqcup \; \sbf^*$) so that the maximal cubes $\{Q(\sbf^*)\}_{\sbf^*}$ and the bad cubes pack, that is,
\begin{equation*}
		\sum_{Q' \in \mathcal{B}^* \cap \mathbb{D}(Q)} |Q'|
		+ \sum_{\mbf{S} : Q(\sbf^*) \subseteq Q} |Q(\sbf^*)|
		\leq 
		C' |Q|, 
		\quad 
		\forall Q \in \mathbb{D},
	\end{equation*}
and for each $\sbf^*$ it holds
\begin{equation}\label{sawestm2.eq}
\frac{1}{M'}\frac{\omega^{\xbf^+_{{Q(\sbf^*)}}}(\mbf{\Psi}(Q(\sbf^*)))}{\sigma(\mbf{\Psi}(Q(\sbf^*)))} 
\le 
\frac{\omega^{\xbf^+_{{Q(\sbf^*)}}}(\mbf{\Psi}({Q}))}{ \sigma(\mbf{\Psi}({Q}))} 
\le 
M' \frac{\omega^{\xbf^+_{{Q(\sbf^*)}}}(\mbf{\Psi}(Q(\sbf^*)))}{\sigma(\mbf{\Psi}(Q(\sbf^*)))}, \quad \forall Q \in \sbf^*.
\end{equation}
\end{lemma}
Given Lemma \ref{compsaw.lem}, we observe that Lemma \ref{pmcorona.eq} is simply 
an illustration of the well-known fact that ``Big Pieces implies Corona".
Since the proof is standard,
let us only sketch it. First, using Lemma \ref{dsinitialcor.lem}, it is enough to produce a corona decomposition for $\dd(Q')$ for each $Q' \in \mathcal{C}$ (with constants independent of $Q'$). With this in mind, we fix $Q' \in \mathcal{C}$, and apply Lemma \ref{compsaw.lem} with $Q_0 := Q'$, hence obtaining $\sbf^*_{Q'} = \dd_{\mathcal{F}, Q'}$, a semi-coherent stopping time regime. Applying again Lemma~\ref{dsinitialcor.lem} wherever we stopped before, i.e., at each $Q_j \in \mathcal{F}$, we obtain new stopping time regimes $\sbf^*_{Q_j}$. Iterating this, we end up exhausting $\dd(Q')$ by stopping-time regimes with the estimate \eqref{sawestm2.eq}. Indeed, this produces a `coronization' of $\dd(Q')$ with no bad cubes, where the estimates \eqref{sawestm2.eq} hold within each stopping time $\sbf^*$. In order to show that the maximal cubes $\{Q(\sbf^*)\}_{\sbf^*}$ satisfy the necessary packing condition, we use \eqref{ampcntct2.eq}. Indeed, by construction the sets $\{E_{Q(\sbf^*)}\}_{\sbf^*}$ are disjoint (up to a set of measure zero) and fixing $Q \in \dd(Q')$ it holds, using \eqref{ampcntct2.eq}, that
\begin{equation}\label{crnapmmaxpk.eq}
	\sum_{\sbf^*: Q(\sbf^*) \subseteq Q} |Q(\sbf^*)| \le \frac43 \sum_{\sbf^*: Q(\sbf^*) \subseteq Q} |E_{Q(\sbf^*)}| \le \frac43|Q|.
\end{equation}

In the case of Theorem \ref{main2.thrm}, almost identical techniques yield the following result.
\begin{lemma}\label{pmcorona2.eq}
Suppose that the hypotheses of Theorem \ref{main2.thrm} hold. Then there exists constants $C', M' > 1$ 
such that following holds. The dyadic grid $\dd$ can be decomposed as 
\[\dd = \mathcal{G}^* \sqcup \  \mathcal{B}^*,\]
where the `good cubes' $\mathcal{G}^*$ are further decomposed into disjoint semi-coherent stopping time regimes $\{\sbf^*\}$ (i.e. $\mathcal{G}^* = \sqcup \; \sbf^*$) so that the maximal cubes $\{Q(\sbf^*)\}_{\sbf^*}$ and the bad cubes pack, that is,
\begin{equation*}
		\sum_{Q' \in \mathbb{D}_Q \cap \mathcal{B}^*} |Q'|
		+ \sum_{\mbf{S} : Q(\sbf^*) \subseteq Q} |Q(\sbf^*)|
		\leq 
		C' |Q|, 
		\quad 
		\forall Q \in \mathbb{D},
	\end{equation*}
and for each $\sbf^*$ it holds
\begin{equation}\label{sawestm3.eq}
\frac{1}{M'} \frac{\omega^{\xbf^+_{{Q(\sbf^*)}}}(\mbf{\Psi}(Q(\sbf^*)))}{ \sigma(\mbf{\Psi}(Q(\sbf^*)))} \le \frac{\omega^{\xbf^+_{{Q(\sbf^*)}}}(\mbf{\Psi}({Q}))}{ \sigma(\mbf{\Psi}({Q}))} \le M' \frac{\omega^{\xbf^+_{{Q(\sbf^*)}}}(\mbf{\Psi}(Q(\sbf^*)))}{\sigma(\mbf{\Psi}(Q(\sbf^*)))}, \quad \forall Q \in \sbf^*,
\end{equation}
and additionally
\begin{equation}\label{sawestm4.eq}
\frac{1}{M'} \frac{\widetilde{\omega}^{\xbf^-_{{Q(\sbf^*)}}}(\mbf{\Psi}(Q(\sbf^*)))}{ \sigma(\mbf{\Psi}(Q(\sbf^*)))} \le \frac{\widetilde{\omega}^{\xbf^-_{{Q(\sbf^*)}}}(\mbf{\Psi}({Q}))}{ \sigma(\mbf{\Psi}({Q}))} \le M' \frac{\widetilde{\omega}^{\xbf^-_{{Q(\sbf^*)}}}(\mbf{\Psi}(Q(\sbf^*)))}{\sigma(\mbf{\Psi}(Q(\sbf^*)))}, \quad \forall Q \in \sbf^*.
\end{equation}
\end{lemma}

The only difference in proving Lemma \ref{pmcorona2.eq}, in comparison to Lemma \ref{pmcorona.eq}, is using a slightly modified version of Lemma \ref{compsaw.lem}. Indeed, we should run the argument in Lemma \ref{compsaw.lem} for both $\omega$ and $\widetilde{\omega}$ separately, hence obtaining $\mathcal{F}$ and $\widetilde{\mathcal{F}}$ so that \eqref{ampcntct.eq} holds for $\mathcal{F}$ and also for $\widetilde{\mathcal{F}}$, 
\[\frac{1}{M'} \frac{\widetilde{\omega}^{\xbf^-_{Q_0}}(\mbf{\Psi}(Q_0))}{ \sigma(\mbf{\Psi}(Q_0))} \le \frac{\widetilde{\omega}^{\xbf^-_{Q_0}}(\mbf{\Psi}(Q))}{ \sigma(\mbf{\Psi}(Q))} \le M' \frac{\widetilde{\omega}^{\xbf^-_{Q_0}}(\mbf{\Psi}(Q_0))}{\sigma(\mbf{\Psi}(Q_0))}, \quad \forall Q \in \mathbb{D}(Q_0), Q \not\subseteq \widetilde{Q}_j \ \forall \widetilde{Q}_j \in \widetilde{\mathcal{F}}.\]
Then, letting $\mathcal{F}'$ denote the maximal cubes (with respect to containment) in $\mathcal{F} \cup \widetilde{\mathcal{F}}$, we have
\begin{equation}\label{amcntctboth.eq}
|\cup_{Q_j \in \mathcal{F}'} Q_j| \le (1/2)|Q_0|
\end{equation}
and the estimates \eqref{sawestm3.eq} and \eqref{sawestm4.eq} hold for $\sbf^* := \sbf^*(Q_0)$ now defined by 
\[\sbf^*(Q_0) := \dd_{\mathcal{F}', Q_0} := \{Q \in \dd(Q_0): Q \not \subseteq Q'_j, \ \forall Q'_j \in \mathcal{F}'\}.\]
Then all of the arguments that follow are identical, since \eqref{amcntctboth.eq} can be used for the packing of the maximal cubes with the only difference being that the constant in \eqref{crnapmmaxpk.eq} changes from $4/3$ to $2$.

\subsection{\for{toc}{\small}Level lines and estimates for the Green function in Whitney regions} \label{sec:green_whitney}

In the sequel we fix a semi-coherent stopping time regime $\sbf^*$ from Lemma \ref{pmcorona.eq} (or Lemma \ref{pmcorona2.eq} if we are proving Theorem \ref{main2.thrm}). We are going to further decompose $\sbf^*$ into stopping times where nice estimates on the Green function hold (and therefore, there will be ``bad cubes" where these estimates fail).

Just as in \cite{BHMN1}, we need to transfer information about the parabolic measure to the Green function. We define the normalized Green function 
\begin{equation} \label{defu.eq}
	u(\mbf{Y}) := \sigma(Q(\sbf^*))\, G(\xbf^+_{{Q(\sbf^*)}}, \mbf{Y}),
\end{equation}
which solves $\mathcal{L}^*u = 0$ in $\{ (X, t) \in \Omega : t < t(\xbf^+_{{Q(\sbf^*)}}) \}$.
In the case we are proving Theorem \ref{main2.thrm}, we also define the normalized adjoint Green function (solving $\mathcal{L}v = 0$ in $\{ (X, t) \in \Omega : t > t(\xbf^-_{{Q(\sbf^*)}}) \}$)
\begin{equation*}
v(\mbf{X}) :=  \sigma(Q(\sbf^*))\, G(\mbf{X}, \xbf^-_{{Q(\sbf^*)}}) . 
\end{equation*}

We note that, since the coefficients $A$ of our equation are locally Hölder continuous by \eqref{ellip.eq} and \eqref{smoothnessonA.eq}, both $u$ and $v$ are locally $C^{2, \alpha}$ for $\alpha > 0$ (in a parabolic sense, see \cite[Chapter 4]{L}) inside $\Omega$, away from their poles, by standard regularity results (like Nash's theorem and Schauder regularity, see \cite[Chapters 5 and 6]{L}). Concretely, we will use (without further mention) the fact that $u$ and $v$ are (qualitatively) twice differentiable inside $\Omega$ (away from the poles).

To produce our initial estimates on the Green function we use Section \ref{bbest}. Actually, we want to obtain estimates on our Whitney regions $U_Q$ (from subsection~\ref{whit.sec}). Recall that these constructions included a constant $K$, that we will take sufficiently large below. In any case, we will insist that 
\begin{equation} \label{eq:K_big}
	K = 2^N, \quad N > {2000}, 
\end{equation}
so that $K >10^{500}$. With this, we obtain a consequence of the results in Section~\ref{sec:PDEestimates} and Lemma~\ref{pmcorona.eq}.
\begin{lemma}\label{gfcomplemt1.eq}
Suppose that the hypotheses of Theorem \ref{main1.thrm} hold, and that $K = 2^N,  N > {2000}$. Let $Q \in \sbf^*$ (Lemma~\ref{pmcorona.eq}). Then there exist two constants $M_1, M_2 > 1$ ($M_2$ depending on $K$)
such that the following holds. For any $Q \in \sbf^*$, without restriction on its side-length, 
\begin{equation}\label{M1def.eq}
(M_1)^{-1} \le \frac{u(\mbf{Y})}{\dist(\mbf{Y}, \partial \Omega)} \le M_1,
\end{equation}
whenever (recall that $\mbf{x}_Q$ denotes the center of $Q$)
\begin{equation}\label{whereM1esthlds.eq}
\mbf{Y} \in \big{\{} (y_0,y, s) \in \ree: \; (y,s) \in 10^{10}Q, \;\; y_0 = \psi(\mbf{x}_Q) + 2M_010^{8}\ell(Q) \big{\}}.
\end{equation}
Moreover, if $Q$ satisfies the additional property that $\ell(Q) \le K^{-10}\ell(Q(\sbf^*))$, then for all $\mbf{X} \in U_Q^{***}(K)$
\begin{equation}\label{M2def.eq}
(M_2)^{-1} \le \frac{u(\mbf{Y})}{\dist(\mbf{Y}, \partial \Omega)} \le M_2, \qquad \forall \, \mbf{Y} \in B\left(\mbf{X}, \frac{\dist(\mbf{X},\pom)}{10K}\right)    .
\end{equation}
\end{lemma}
\begin{proof} 
	Let us rapidly sketch how to obtain \eqref{M1def.eq}. Indeed, the upper bound follows like this:
	\begin{multline*}
		u(\mbf{Y})
		\; = \; 
		\sigma(Q(\sbf^*)) \, G(\mbf{X}_{Q(\sbf^*)}^+, \mbf{Y})
		\; \lesssim \; 
		\sigma(Q(\sbf^*)) \, G\big(\mbf{X}_{Q(\sbf^*)}^+, \cA_{10^{12}\ell(Q)}^-(\psi(\mbf{x}_Q), \mbf{x}_Q)\big)
		\\ \approx  \;
		\sigma(Q(\sbf^*))  \, \ell(Q)^{-n} \,  \omega^{\mbf{X}_{Q(\sbf^*)}^+}\big(\Delta_{10^{12}\ell(Q)}(\psi(\mbf{x}_Q), \mbf{x}_Q)\big)
		\; \approx  \; 
		\sigma(Q(\sbf^*)) \, \ell(Q)^{-n} \,  \omega^{\mbf{X}_{Q(\sbf^*)}^+}(\mbf{\Psi}(Q))
		\\ \approx \;
		\sigma(\mbf{\Psi}(Q)) \, \ell(Q)^{-n} \, 
\omega^{\mbf{X}_{Q(\sbf^*)}^+}(\mbf{\Psi}(Q(\sbf^*)))
		\; \approx \;
		\ell(Q)
		\; \approx \;
		\dist(\mbf{Y}, \Sigma).
	\end{multline*}
Let us explain more carefully each step. The first step is the definition of $u$ in \eqref{defu.eq}. The second follows from the parabolic Harnack's inequality for adjoint solutions like $u$ (see \cite[Theorem 0.2]{FGS_BH} for a precise statement). The third follows from the CFMS estimate in Lemma~\ref{CFMS} because $\mbf{X}_{Q(\sbf^*)}^+ = \cA^+_{M_\star \diam(Q(\sbf^*))}(\psi(\mbf{x}_{Q(\sbf^*)}), \mbf{x}_{Q(\sbf^*)}) \in \P_{\kappa, 10^{12}\ell(Q)}^+ (\psi(\mbf{x}_Q), \mbf{x}_Q)$ (see the definitions of all the objects in \eqref{paraboladef.eq}, \eqref{kappadef} and \eqref{CSfortop.eq}). The fourth follows from the doubling property of the parabolic measure in Lemma~\ref{calisdoubpre.lem}. The fifth comes from Lemma~\ref{pmcorona.eq}. The second-to-last estimate follows from 
Corollary \ref{Bourgain}, and the fact 
that $\sigma(\mbf{\Psi}(Q)) \approx \ell(Q)^{n+1}$ because $\psi$ is Lip(1,1/2),
and the last estimate from Lemma~\ref{distgraphlem.lem} and \eqref{whereM1esthlds.eq}.
	
	In turn, to obtain the lower bound in \eqref{M1def.eq}, we can repeat the computations with the point $\cA_{10^{12}\ell(Q)}^+(\psi(\mbf{x}_Q), \mbf{x}_Q)$ replacing $\cA_{10^{12}\ell(Q)}^-(\psi(\mbf{x}_Q), \mbf{x}_Q)$, hence reversing Harnack's inequality.
	
	Lastly, to show \eqref{M2def.eq}, the argument is similar to the above using massive dilates of the cube $Q$ (having \eqref{eq:U_Q_in_balls} in mind), depending on $K$, so that we obtain the estimates for $u$ in the whole region $U_Q^{***}(K)$ (therefore the constants will depend on $K$, too). The assumption  $\ell(Q) \le K^{-10}\ell(Q(\sbf^*))$ ensures that the pole $\xbf^+_{{Q(\sbf^*)}}$ is still far to the right of these dilations, so that all the estimates from Section~\ref{sec:PDEestimates} continue to hold. We leave further details to the interested reader. 
\end{proof}

In the case that we are proving Theorem \ref{main2.thrm}, the analogous result reads as follows.
\begin{lemma}\label{gfcomplemt2.eq}
Suppose that the hypotheses of Theorem \ref{main2.thrm} hold and that $K = 2^N, N > {2000}$. Let $Q \in \sbf^*$ (Lemma \ref{pmcorona2.eq}). Then there exist two constants $M_1, M_2 > 1$ ($M_2$ depending on $K$)
such that the following holds. For any $Q \in \sbf^*$, without restriction on its side-length,
\begin{equation*}
(M_1)^{-1} \le \frac{u(\mbf{Y})}{\dist(\mbf{Y}, \partial \Omega)} \le M_1, 
\end{equation*}
whenever (recalling that $\mbf{x}_Q$ is the center of $Q$)
\begin{equation*}
\mbf{Y} \in \big{\{} (y_0,y, s): \; (y,s) \in 10^{10}Q,\;\; y_0 = \psi(\mbf{x}_Q) + 2M_010^{8}\ell(Q)\big{\}}.
\end{equation*}
Moreover, if $Q$ satisfies the additional property that $\ell(Q) \le K^{-10}\ell(Q(\sbf^*))$, then for all $\mbf{X} \in U_Q^{***}(K)$
\begin{equation}\label{M2defa.eq}
(M_2)^{-1} \le \frac{u(\mbf{Y})}{\dist(\mbf{Y}, \partial \Omega)} \le M_2, \qquad \forall \, \mbf{Y} \in B\left(\mbf{X}, \frac{\dist(\mbf{X},\pom)}{10K}\right),
\end{equation}
and also
\begin{equation}\label{M2defb.eq}
(M_2)^{-1} \le \frac{v(\mbf{Y})}{\dist(\mbf{Y}, \partial \Omega)} \le M_2,  \qquad \forall \, \mbf{Y} \in B\left(\mbf{X}, \frac{\dist(\mbf{X},\pom)}{10K}\right).
\end{equation}
\end{lemma}

For some of the forthcoming arguments, the estimates in Lemmas \ref{gfcomplemt1.eq} and \ref{gfcomplemt2.eq} in $U_Q^{***}$ will be enough. However, it will be crucial for us later to be able to use level sets of these Green functions, too. The following lemma makes it explicit, how these level sets traverse our Whitney regions.

\begin{lemma}\label{fndndpt.lem}
Suppose that the hypotheses of Theorem \ref{main1.thrm} or Theorem \ref{main2.thrm} hold. Then there exists $N_0 \ge 2000$ such that: if $K = 2^N$ for $N \ge N_0$, and $Q \in \sbf^*$ with $\ell(Q) \le K^{-10}\ell(Q(\sbf^*))$, then 
\begin{equation}\label{nondegpt.eq}
\mbf{y} \in 10Q 
\quad \implies \quad 
\exists \; y_0 \in \RR \;\; \text{with} \;\;\;
(y_0,\mbf{y}) \in U_Q(K) \quad \text{ and } \quad \partial_{y_0} u(y_0,\mbf{y}) \geq c_1,
\end{equation} 
for some $c_1 > 0$ which is independent of $K$.

Moreover, for every $\delta \in (0,1/2]$, there exists $N = N(\delta) \geq 2000$ 
such that if $K = 2^N$, then for each $Q \in \sbf^*$ with $\ell(Q) \le K^{-10}\ell(Q(\sbf^*))$ we have
\[
\mbf{y} \in 10^5Q, \;\; r \in [\delta\ell(Q), \delta^{-1}\ell(Q)] 
\quad \implies \quad 
\exists \; y_r \in \mathbb{R} \;\;
\text{so that} \;\;\;
(y_r, \mbf{y}) \in U_Q
\quad \text{ and } \quad 
u(y_r,\mbf{y}) = r.\]
\end{lemma}
\begin{proof}
Let us start by proving the second half of the lemma. Before that, note first that we can make $N(\delta)$ to be a decreasing function in $\delta$.

Fix $\delta \in (0,1/2]$ and $\mbf{y} \in 10^5Q$. Recalling the definition of $U_Q$ in \eqref{whitdef.eq}, and using the intermediate value theorem,
it suffices to show that for $K = 2^N$ with $N = N(\delta)$ large enough
\begin{equation}\label{smallfromhcatbdry.eq}
u(\psi(\mbf{y}) + K^{-1}\ell(Q),\mbf{y}) < \delta\ell(Q),
\end{equation}
and that there exists $y_0$ such that $(y_0, \mbf{y}) \in U_Q(K)$ with
\begin{equation}\label{bigfromanc.eq}
u(y_0,\mbf{y}) > \delta^{-1}\ell(Q).
\end{equation}

To show \eqref{smallfromhcatbdry.eq}, we use the Lemma \ref{carlest.lem}. Letting $\mbf{X} := \mbf{\Psi}(\mbf{x}_Q)$ and $R := 10^8\ell(Q)$, we clearly have
\[\mathcal{A}^-_R(\mbf{X}) \in \big{\{} (y_0,y, s): \; (y,s) \in 10^{10}Q, \;\; y_0 = \psi(\mbf{x}_Q) + 2M_010^{8}\ell(Q)\big{\}}.\]
Then, by \eqref{carlhcest.eq}, and later \eqref{M1def.eq} (along with Lemma~\ref{distgraphlem.lem} repeatedly), we obtain 
\[
u(\psi(\mbf{y}) + K^{-1}\ell(Q), \mbf{y})
\lesssim_{M_0}
\left( \frac{K^{-1} \ell(Q)}{10^8 \ell(Q)} \right)^\alpha u(\mathcal{A}^-_R(\mbf{X}))
\lesssim_{M_0, M_1}
\left( \frac{K^{-1} \ell(Q)}{R} \right)^\alpha R
\approx 
K^{-\alpha} \ell(Q),
\] 
from where \eqref{smallfromhcatbdry.eq} simply follows after choosing $K$ sufficiently large, depending on $\delta$.

Let us now show \eqref{bigfromanc.eq}. Here we are going to use \eqref{M1def.eq} again, but for an ancestor of $Q$. Let $Q^*$ be the unique cube in $\sbf^*$ which contains $Q$ with $\ell(Q^*) = 2^{\lfloor N/2 \rfloor + 1} \ell(Q) \approx K^{1/2} \ell(Q)$, which is a possible choice since $Q \in \sbf^*$ with $\ell(Q) \le K^{-10}\ell(Q(\sbf^*))$. Now consider the point 
\[\mbf{Y} := (\psi(\mbf{x}_{Q^*}) + 2M_010^{8}\ell(Q^*), \mbf{y}),\]
that is, the point over $\mbf{y}$ in the set \eqref{whereM1esthlds.eq} for $Q^*$. By \eqref{M1def.eq} (and Lemma~\ref{distgraphlem.lem}) it holds that
\begin{equation}\label{Khlfbd1.eq}
u(\mbf{Y}) \ge M_1\dist(\mbf{Y}, \partial\Omega) \gtrsim M_1 \ell(Q^*) \ge c K^{1/2} \ell(Q),
\end{equation}
where $c$ is independent of $K$. Moreover, by definition of $\mbf{Y}$ (and Lemma~\ref{distgraphlem.lem}) it holds that
\begin{equation}\label{Khlfbd2.eq}
\|\mbf{Y} - \mbf{\Psi}(\mbf{y})\| \approx K^{1/2}\ell(Q),
\end{equation}
where the implicit constant is independent of $K$. Hence, by choosing $K$ large enough (depending on $\delta$), \eqref{Khlfbd1.eq} implies that $u(\mbf{Y}) > \delta^{-1}\ell(Q)$, and \eqref{Khlfbd2.eq} ensures $(y_0, \mbf{y}) := \mbf{Y} \in U_Q(K)$ (see \eqref{whitdef.eq}). This finishes the proof of the second half of the lemma.

Finally, let us prove \eqref{nondegpt.eq}. Fix $N_0 := N(1/2)$ (i.e., put $\delta = 1/2$ in the above argument), and denote $K_0 := 2^{N_0}$ (so it is clear that both $N_0$ and $K_0$ only depend on $n$, $\|\psi\|_{\Lip(1,1/2)}$, and the $A_\infty$ constants for $\omega^{\xbf^+_{Q_0}}$). Then, for a given $\mbf{y} \in 10Q$, we repeat the argument above to find some $y_0$ such that $(y_0, \mbf{y}) \in U_Q(K_0)$ (therefore, by \eqref{whitdef.eq}, $y_0 \in (\psi(\mbf{y}) + K_0^{-1}\ell(Q), \psi(\mbf{y}) + K_0\ell(Q))$) such that
\[
u(\psi(\mbf{y}) + K_0^{-1}\ell(Q), \mbf{y}) 
< (1/2) \ell(Q)
< 2 \ell(Q)
< u(y_0, \mbf{y}). 
\]
Therefore, by the mean value theorem, 
there exists $x_0 \in (\psi(\mbf{y}) + K_0^{-1}\ell(Q), y_0)$ (so that $(x_0, \mbf{y}) \in U_Q(K_0) \subseteq U_Q(K)$ for any $K \geq K_0$, i.e., $N \geq N_0$) such that 
\[
\partial_{x_0} u(x_0, \mbf{y})
=
\frac{u(y_0, \mbf{y}) - u(\psi(\mbf{y}) + K_0^{-1}\ell(Q), \mbf{y}) }{y_0 - (\psi(\mbf{y}) + K_0^{-1}\ell(Q))}
\geq 
\frac{2\ell(Q) - (1/2) \ell(Q)}{(\psi(\mbf{y}) + K_0\ell(Q)) - (\psi(\mbf{y}) + K_0^{-1}\ell(Q))}
=
C(K_0)
> 0,
\]
which is indeed the result stated in the first half of the lemma.
\end{proof}

The conditions for our matrix $A$, allow us to obtain the following.
\begin{lemma}\label{ptwiseestinsaw4der.lem} 
Assume the hypotheses of Theorem \ref{main1.thrm} (or Theorem \ref{main2.thrm}). 
Let $\sbf^*$ be as in Lemma~\ref{pmcorona.eq} (or \ref{pmcorona2.eq}). Suppose $K = 2^N,  N > {2000}$. If $Q \in \sbf^*$ satisfies $\ell(Q) \le K^{-10}\ell(Q(\sbf^*))$,
then 
\begin{equation}\label{ptwisederest.eq}
|\nabla^k_Y u(\mbf{Y}) | \lesssim \dist(\mbf{Y}, \pom)^{1-k}, \quad |\partial_t u(\mbf{Y})| \lesssim \dist(\mbf{Y}, \pom)^{-1} \quad \forall \; \mbf{Y} \in U_Q^{***}, \; k = 0, 1,2,
\end{equation}
\begin{equation}\label{avgderest.eq}
 |\partial_t \nabla_Y u(\mbf{Y})| \lesssim \dist(\mbf{X}, \partial\Omega)^{-2}, \quad \forall \mbf{Y} \in U_Q^{***}.
\end{equation}
Moreover, in the case that we assume the hypotheses of Theorem \ref{main2.thrm}, then \eqref{ptwisederest.eq} and \eqref{avgderest.eq} also hold with $v$ in place of $u$.
\end{lemma}

\begin{proof}
	These estimates are a consequence of standard interior regularity for solutions of linear parabolic PDE with twice differentiable coefficients. Let us include a complete proof for completeness.
	
	Fix $\mbf{X} \in U_Q^{***}$, and denote by $\mathcal{J}$ a Whitney box containing $\mbf{X}$, for instance $\mathcal{J} := \mathcal{J}_{\dist(\mbf{X}, \pom)/100}(\mbf{X})$. By \cite[Theorem 4.9]{L} applied in $4\mathcal{J} \subset \Omega$ (check the definitions in \cite[Chapter IV]{L}),
	\begin{equation*} 
		\sup_{\mbf{Y} \in 4\mathcal{J}} \Big[ 
		\dist(\mbf{Y}, \partial(4\mathcal{J})) \, |\nabla u(\mbf{Y})| 
		+ \dist(\mbf{Y}, \partial(4\mathcal{J}))^2 \, |\nabla^2 u(\mbf{Y})| 
		+ \dist(\mbf{Y}, \partial(4\mathcal{J}))^2 \, |\partial_t u(\mbf{Y})| \Big]
		\lesssim 
		\| u \|_{L^\infty(4\mathcal{J})},
	\end{equation*}
	(where, across the proof, $\nabla$ denotes the spatial gradient)
	and also, for $\alpha > 0$, 
	\begin{equation} \label{eq:holder_u}  
		\sup_{\substack{(Y, t), (Y, s) \in 4\mathcal{J} \\ t \neq s}} 
		\min \big\{ \! \dist((Y, t), \partial(4\mathcal{J})), \dist((Y, s), \partial (4\mathcal{J})) \big\}^{2+\alpha} 
		\frac{|\nabla u(Y, t) - \nabla u(Y, s)|}{|(Y, t) - (Y, s)|^{1+\alpha}}
		\lesssim 
		\| u \|_{L^\infty(4\mathcal{J})}.
	\end{equation}
	The implicit constants are uniform because the $C^\alpha$ norm of $A$ is controlled by \eqref{ellip.eq} and \eqref{smoothnessonA.eq}.
	
	Note that, since $\mathcal{J}$ is a Whitney box and $\mbf{X} \in \mathcal{J}$, it holds $\dist(\mbf{X}, \partial(4\mathcal{J})) \approx \diam(\mathcal{J}) \approx \dist(\mbf{X}, \pom)$. Thus, the estimate in the first display yields \eqref{ptwisederest.eq} because $\|u\|_{L^\infty(4\mathcal{J})} \lesssim \sup_{\mbf{Y} \in 4\mathcal{J}} \dist(\mbf{Y}, \pom) \approx \dist(\mbf{X}, \pom)$ by Lemma~\ref{gfcomplemt1.eq} (and again properties of Whitney boxes).
	
	The argument to obtain \eqref{avgderest.eq} is similar. Differentiating the equation that $u$ solves, we get 
	\begin{equation*}
		\partial_t (\partial_t u) - \div(A \nabla (\partial_t u))
		=
		\div ((\partial_t A) \nabla u).
	\end{equation*}
	Hence, applying \cite[Theorem 4.8]{L} in $2\mathcal{J}$, we obtain, after some simplifications as before,
	\begin{multline} \label{eq:second_derivative}
		\dist(\mbf{X}, \pom) \, |\nabla (\partial_t u(\mbf{X}))|
		\lesssim 
		\| \partial_t u \|_{L^\infty(2\mathcal{J})}
		+ \dist(\mbf{X}, \pom) \, \| (\partial_t A) \nabla u \|_{L^\infty(2\mathcal{J})}
		\\ + \sup_{\mbf{Y} \neq \mbf{Z} \in 2\mathcal{J}} 
		\min \big\{ \! \dist(\mbf{Y}, \partial(2\mathcal{J})), \dist(\mbf{Z}, \partial (2\mathcal{J})) \big\}^{1+\alpha} 
		\frac{|\partial_t A(\mbf{Y}) \nabla u(\mbf{Y}) - \partial_t A(\mbf{Z}) \nabla u(\mbf{Z})|}{|\mbf{Y}-\mbf{Z}|^\alpha}.
	\end{multline}
	The first term in the right-hand-side is controlled (up to constants) by $\dist(\mbf{X}, \pom)^{-1}$ by \eqref{ptwisederest.eq}. So is the second one by using \eqref{ptwisederest.eq} and also \eqref{smoothnessonA.eq}.
	In turn, to control the last term, we compute
	\begin{multline*}
		|\partial_t A(\mbf{Y}) \nabla u(\mbf{Y}) - \partial_t A(\mbf{Z}) \nabla u(\mbf{Z})|
		\leq 
		|\partial_t A(\mbf{Y}) - \partial_t A(\mbf{Z})| \, |\nabla u(\mbf{Y})|
		+ |\partial_t A(\mbf{Z})| \, |\nabla u(\mbf{Y}) - \nabla u(\mbf{Z})|
		\\ \lesssim 
		\dist(\mbf{X}, \pom)^{-4} \, |\mbf{Y} - \mbf{Z}|^2 
		+ \dist(\mbf{X}, \pom)^{-3} \, |\mbf{Y} - \mbf{Z}|
		+ \dist(\mbf{X}, \pom)^{-2} \, |\nabla u(\mbf{Y}) - \nabla u(\mbf{Z})|,
	\end{multline*}
	where in the last estimate we have used the mean value theorem along with \eqref{smoothnessonA.eq}, and also \eqref{ptwisederest.eq}. Noting that $\dist(\mbf{Y}, \partial(2\mathcal{J})), |\mbf{Y} - \mbf{Z}| \lesssim \diam(\mathcal{J}) \approx \dist(\mbf{X}, \pom)$ if $\mbf{Y}, \mbf{Z} \in 2\mathcal{J}$, the last display yields
	\begin{multline} \label{eq:cross_terms}
		\sup_{\mbf{Y} \neq \mbf{Z} \in 2\mathcal{J}} 
		\min \big\{ \! \dist(\mbf{Y}, \partial(2\mathcal{J})), \dist(\mbf{Z}, \partial (2\mathcal{J})) \big\}^{1+\alpha} \, \frac{|\partial_t A(\mbf{Y}) \nabla u(\mbf{Y}) - \partial_t A(\mbf{Z}) \nabla u(\mbf{Z})|}{|\mbf{Y}-\mbf{Z}|^\alpha}
		\\ \lesssim 
		\dist(\mbf{X}, \pom)^{-1}
		+ \, \dist(\mbf{X}, \pom)^{-2} \!\! \sup_{\mbf{Y} \neq \mbf{Z} \in 2\mathcal{J}} \!\! 
		\min \big\{ \! \dist(\mbf{Y}, \partial(2\mathcal{J})), \dist(\mbf{Z}, \partial (2\mathcal{J})) \big\}^{1+\alpha}  \, \frac{|\nabla u(\mbf{Y}) - \nabla u(\mbf{Z})|}{|\mbf{Y}-\mbf{Z}|^\alpha}.
	\end{multline}
	Let us estimate the last term using the mean value theorem and \eqref{ptwisederest.eq}, where $\mbf{Y} = (Y, t), \mbf{Z} = (Z, s)$: 
	\begin{equation*}
		|\nabla u(\mbf{Y}) - \nabla u(\mbf{Z})| 
		\leq 
		|\nabla u(Y, t) - \nabla u(Z, t)|
		+ |\nabla u(Z, t) - \nabla u(Z, s)|
		\lesssim 
		\dist(\mbf{X}, \pom)^{-1} |\mbf{Y} - \mbf{Z}|
		+ |\nabla u(Z, t) - \nabla u(Z, s)|.
	\end{equation*}
	Inserting this back in \eqref{eq:cross_terms}, we obtain
	\begin{align*}
		&\sup_{\mbf{Y} \neq \mbf{Z} \in 2\mathcal{J}} \min \big\{ \! \dist(\mbf{Y}, \partial(2\mathcal{J})), \dist(\mbf{Z}, \partial (2\mathcal{J})) \big\}^{1+\alpha} \, \frac{|\partial_t A(\mbf{Y}) \nabla u(\mbf{Y}) - \partial_t A(\mbf{Z}) \nabla u(\mbf{Z})|}{|\mbf{Y}-\mbf{Z}|^\alpha}
		\\ &\;\;\; \lesssim 
		\dist(\mbf{X}, \pom)^{-1}
		\\[-.2cm] &\;\; \;\;\; \;\;\;+ \dist(\mbf{X}, \pom)^{-2} \!\!  \sup_{\substack{(Z, t), (Z, s) \in 2\mathcal{J} \\ t \neq s}} \! \! \min \big\{ \! \dist((Z, t), \partial(4\mathcal{J})), \dist((Z, s), \partial (4\mathcal{J})) \big\}^{2+\alpha} \, \frac{|\nabla u(Z, t) - \nabla u(Z, s)|}{|(Z, t)-(Z, s)|^{1+\alpha}}
		\\[-.2cm] & \;\;\;\lesssim 
		\dist(\mbf{X}, \pom)^{-1},
	\end{align*}
	where the last estimate is due to \eqref{eq:holder_u} (and in the first step, we also used that $\dist(\mbf{Y}, \partial (4\mathcal{J})) \approx \diam(\mathcal{J}) \approx \dist(\mbf{X}, \pom)$ because $\mbf{Y} \in 2\mathcal{J}$). Back in \eqref{eq:second_derivative}, this finishes the proof of \eqref{avgderest.eq}.
	
\end{proof}

\section{Step 2: A local square function estimate and refined corona decomposition} \label{sec:step2}

In this section, we will reduce the proofs of Theorems~\ref{main1.thrm} and \ref{main2.thrm} to constructing some regular $\Lip(1, 1/2)$ graphs that approximate our original graph $\Sigma$ in a corona decomposition sense (concretely, see Subsection~\ref{sec:reduction_corona}). To construct those graphs (which we will do in the next sections, we will use level lines of the Green's function $u$. And actually, it will be convenient that these level lines ``follow the shape'' of $\Sigma$: concretely, we would like that $\partial_{x_0}u > 0$ (uniformly). Indeed, we will prove in Subsection~\ref{sec:second_corona} that this property holds in our Whitney regions. The price we need to pay is that we need to refine our initial corona decomposition for the parabolic measure. And to prove the packing condition for the ``bad cubes'', we will need estimates for the regions where the Green's function behaves poorly: that will be the content of Subsection~\ref{sec:IBP}.

In fact, the integration by parts scheme that we will develop in Subsection~\ref{sec:IBP} to obtain square function estimates for the Green's function, is one of the main novelties of this paper. Inspired by the work of \cite{LN07} for the heat equation, we develop an integration by parts method that allows us to deal with symmetric matrices without needing to simultaneously use solutions to the forward equation (that for the operator $\mathcal{L}$) and the adjoint equation (for $\mathcal{L}^*$). Concretely, that allows us to only require the $L^p$ solvability for $\mathcal{L}$ (and not for $\mathcal{L}^*$) when proving Theorem~\ref{main1.thrm}.

\subsection{\for{toc}{\small}Local square function estimate via integration by parts} \label{sec:IBP}

In order to later rigorously integrate by parts, we need a smooth cutoff function adapted to sawtooth regions like $\Omega_{\sbf}$. The following lemma is sufficient for our purposes and we have put its proof in Appendix \ref{cutoff.sect}.
\begin{lemma}\label{cutofffnlem.lem}
Let $K \ge 4$ and $\sbf \subseteq \dd$ a semi-coherent stopping time, with maximal cube $Q(\sbf)$. For every $N_1,N_2 \in \mathbb{N} \cup \{0\}$, $N_1 \le N_2$, define the following truncation of $\sbf$ as
\begin{equation*}
\sbf_{N_1,N_2} := \{Q \in \sbf:   2^{-N_2}\ell(Q(\sbf)) \le \ell(Q) \le 2^{-N_1} \ell(Q(\sbf))\}
\end{equation*}
Then, for each $N$, there exists a cut off function $\eta = \eta_N$ with the following properties:
\begin{equation}\label{iscutoff.eq}
\mathbbm{1}_{\Omega_{\sbf_{N_1,N_2}}^{**}} \le \eta \le \mathbbm{1}_{\Omega_{\sbf_{N_1,N_2}}^{***}},
\end{equation}
\begin{equation}\label{cutoffbds1.eq}
|\nabla_X \eta(X,t)|\dist((X,t), \partial\Omega) + |\partial_t \eta(X,t)|\dist((X,t), \partial\Omega)^2 \lesssim 1,
\end{equation}
and 
if we define the ``boundary'' of the region where $\eta \equiv 1$ by
\[\mathcal{F}_b := \big\{Q \in \mathbb{D}: \exists \, \mbf{X} \in U_{Q}^{***} \text{ with } \nabla_{\mbf{X}} \eta (\mbf{X}) \neq 0\big\},\]
then the ``perimeter'' of $\{ \eta \equiv 1\}$ behaves well in the sense that 
\begin{equation}\label{cutoffbds2.eq}
\sum_{Q \in \mathcal{F}_b } |Q| \lesssim |Q(\sbf)|.
\end{equation}
In particular, it holds
\begin{equation}\label{cutoffbds3.eq}
\iiint_{\mathbb{R}^{n+1}}\big( |\nabla_X \eta(X,t)| + |\partial_t \eta(X,t)|\dist((X,t), \partial\Omega)\big) \, \d X \d t \lesssim |Q(\sbf)|.
\end{equation}
Here all the implicit constants depend only on the structural constants and $K$ (defining the Whitney regions in Subsection~\ref{whit.sec}). 
\end{lemma}

The following square function estimate is a crucial ingredient in our proof of Theorems \ref{main1.thrm} and \ref{main2.thrm}. Indeed, after proving Lemma \ref{mainsfest.lem}, we will no longer have to consider whether we are proving Theorem \ref{main1.thrm} or Theorem \ref{main2.thrm}.

We recall that $K$ is the large (see \eqref{eq:K_big}) parameter in the definitions of the various fattened Whitney regions $U_Q, U_Q^*$, etc. (see Subsection \ref{whit.sec}).
\begin{lemma}\label{mainsfest.lem}
Suppose that the hypotheses of Theorem \ref{main1.thrm} or Theorem \ref{main2.thrm} hold. If $K = 2^{N}$, with $N \ge 2000$  is sufficiently large and $N_1 := 10N$  then for all $N_2 \ge N_1$ it holds
\begin{equation}\label{mnsqfnest.eq}
\iiint_{\Omega_{\sbf^{*}_{N_1,N_2}}^{**}} \big(|\partial_t u(\mbf{X})|^2 + |\nabla^2_X u(\mbf{X})|^2\big) \dist(\mbf{X}, \partial \Omega) \, \d\mbf{X} \lesssim |Q(\sbf^*)|,
\end{equation} 
with implicit constants depending on the structural constants and $K$,
and independent of $N_2$.

Moreover, if $\, \sbf \subseteq \dd$ is any semi-coherent stopping time with $\sbf \subseteq \sbf^*$ and $\ell(Q(\sbf)) \le 2^{-10N}\ell(Q(\sbf^*))$, then for any $N_3 > 0$ we have
\begin{equation}\label{locmnsqfnest.eq}
\iiint_{\Omega_{\sbf_{0,N_3}}^{**}} \big(|\partial_t u(\mbf{X})|^2 + |\nabla^2_X u(\mbf{X})|^2\big) \dist(\mbf{X}, \partial \Omega) \, \d\mbf{X} \lesssim |Q(\sbf)|, 
\end{equation}
with implicit constants depending on the structural constants and $K$.
\end{lemma}

\begin{remark} \label{rem:poles_far_ibp} 
	For the proof of Lemma~\ref{mainsfest.lem}, it is crucial that $u$ is an adjoint solution wherever $\eta \neq 0$. And that is true (see the comment right after \eqref{defu.eq}) because our choice of pole $\mbf{X}^+_{Q(\sbf^*)}$ in \eqref{CSfortop.eq} was made with $M_\star$ so large that $\{\eta \neq 0\}$ lies way back in time from $t(\mbf{X}^+_{Q(\sbf^*)})$. Indeed, since in Lemma~\ref{mainsfest.lem} the cubes $Q$ that compose the sets $\mbf{S}^*_{N_1, N_2}$ and $\mbf{S}_{0, N_3}$ always satisfy $\ell(Q) \leq 2^{-10N}\ell(Q(\sbf^*)) = K^{-10} \ell(Q(\sbf^*))$, it is not difficult to conclude from \eqref{eq:U_Q_in_balls} that $\Omega_{\sbf^{*}_{N_1,N_2}}^{**}, \Omega_{\sbf_{0,N_3}}^{**} \subseteq Q_{K^{-7} M_\star \diam(Q(\sbf^*))}(\psi(\mbf{x}_Q), \mbf{x}_Q)$, which implies that $\mbf{X}^+_{Q(\sbf^*)}$ is indeed always far enough to the future. 
\end{remark}

The proof of Lemma \ref{mainsfest.lem} uses an integration by parts argument, which is different whether we assume the hypotheses of Theorem~\ref{main1.thrm} or Theorem~\ref{main2.thrm}. So let us separate them. Throughout the proof, we will frequently abbreviate $\nabla := \nabla_X$ and forget about the $\d\mbf{X}$, simply for brevity. 

\begin{proof}[Proof of Lemma \ref{mainsfest.lem} under the hypotheses of Theorem \ref{main1.thrm}]
Let us prove \eqref{mnsqfnest.eq} (indeed, \eqref{locmnsqfnest.eq} follows by picking the cut-off function in the next line adapted to $\sbf_{0, N_3}$ instead of $\sbf^*_{N_1,N_2}$). Let $\eta$ be the cut-off function for $\sbf^*_{N_1,N_2}$ provided by Lemma \ref{cutofffnlem.lem}. By Lemma \ref{cutofffnlem.lem} and \eqref{M2def.eq} it suffices to show 
\begin{equation}\label{mnsqfnest1.eq}
\iiint \big(|\partial_t u(\mbf{X})|^2 + |\nabla^2_X u(\mbf{X})|^2\big) \, u \eta \,  \d\mbf{X} \lesssim |Q(\sbf^*)|.
\end{equation}

Before continuing, we let $\tt E$ denote a generic term that enjoys the estimate $|{\tt E}| \lesssim |Q(\sbf^*)|$. In particular, we denote by ${\tt E}_i$ ($i = 1,2,\dots, 9$) any term with the following bounds
\begin{align*}
&|{\tt E}_1| \lesssim \iiint |\nabla u|^2 u \, |\partial_t \eta| \, \d\mbf{X}, \quad \qquad \,
|{\tt E}_2| \lesssim \iiint |\partial_t u| \, |\nabla u| \, u \, |\nabla \eta|  \, \d\mbf{X}, \quad \,\, 
|{\tt E}_3| \lesssim \iiint  |\nabla u|^3  |\nabla \eta| \, \d\mbf{X}, 
\\
&|{\tt E}_4| \lesssim \iiint u \, |\nabla A| \,  |\nabla u|^2 |\nabla \eta| \, \d\mbf{X}, \quad \;\,
|{\tt E}_5| \lesssim \iiint u  \, |\nabla u| \, |\nabla^2 u| \, |\nabla \eta| \, \d\mbf{X}, \quad 
|{\tt E}_6| \lesssim \iiint |\partial_t A| \, |\nabla u|^2 u \eta  \, \d\mbf{X}, 
\\
&|{\tt E}_7| \lesssim \iiint |\nabla A| \, |\nabla^2 u| \, |\nabla u| \, u \eta \, \d\mbf{X}, \quad \!
|{\tt E}_8| \lesssim \iiint |\nabla A|^2 |\nabla u|^2 \, u \eta \, \d\mbf{X}, \quad \;\;\;
|{\tt E}_9| \lesssim \iiint  |\nabla u|^3 |\nabla A| \, \eta \, \d\mbf{X}.
\end{align*}
Indeed, each term has the control $|{\tt E}_i| \lesssim |Q(\sbf^*)|$. 
To show that, first use point-wise bounds from Lemma~\ref{ptwiseestinsaw4der.lem}, \eqref{smoothnessonA.eq} and \eqref{iscutoff.eq}, and then finish by integrating with the nice properties of the derivatives of $\eta$ in \eqref{cutoffbds3.eq} for the terms $\tt E_1$--$\tt E_5$, and the Carleson bounds for $A$ in \eqref{carlesonmeasure.eq} for the terms $\tt E_6$--$\tt E_9$ (to use this, note that the region $\Omega_{\sbf^*}^{***}$ can be contained in a cube of the form $\mathcal{J}_{C\ell(Q(\sbf^*))}(\mbf{X}_{\mbf{S}^*})$, where $C = C(K)$ is a large constant, and $\mbf{X}_{\mbf{S}^*}$ is any point in $U_{Q(\sbf^*)}^{**}$).

We are going to modify the argument in \cite{LN07} to our variable coefficient setting. In particular we introduce two quantities (where subscripts of $u$ mean derivatives) 
\[\alpha := \iiint u \, (a_{i,j} u_{x_i,x_k})\, (a_{k,\ell}u_{x_j, x_\ell})\, \eta, 
\qquad \beta := \iiint u \, (u_t)^2 \eta.\]
Here we have adopted Einstein summation notation for convenience.
The ellipticity of $A$ yields (this requires some elementary linear algebra work, which we defer to Appendix~\ref{appendix:matrices})
\begin{equation} \label{hessianAndMixedIndices.eq}
	\alpha \approx  \iiint u \, |\nabla^2 u|^2 \eta.
\end{equation}
Thus, since $\alpha, \beta \geq 0$, to prove \eqref{mnsqfnest1.eq}, it is enough to show that
\begin{equation} \label{alpha2beta.eq}
	\alpha + 2\beta = {\tt E},
\end{equation}
where $\tt E$ satisfies the estimate $|{\tt E}| \lesssim |Q(\sbf^*)|$. 

Let us begin by treating $\beta$. 
First using the equation $\mathcal{L}^* u = 0$, then integrating by parts, we obtain
\[\beta = \iiint -(\dv A \nabla u) \, u_t u \eta 
=\iiint A \nabla u \cdot \nabla (u_t u \eta) 
= {\tt E}   + \iiint  A \nabla u \cdot \nabla u_t \, u \eta 
+ \iiint  A \nabla u \cdot \nabla u \, u_t   \eta,\]
where, after using the product rule, the error term is of type ${\tt E}_2$. Next, we will use this fact:
\begin{equation} \label{symcalc.eq}
	(1/2) \, \partial_t (A\nabla u \cdot \nabla u) = A\nabla u \cdot \nabla u_t + (1/2) (\partial_t A)\nabla u \cdot \nabla u.
\end{equation}
To verify it, just expand and remember that \textbf{$A$ is symmetric}. Using this identity to continue rewriting $\beta$, it turns out that, up to a new error term of type ${\tt E}_6$, we have
\[\beta = E  +\iiint (1/2) \, \partial_t (A\nabla u \cdot \nabla u) \, u\eta +  \iiint  A \nabla u  \cdot \nabla u \, u_t \eta.\]
Finally, integrating by parts in $t$ on the first term we have, up to a new ``error term" of type ${\tt E}_1$,
\begin{equation} \label{betasimplified.eq}
	\beta 
	=
	{\tt E} - \frac12 \iiint  A \nabla u \cdot \nabla u \, u_t \eta 
	+  \iiint  A \nabla u  \cdot \nabla u \, u_t \eta
	= {\tt E} + \frac12 \iiint  A \nabla u \cdot \nabla u \, u_t \eta.
\end{equation}

Next, our goal is to show that $\alpha$ is nearly of the same form as $\beta$, to get cancellations. First, using the product rule both for $\partial_{x_k} (a_{i, j}u_{x_i})$ and $\partial_{x_j}(a_{k, \ell} u_{x_\ell})$ simultaneously, we obtain
\[
\alpha 
=
\iiint u \, \partial_{x_k}(a_{i,j} u_{x_i})\, \partial_{x_j}(a_{k,\ell}u_{x_\ell})\, \eta
- \tt E_8 - \tt E_7 - \tt E_7,
\]
and we will abbreviate the errors by $\tt E$. Next, we integrate by parts in $x_k$ and use the product rule to obtain, up to a new error term of type $ - \tt E_4 - \tt E_5$ (when $\partial_{x_k}$ lands on $\eta$),
\[\alpha = E -  \iiint u_{x_k} (a_{i,j} u_{x_i})\, \partial_{x_j}(a_{k,\ell}u_{x_\ell})\, \eta 
-  \iiint u \, (a_{i,j} u_{x_i})\, \partial_{x_k} \partial_{x_j}(a_{k,\ell}u_{x_\ell})\, \eta =: {\tt E} + I + II.\]

We estimate $I$ and $II$ separately, starting with $II$. Using the equation $\mathcal{L}^* u = 0$ (and exchanging the order of derivatives when needed) and rewriting the resulting expression using that $A$ is symmetric, then invoking \eqref{symcalc.eq} (and hence introducing an error term of type $\tt E_6$), and lastly integrating by parts in $t$ (introducing an error of type $\tt E_1$), we have
\begin{multline*}
	II =  -  \iiint u \, (a_{i,j} u_{x_i})\, \partial_{x_j} \partial_{x_k}(a_{k,\ell}u_{x_\ell}) \, \eta 
	= \iiint u \, (a_{i,j} u_{x_i})  \, \partial_{x_j} u_t \, \eta
	=
	\iiint u \, A\nabla u \cdot \nabla u_t \, \eta 
	\\ = {\tt E} + \frac12 \iiint  u \, \partial_t (A\nabla u \cdot \nabla u)\, \eta 
	= {\tt E} - \frac12 \iiint  A \nabla u \cdot \nabla u \, u_t \eta.
\end{multline*}
We have a desirable estimate for $II$ and now turn to $I$. We start by integrating by parts in $x_j$, to obtain, up to a new error term of type ${\tt E}_3$,
\[I = {\tt E} + \iiint  u_{x_k,x_j} (a_{i,j} u_{x_i})(a_{k,\ell}u_{x_\ell})\, \eta + \iiint u_{x_k} \partial_{x_j} (a_{i,j} u_{x_i})\, a_{k,\ell}u_{x_\ell}\, \eta =: {\tt E} + I_1 + I_2,\]
By the equation $\mathcal{L}^* u = 0$ and symmetry of $A$, we simply compute
\[I_2 = - \iiint  A \nabla u \cdot  \nabla u \, u_t \eta.\]
To handle $I_1$ and get some cancellation with $I_2$ in this form, we use a similar identity to \eqref{symcalc.eq},
\[(1/2) \, \partial_{x_k} (A\nabla u \cdot \nabla u) = A\nabla u \cdot \nabla u_{x_k} + (1/2) \, (\partial_{x_k} A)\nabla u \cdot \nabla u.\]
Now, rewriting $I_1$, then using this identity (introducing an error term of type $\tt E_9$), next integrating by parts in $\partial_{x_k}$ (introducing a new error term of type $\tt E_3$), and lastly using the equation $\mathcal{L}^* u = 0$, we obtain
\begin{multline*}
	I_1 
	= 	\iiint A \nabla u_{x_k} \cdot \nabla u \, (a_{k,\ell}u_{x_\ell}) \, \eta 
	=  {\tt E} + \frac12 \iiint \partial_{x_k}(A \nabla u \cdot \nabla u)\, (a_{k,\ell}u_{x_\ell}) \, \eta
	\\ 
	= {\tt E} - \frac12 \iiint A \nabla u \cdot \nabla u \, \partial_{x_k}(a_{k,\ell} u_{x_\ell}) \, \eta  = {\tt E} + \frac12  \iiint  A \nabla u \cdot \nabla u \, u_t \eta.
\end{multline*}
Joining all the previous computations, we obtain 
\[I = {\tt E} - \frac12 \iiint  A \nabla u \cdot \nabla u \, u_t \eta.\]
and combining this with our estimate for $II$ before, we end up with
\[\alpha = {\tt E}   - \iiint  A \nabla u  \cdot \nabla u \, u_t \eta, \]
which together with \eqref{betasimplified.eq} yield the cancellation needed to obtain \eqref{alpha2beta.eq} and finish the proof.
\end{proof}

\begin{proof}[Proof of Lemma \ref{mainsfest.lem} under the hypotheses of Theorem \ref{main2.thrm}] 
In this proof we are going to follow more closely the ideas in \cite[Proposition 4.36]{HMT}. Again, \eqref{locmnsqfnest.eq} and \eqref{mnsqfnest.eq} only differ in which cut-off function we use in Lemma \ref{cutofffnlem.lem}, so we only prove \eqref{mnsqfnest.eq}. Let $\eta$ be the cut-off function for $\sbf^*_{N_1,N_2}$ introduced in Lemma \ref{cutofffnlem.lem}.

First, let us show that only the spatial derivatives are relevant. Indeed, using the equation $\mathcal{L}^*u = 0$ (which we can use in a pointwise sense because the pole of $u$ is far from the region we are working in, see Remark~\ref{rem:poles_far_ibp}) and the product rule (along with the boundedness of $A$), we obtain
\begin{multline*}
	\iiint_{\Omega_{\sbf^{*}_{N_1,N_2}}^{**}} |\partial_t u(\mbf{X})|^2 \dist(\mbf{X}, \partial \Omega) \, \d\mbf{X}
	=
	\iiint_{\Omega_{\sbf^{*}_{N_1,N_2}}^{**}} |\div(A^* \nabla u)|^2 \dist(\mbf{X}, \partial \Omega) \, \d\mbf{X}
	\\ \lesssim
	\iiint_{\Omega_{\sbf^{*}_{N_1,N_2}}^{**}} |\nabla A|^2 |\nabla u|^2 \dist(\mbf{X}, \partial \Omega) \, \d\mbf{X}
	+ \iiint_{\Omega_{\sbf^{*}_{N_1,N_2}}^{**}} |\nabla^2 u|^2 \dist(\mbf{X}, \partial \Omega) \, \d\mbf{X},
\end{multline*}
and the first term is $\lesssim |Q(\sbf^*)|$ by \eqref{ptwiseestinsaw4der.lem} and our hypothesis on $A$, namely \eqref{L2Carleson.eq}. Therefore, it suffices to deal with the spatial derivatives.  
Moreover, by the properties of $\eta$ (concretely \eqref{iscutoff.eq}) and \eqref{M2defb.eq} (recalling that $N_1 = 10 N$ to use that estimate) it suffices to show 
\begin{equation*} 
\iiint  |\nabla^2 u(\mbf{X})|^2 \, v(\mbf{X}) \, \eta \, \d\mbf{X} \lesssim |Q(\sbf^*)|.
\end{equation*}
In fact, denoting by $\partial u := \partial_{x_k} u$ an arbitrary spatial derivative of $u$, this reduces to
\begin{equation}\label{mnsqfnestns2.eq}
\iiint  |\nabla \partial u(\mbf{X})|^2 \, v(\mbf{X}) \, \eta \, \d\mbf{X} \lesssim |Q(\sbf^*)|.
\end{equation}
Using the ellipticity of $A$ and the product rule, we first obtain
\begin{multline} \label{eq:split_ibp_ns}
\iiint  |\nabla \partial u|^2\, v \eta \approx \iiint  A^*\nabla( \partial u) \cdot  \nabla (\partial u) \, v\eta
\\  = \iiint  A^* \nabla (\partial u) \cdot  \nabla(\partial u v\eta)- \frac12 \iiint  A^*\nabla ((\partial u)^2) \cdot \nabla(v\eta) =: I + II.
\end{multline}
Below we let ${\tt E}$ be a generic error term (i.e. $|{\tt E}| \lesssim |Q(\sbf^*)|$), as was in the proof of Lemma \ref{mainsfest.lem} under the hypotheses of Theorem \ref{main1.thrm}. We note that the terms ${\tt E}_i$ ($i  = 1,2,\dots 9$) are still under control if we replace any occurrence of $u$ (including its derivatives) with $v$, as these terms are controlled using the Carleson/size conditions on $A$ (Definition~\ref{smoothL1carl.def}), properties of $\eta$ (Lemma~\ref{cutofffnlem.lem}), and Lemma~\ref{ptwiseestinsaw4der.lem}.

Now we continue by handling term $I$. First, we compute
\[
I = \iiint \partial( A^*(\nabla u)) \cdot  \nabla (\partial u  v\eta) - \iiint \partial( A^*)(\nabla u) \cdot  \nabla(\partial u v\eta)
= - \iiint  A^*\nabla u \cdot  \nabla\partial(\partial u v\eta) + {\tt E},
\]
where in the first equality we have used the product rule; and in the second equality, we have integrated the first term by parts, and we have expanded the gradient in the second term to identify an error term of type $\tt E_4 + \tt E_7 + \tt E_9$. Then, using the equation $\mathcal{L}^* u = 0$ and integrating by parts again:
\[
I 
=
- \iiint \partial_t u \, \partial(\partial u v\eta) + {\tt E}
=
\iiint \partial_t (\partial u) \, \partial u \, v\eta + {\tt E}.
\]
Note the difference with respect to the computation in \cite{HMT}: we get an extra term, apart from the error term. Our goal now is to cancel it with $II$.

Now we treat $II$. Start by using the duality between $A$ and $A^*$, and then the product rule to get
\begin{multline*}
II
=
- \frac12 \iiint A \nabla (v \eta) \cdot \nabla ((\partial u)^2)
=
- \frac12 \iiint A \nabla v \cdot \nabla ((\partial u)^2 \eta) 
+ \frac12 \iiint A \nabla v \cdot \nabla v (\partial u)^2
\\ - \frac12 \iiint A \nabla \eta \cdot \nabla ((\partial u)^2) \, v
=
\frac12 \iiint \partial_t v \, (\partial u)^2 \eta
+ \tt E,
\end{multline*}
where in the last equality we have used the equation $\mathcal{L} v = 0$ in the first term, and we have noticed that the last two terms are $\tt E_3 + \tt E_5$. Then, integrate by parts in $t$ and use the product rule 
\[
II
=
- \frac12 \iiint v \, \partial_t ((\partial u)^2 \eta)
+ {\tt E}
=
\iiint v \, \partial_t (\partial u) \, \partial u \, \eta
- \frac12 \iiint v \, \partial_t v \, (\partial u)^2
+ {\tt E}, 
\]
where we notice the last term is an error of type $\tt E_1$. Therefore, putting $I$ and $II$ back in \eqref{eq:split_ibp_ns}, we obtain the desired cancellation which yields 
\[I + II = {\tt E},\]
which shows \eqref{mnsqfnestns2.eq} and finishes the proof.
\end{proof}

In the sequel, we make no distinction between whether we assume Theorem \ref{main1.thrm} or Theorem \ref{main2.thrm}.

Next, let us obtain a consequence of Lemma~\ref{ptwiseestinsaw4der.lem}, where the square function estimates take exactly the form that we will need later in Subsection~\ref{subsec:sqfunest}.

\begin{lemma}\label{impsfnest.lem} Suppose that the hypotheses of Theorem \ref{main1.thrm} or Theorem \ref{main2.thrm} hold. Let $K = 2^{N}$, with $N \ge 2000$ being sufficiently large, and $N_1 := 10N$. Then, for all $N_2 \ge N_1$, it holds
\begin{multline}\label{redderivtxest.eq}
\iiint_{\Omega_{\sbf^*_{N_1,N_2}}^*}
\Big( |\dist(\cdot, \partial \Omega)\,\partial_t u|^2
+
|\dist(\cdot, \partial \Omega)\,\nabla_X^2 u|^2
\Big)
 \, \frac{\d\mbf{X}}{\dist(\cdot, \partial \Omega)(\mbf{X})}
\\ + 
\iiint_{\Omega_{\sbf^*_{N_1,N_2}}^*}
\Big(
|\dist(\cdot, \partial \Omega)^2\,\partial_t u\, \nabla_X^2u|^2
+
|\dist(\cdot, \partial \Omega)^2\,\nabla_X\partial_t u|^2
\Big)
 \, \frac{\d\mbf{X}}{\dist(\cdot, \partial \Omega)(\mbf{X})}
\, \lesssim \,|Q(\sbf^*)|,
\end{multline}
with implicit constants depending on the structural constants and $K$,
and independent of $N_2$.
 
Moreover, if $\sbf \subseteq \sbf^*$ is semi-coherent and $\ell(Q(\sbf)) \le 2^{-10N}\ell(Q(\sbf^*))$, then for any $N_3 > 0$
\begin{multline}\label{redderivtxestsub.eq}
\iiint_{\Omega_{\sbf_{0, N_3}}^*}
\Big( |\dist(\cdot, \partial \Omega)\,\partial_t u|^2
+
|\dist(\cdot, \partial \Omega)\,\nabla_X^2 u|^2
\Big)
\, \frac{\d\mbf{X}}{\dist(\cdot, \partial \Omega)(\mbf{X})}
\\ + 
\iiint_{\Omega_{\sbf_{0, N_3}}^*}
\Big(
|\dist(\cdot, \partial \Omega)^2\,\partial_t u\, \nabla_X^2u|^2
+
|\dist(\cdot, \partial \Omega)^2\,\nabla_X\partial_t u|^2
\Big)
\, \frac{\d\mbf{X}}{\dist(\cdot, \partial \Omega)(\mbf{X})}
\, \lesssim \,|Q(\sbf)|,
\end{multline}
with implicit constants depending on the structural constants and $K$,
and independent of $N_3$.
\end{lemma}
\begin{proof}
	Let us only show \eqref{redderivtxest.eq}, because \eqref{redderivtxestsub.eq} is analogous in view of Lemma~\ref{mainsfest.lem}. In turn, the first two terms in \eqref{redderivtxest.eq} are exactly those in Lemma~\ref{mainsfest.lem}. For the third one, we use the pointwise interior estimates from Lemma~\ref{ptwiseestinsaw4der.lem} (along with Lemma~\ref{mainsfest.lem})
	\[
	\iiint_{\Omega_{\sbf^*_{N_1,N_2}}^*}
	|\dist(\cdot, \partial \Omega)^2\,\partial_t u\, \nabla_X^2u|^2
	\, \frac{\d\mbf{X}}{\dist(\cdot, \partial \Omega)(\mbf{X})}
	\lesssim 
	\iiint_{\Omega_{\sbf^*_{N_1,N_2}}^*}
	|\nabla_X^2u|^2 \dist(\cdot, \partial \Omega)\, \d\mbf{X}
	\lesssim 
	|Q(\sbf^*)|.
	\]
	
	The fourth one is more elaborate, and will follow from a Caccioppoli-type inequality for $\partial_t u$. However, since our PDE have variable coefficients, $\partial_t u$ does not solve a homogeneous PDE, and we will need to use our hypothesis on $A$ (see Definition~\ref{smoothL1carl.def}) to finish the argument. Let us include the details. 	
	Using a (parabolic) Whitney decomposition
(with cubes as in \eqref{cuba}), we estimate 
	\begin{equation} \label{splitWhitney.eq}
		\iiint_{\Omega^*_{\sbf^*_{N_1, N_2}}} |\dist(\mbf{X}, \pom)^2 \nabla_X \partial_t u(\mbf{X})|^2 \frac{\d \mbf{X}}{\dist(\mbf{X}, \pom)}
		\lesssim 
		\sum_{\mathcal{J} : \mathcal{J} \cap \Omega^*_{\sbf^*_{N_1, N_2}} \neq \emptyset} \ell(\mathcal{J})^3 \iiint_{\mathcal{J}} |\nabla_X \partial_t u|^2.
	\end{equation}	
	
	Now write $v:=\partial_t u$, and fix one such Whitney cube $\mathcal{J}$. We proceed to prove a Caccioppoli inequality. Find a cut-off $\phi = \phi_{\mathcal{J}} \in C^\infty_c(\RR^n)$ satisfying $\mathbbm{1}_{\mathcal{J}} \leq \phi \leq \mathbbm{1}_{2\mathcal{J}}$ and $\ell(\mathcal{J}) |\nabla_X \phi| + \ell(\mathcal{J})^2 |\partial_t \phi| \lesssim 1$. Then compute
	\begin{equation*}
		\iiint_{2\mathcal{J}} |\nabla_X v|^2 \phi^2
		\lesssim 
		\iiint_{2\mathcal{J}} A \nabla_X v \cdot \nabla_X v \phi^2
		= 
		\iiint_{2\mathcal{J}} A \nabla_X v \cdot \nabla_X (v\phi^2)
		- \iiint_{2\mathcal{J}} A \nabla_X v \cdot \nabla_X (\phi^2) v
		=: 
		I + II. 
	\end{equation*}
	To deal with $II$, use Young's inequality with $\varepsilon_1 > 0$
	\begin{equation*}
		|II| 
		\lesssim
		\eps_1 \iiint_{2\mathcal{J}} |\nabla_X v|^2 \phi^2
		+ \eps_1^{-1} \iiint_{2\mathcal{J}} v^2 |\nabla_X \phi|^2
		\lesssim 
		\eps \iiint_{2\mathcal{J}} |\nabla_X v|^2 \phi^2
		+ \eps^{-1} \ell(\mathcal{J})^{-2} \iiint_{2\mathcal{J}} v^2,
	\end{equation*}
	and for $\varepsilon_1$ small enough we will be able to later hide the first term. For the term $I$, note that differentiating the equation that $u$ solves, we obtain that $v$ solves the (inhomogeneous) equation $\partial_t v - \div(A\nabla_X v) = \div((\partial_t A)\nabla_X u)$. Therefore, interpreting that equation in the weak sense, we get
	\begin{equation*}
		I 
		= 
		- \iiint_{2\mathcal{J}} v \partial_t (v \phi^2)
		+ \iiint_{2\mathcal{J}} (\partial_t A) \nabla_X u \cdot \nabla_X (v\phi^2)
		=:
		I_1 + I_2. 
	\end{equation*}
	The term $I_1$ is easy to handle. Using the product rule and an integration by parts, we obtain
	\begin{equation*}
		|I_1|
		\leq
		\big| \frac12 \iiint_{2\mathcal{J}} \partial_t (v^2) \phi^2 \big|
		+ \big|\iiint_{2\mathcal{J}} v^2 \partial_t (\phi^2) \big|
		=
		\frac32 \, \big|\iiint_{2\mathcal{J}} v^2 \partial_t (\phi^2)\big|
		\lesssim 
		\ell(\mathcal{J})^{-2} \iiint_{2\mathcal{J}} v^2,
	\end{equation*}
	In turn, for $I_2$, using the product rule and Young's inequality with $\varepsilon_2 > 0$, we get
	\begin{multline*}
		I_2
		=
		\iiint_{2\mathcal{J}} (\partial_t A) \nabla_X u \cdot \nabla_X v\phi^2
		+ 2\iiint_{2\mathcal{J}} (\partial_t A) \nabla_X u \cdot \nabla_X \phi v \phi
		\\ \lesssim
		\eps_2 \iiint_{2\mathcal{J}} |\nabla_X v|^2 \phi^2
		+ \eps_2^{-1} \iiint_{2\mathcal{J}} |\partial_t A|^2 |\nabla_X u|^2 \phi^2
		+ \iiint_{2\mathcal{J}} v^2 |\nabla_X \phi|^2
		+ \iiint_{2\mathcal{J}} |\partial_t A|^2  |\nabla_X u|^2 \phi^2
		\\ \leq 
		\eps_2 \iiint_{2\mathcal{J}} |\nabla_X v|^2 \phi^2
		+ \ell(\mathcal{J})^{-2} \iiint_{2\mathcal{J}} v^2 
		+ C(\eps_2) \iiint_{2\mathcal{J}} |\partial_t A|^2,
	\end{multline*}
	where in the last estimate we have used the estimates for $u$ from Lemma~\ref{ptwiseestinsaw4der.lem} (because points in $\mathcal{J}$ fall into the assumptions of this lemma because $\mathcal{J} \cap \Omega^*_{\sbf^*_{N_1, N_2}} \neq \emptyset$ with $N_1$ large enough). Then, for $\eps_2$ small enough, we will hide the first term.
	
	Putting all the above estimates together back in \eqref{splitWhitney.eq} (and hiding the appropriate terms as mentioned before), and recalling that $\phi_{\mathcal{J}} \equiv 1$ in $\mathcal{J}$ and that $v = \partial_t u$, we obtain 
	\begin{multline*}
		\iiint_{\Omega^*_{\sbf^*_{N_1, N_2}}} |\dist(\mbf{X}, \pom)^2 \nabla_X \partial_t u(\mbf{X})|^2 \frac{\d \mbf{X}}{\dist(\mbf{X}, \pom)}
		\lesssim 
		\sum_{\mathcal{J} : \mathcal{J} \cap \Omega^*_{\sbf^*_{N_1, N_2}} \neq \emptyset} \ell(\mathcal{J})^3 \iiint_{2\mathcal{J}} |\nabla_X \partial_t u|^2 \phi_{\mathcal{J}}^2
		\\ \lesssim 
		\sum_{\mathcal{J} : \mathcal{J} \cap \Omega^*_{\sbf^*_{N_1, N_2}} \neq \emptyset} \ell(\mathcal{J})^3 \left( \ell(\mathcal{J})^{-2} \iiint_{2\mathcal{J}} |\partial_t u|^2  + \iiint_{2\mathcal{J}} |\partial_t A|^2 \right)
		\\ \lesssim 
		\iiint_{\Omega^{**}_{\sbf^*_{N_1, N_2}}} |\partial_t u|^2 \dist(\cdot, \pom) \d\mbf{X}
		+ \iiint_{\Omega^{**}_{\sbf^*_{N_1, N_2}}} |\partial_t A|^2 \dist(\cdot, \pom)^3 \d\mbf{X}
		\lesssim 
		|Q(\sbf^*)|,
	\end{multline*}
where in the next-to-last step, we used that, for
a sufficiently large choice of the parameter $K$, all the doubled Whitney boxes
$2\mathcal{J}$, appearing in the sum, are 
contained in $\Omega^{**}_{\sbf^*_{N_1, N_2}}$,
 and then in the last step we used
 Lemma~\ref{mainsfest.lem} and our hypothesis on $A$ (namely \eqref{L2Carleson.eq}). 
\end{proof}

\subsection{\for{toc}{\small}Refinement of the corona decomposition to obtain vertical growth of $u$} \label{sec:second_corona}

We also deduce the following from Lemma \ref{mainsfest.lem}.
\begin{lemma}\label{fkwhsa.lem}
Let $\sbf^*$ from Lemma~\ref{pmcorona.eq} (or \ref{pmcorona2.eq}). Suppose $K = 2^N$, with $N \ge 2000$ sufficiently large. Then there exists a constant $C'' \geq 1$ (depending on $K$)
such the following holds.
The stopping time regime $\sbf^*$ can be decomposed as $\sbf^* = \mathcal{G}_{\sbf^*} \sqcup \mathcal{B}_{\sbf^*}$, where the ``bad cubes'' $\mathcal{B}_{\sbf^*}$ 
contain all the cubes $Q \in \sbf^*$ with $\ell(Q) \geq 2^{-10N} \ell(Q(\sbf^*))$ and moreover
satisfy a packing condition
\begin{equation}\label{degpack.eq}
\sum_{Q \in \mathcal{B}_{\sbf^*}, Q \subseteq P} |Q| \le C'' |P|, \quad \forall P \in \dd
\end{equation}
and for the ``good cubes'' $Q \in \mathcal{G}_{\sbf^*}$ it holds (for $c_1$ from \eqref{nondegpt.eq})
\begin{equation}\label{nondegwhluq.eq}
\partial_{y_0} u(\mbf{Y}) \geq c_1/8, \quad \quad \forall \mbf{Y} \in U_Q.
\end{equation}
\end{lemma}
\begin{proof}
Let $\eps > 0$ be small to be chosen, possibly depending on the structural constants and $K$.
We define $\mathcal{B}_{\sbf^*}$ as the collection of cubes $Q \in \sbf^*$ that are either ``large''
\begin{equation}\label{toobig4est.eq}
\ell(Q) \ge 2^{-10N}\ell(Q(\sbf^*)) =  K^{-10}\ell(Q(\sbf^*))
\end{equation}
or the Green's function $u$ has some (undesired) oscillation at their scale, in the form
\begin{equation}\label{badosc.eq}
\iiint_{U_Q^*} (|\nabla^2 u|^2 + |\partial_t u|^2) \dist(\mbf{X}, \partial\Omega)\, \d\mbf{X} \ge \eps |Q|.
\end{equation}
The cubes in $\mathcal{B}_{\sbf^*}$ satisfying \eqref{toobig4est.eq} pack trivially. Then, if we call $\mathcal{B}'_{\sbf^*}$ the cubes for which \eqref{badosc.eq} holds, we may use the bounded overlap of the $U_Q^*$ regions and \eqref{locmnsqfnest.eq} to obtain 
\begin{equation} \label{packingsqfun.eq}
	\sum_{Q \in \mathcal{B}'_{\sbf^*}, Q \subseteq P} |Q| \le \eps^{-1} \sum_{Q \in \mathcal{B}'_{\sbf^*}, Q \subseteq P} \iiint_{U_Q^*} (|\nabla^2 u|^2 + |\partial_t u|^2) \dist(\mbf{X}, \partial\Omega) \d \mbf{X} \lesssim \eps^{-1} |P|, \quad \forall P \in \mathbb{D},
\end{equation}
Indeed, to use \eqref{locmnsqfnest.eq}, we may basically assume that $P \in \sbf^*$ with $\ell(P) \le 2^{-10N}\ell(Q(\sbf^*))$; for if this were not the case (let us quickly sketch the standard argument), then either 
$\{Q \in \mathcal{B}'_{\sbf*}, Q \subseteq P\}= \emptyset$,
or else we could apply the estimate to the collection $\{P'\}$ of maximal cubes in
$\{Q \in \mathcal{B}'_{\sbf^*}, Q \subseteq P\}$
and sum over $P'$, which are disjoint subcubes of $P$.

In any case, \eqref{packingsqfun.eq} above shows the packing of $\mathcal{B}_{\sbf^*}$ with constant depending on $\eps$. Therefore, we need to show that \eqref{nondegwhluq.eq} holds for $Q \in \sbf^* \setminus \mathcal{B}_{\sbf^*} =: \mathcal{G}_{\sbf^*}$, provided $\eps$ is chosen sufficiently small. 

To this end, for the sake of a contradiction, suppose there is $\mbf{X} = (x_0, \mbf{x}) \in U_Q$ for $Q \in \mathcal{G}_{\sbf^*}$ with
\begin{equation*}
\partial_{x_0} u(\mbf{X}) \le c_1/8.
\end{equation*}
Then, by construction of $U_Q$ (see \eqref{whitdef.eq}), it holds $\mbf{x} \in 10^5Q$. Therefore, Lemma \ref{fndndpt.lem} gives the existence of some point $\mbf{Y} = (y_0,\mbf{x}) \in U_Q$ such that 
\[\partial_{x_0} u(\mbf{Y}) \ge c_1.\]
Now, by the estimates for the second derivatives of $u$ shown in Lemma~\ref{ptwiseestinsaw4der.lem} (note that $\ell(Q) \leq K^{-10}\ell(Q(\sbf^*))$ since $Q \in \mathcal{G}_{\sbf^*}$), there exists a small enough $c = c(K) > 0$ so that
\begin{equation}\label{degnearby.eq}
\partial_{x_0} u(\mbf{X}') \le c_1/4, \quad \forall \mbf{X}' \in \mathcal{J}_{c\ell(Q)}(\mbf{X}), 
\quad \text{ and } \quad 
\mathcal{J}_{c\ell(Q)}(\mbf{X})\subseteq U_Q^*,
\end{equation}
and similarly
\begin{equation}\label{nondegnearby.eq}
\partial_{x_0} u(\mbf{Y}') \ge c_1/2, \quad \forall \mbf{Y}' \in \mathcal{J}_{c\ell(Q)}(\mbf{Y}), 
\quad \text{ and } \quad \mathcal{J}_{c\ell(Q)}(\mbf{Y}) \subseteq U_Q^*.
\end{equation}
Now, recalling that $Q_{c\ell(Q)}(\mbf{x})$ is the projection of $\mathcal{J}_{c\ell(Q)}$ into space-time $\mathbb{R}^n$, define $\theta_0 := y_0 - x_0$,
so that  the fundamental theorem of calculus and the estimates \eqref{degnearby.eq} and \eqref{nondegnearby.eq} yield
\begin{equation*}
c_1/4 
\le \frac{1}{|Q_{c\ell(Q)}(\mbf{x})|} \iint_{Q_{c\ell(Q)}(\mbf{x})} \int_0^{\theta_0} \partial_{z_0}^2 u(x_0 + z_0,\mbf{z}) \, \d z_0 \, \d\mbf{z} 
\lesssim_K |Q|^{-1} \iiint_{U_Q^*}  |\nabla^2 u (\mbf{Z})| \, \d\mbf{Z} 
\lesssim_K \eps^{1/2},
\end{equation*}
where in the last estimate we have used Cauchy-Schwarz followed by the opposite inequality to \eqref{badosc.eq}, because this time $Q \in \mathcal{G}_{\sbf^*}$ (and also the easy fact that $\dist(\mbf{Z}, \pom) \approx_K \ell(Q)$ for $\mbf{Z} \in U_Q^*$, coming from Lemma~\ref{distgraphlem.lem}).
Therefore, for $\eps = \eps(c_1, K)$ sufficiently small, we get the contradiction that finishes the proof of the lemma. 
\end{proof}

Now let us break up $\mathcal{G}_{\sbf^*}$ into stopping time regimes, so that we obtain a coronization.
\begin{lemma}\label{graphregimes.lem}
In the conditions of Lemma \ref{fkwhsa.lem}, we can write
\[\mathcal{G}_{\sbf^*} = \sqcup_j \; \sbf_j,\]
where each $\sbf_j$ is a semi-coherent stopping time regime 
and the maximal cubes $\{Q(\sbf_j)\}$ pack, i.e.,
\[\sum_{Q(\sbf_j): Q(\sbf_j) \subseteq P} |Q(\sbf_j)| \lesssim |P|, \quad \forall P \in \dd,\]
with implicit constants depending on the structural constants and $K$.
\end{lemma}
\begin{proof} 
This proof is standard, too. Let us sketch it quickly. Starting at $Q(\sbf^*)$, we call $\mathcal{F} = \{Q_j\}_j$ the family of maximal cubes (with respect to containment) in $\mathcal{B}_{\sbf^*} \cap \dd(Q(\sbf^*))$. Then, if we call the first stopping-time regime $\sbf := \dd_{\mathcal{F}, Q(\sbf^*)}$ as in \eqref{saw4pm.eq}, it is easy to see that $\sbf$ is semi-coherent (follows directly from maximality of $\mathcal{F}$). After this first step, we iterate the same stopping-time procedure below each $Q_j$, so that every time we find a new stopping-time regime $\sbf$, its top cube $Q(\sbf)$ is right below some cube in $\mathcal{B}_{\sbf^*}$ --with the exception of the first level $Q(\sbf^*)$--. In this way, after iterating the procedure and exhausting $\sbf^*$, the top cubes $Q(\sbf)$ pack because their measure is controlled by that of their parents, which are $\mathcal{B}_{\sbf^*}$ cubes which already pack by \eqref{degpack.eq}. 
\end{proof}

\subsection{\for{toc}{\small}Reduction to finding regular Lip(1,1/2) approximating graphs in each stopping-time} \label{sec:reduction_corona}

To conclude this section, we set the stage for the remainder of the proof of Theorems \ref{main1.thrm} and \ref{main2.thrm}. In particular, Steps 3 and 4 are aimed at showing the following (concretely, it will follow from Lemmas~\ref{psiSclose.lem} and \ref{psiSlip.lem}, and Proposition~\ref{regular.prop}).
\begin{proposition}\label{bigprop.prop}
Suppose that $K = 2^N$ with $N \ge 2000$ sufficiently large. Then there exist constants $b_1,b_2$ and $C_1$ (depending on $K$) 
such that the following holds.
For each $\sbf$ produced in Lemma \ref{graphregimes.lem}, there exists a regular $\Lip(1,1/2)$ function $\psi_{\sbf}$, with graph $\Sigma_\sbf$, such that 
\begin{equation} \label{eq:approx_graph2}
	\sup_{\xbf \in \mbf{\Psi}(2Q)} \dist(\xbf, \Sigma_{\mbf{S}})
	\leq 
	C_1 \diam(Q), \quad \forall Q \in \sbf.
\end{equation}
\end{proposition}

We claim that to prove Theorem \ref{main1.thrm} and \ref{main2.thrm} it suffices to prove Proposition \ref{bigprop.prop}. Indeed, by Lemma \ref{coronaenough.lem}, we just need to see that we can take $\mathcal{B} := \cup_{\sbf^*} \mathcal{B}_{\sbf^*}$ and $\mathcal{G} :=  \cup_{\sbf^*} \mathcal{G}_{\sbf^*}$ and then use Lemma \ref{graphregimes.lem} in concert with either Lemma \ref{gfcomplemt1.eq} or Lemma \ref{gfcomplemt2.eq} to pack the maximal cubes in the stopping time (sub-)regimes $\{\sbf\}$.

\section{Step 3: Construction of approximating graphs via a stopping-time distance, and modified level sets} \label{construction.sec}

Having reduced matters to showing Proposition \ref{bigprop.prop}, we now fix $\sbf$ a ``graph stopping time regime" from Lemma~\ref{graphregimes.lem}, aiming to produce the regular $\Lip(1,1/2)$ function  $\psi_{\sbf}$. Note that we have still left $K = 2^N$ with $N \ge 2000$ as a parameter at our disposal (which will be fixed shortly):
recall (see Lemma~\ref{fkwhsa.lem}) that $\sbf$ contains only cubes $Q$ with $\ell(Q) \leq 2^{-10N} \ell(Q(\sbf^*))$, so the parameter $N$ will help us obtaining some ``space'' between our region $\sbf$ and the top scale $Q(\sbf^*)$ which will help us later when obtaining estimates.

In this section, we will first (along Subsection~\ref{sec:construction}) focus on analyzing some geometrical properties of our Whitney regions to construct $\psi_\sbf$ using level lines of the Green's function $u$. The fact that those level lines only behave well on Whitney regions associated to each $\sbf$ introduces several technical difficulties, like the need to use stopping-time distances and cut-offs. In the second part (Subsection~\ref{sec:lip}) we take advantage of the connection between the Green's function $u$ and the graphs $\psi_\sbf$ that we just constructed, to transfer estimates (obtained from the PDE) from $u$ to $\psi_\sbf$. From these, we will deduce that $\psi_\sbf$ is indeed Lipschitz, and approximates the original graph $\Sigma$ at the scale of $\sbf$.

\subsection{\for{toc}{\small}Construction of the approximating Lip(1,1/2) graphs} \label{sec:construction}

In contrast to \cite{BHMN1}, our sawtooth domains will not be defined in terms of a (regularized) distance to a `contact set', instead we use a regularized version of a stopping time distance from \cite{DS1}. For our fixed $\sbf$ we define 
\[d(\mbf{x}) := \inf_{Q \in \sbf}\big( \dist(\mbf{x}, Q) + \diam(Q)\big), \qquad \mbf{x} \in \mathbb{R}^n. \]
Then the function $d(\mbf{x})$ is clearly $\Lip(1,1/2)$ with constant at most $1$ since it is the point-wise infimum over such functions. We set also
\[F := \big\{\mbf{x} \in \mathbb{R}^n: d(\mbf{x}) = 0\big\}.\]
Unfortunately, he function $d(\mbf{x})$ is not regular enough for us to carry out our computations later, so we prefer to work with the following regularization of $d$.

\begin{lemma}[{\cite[Lemma 3.24]{BHHLN-CME}}] \label{regdistlem.lem}
There exists a function $h: \RR^n \longrightarrow \RR$ satisfying
\begin{align}\label{regular-dist}
 d(\mbf{x})/10  \le  h(\mbf{x})  \le  C_nd(\mbf{x}),
\qquad
\|h\|_{\Lip(1,1/2)}\lesssim 1, \qquad \|\mathcal{D}_t h\|_{\BMO(\re)}
\lesssim
1,
\end{align}
for a constant $C_n > 1$ (so concretely, $h(\mbf{x}) = 0$ on $F$), and
\begin{equation*}
d(\mbf{x})^{2\,k-1}\,|\partial_t^k h(\mbf{x})|+d(\mbf{x})^{k-1}\,|\nabla_x^k h(\mbf{x})|\lesssim_k 1,
\qquad \forall\,\mbf{x}\notin F,\, k\in\NN,
\end{equation*}
where the implicit constants and $C_n$ depend only on $n$ (and on $k$ in the last estimate).  
\end{lemma}

The purpose of the function $h(\mbf{x})$ is to create a sawtooth (above its graph) where we have the non-degeneracy estimates for $u$. Roughly speaking, for each $\mbf{x}$, we want to work with level sets of $u$ that are roughly above $h(\mbf{x})$, and ensure these points lie in our $U_Q$ regions, because $u$ behaves well there. Let us first use the stopping time distance to identify the cube with the correct scale.

\begin{lemma}\label{cubeid.lem}
Let $Q_0 \in \sbf$, and suppose that $\mbf{x} \in 500Q_0$ and $\rho \in (h(\mbf{x})/10, 10^{10}C_n\diam(Q_0))$. 
Then there exists a cube $Q \subseteq Q_0$, $Q \in \sbf$ such that 
\[\mbf{x} \in 500Q \quad \text{ and } \quad 10^{-10}\diam(Q) \le \rho \le 10^{10}C_n\diam(Q).\]
\end{lemma}
\begin{proof}
If $\rho \ge 10^{-10}\diam(Q_0)$, then we may just select $Q = Q_0$.
Hence, assume $\rho < 10^{-10} \diam(Q_0)$.

To begin, since $\rho > h(\mbf{x})/10$, Lemma~\ref{regdistlem.lem} implies that $\rho > d(\mbf{x})/100$. 
So by definition of $d(\mbf{x})$ there exists some $Q' \in \sbf$ such that
\[\dist(\mbf{x}, Q') + \diam(Q') < 100\rho.\]
Then, we may traverse one by one the ancestors of $Q'$ until we find the minimal $Q$ with $10^5\rho \le \diam(Q)$, which will then satisfy 
\begin{equation}\label{rhosand.eq}
	10^5\rho \le \diam(Q) \le 2\cdot 10^5 \rho \leq 10^{-4} \diam(Q_0).
\end{equation}
Therefore, $Q' \subseteq Q \subseteq Q_0$, whence $Q \in \sbf$ by semi-coherency. Also, since $Q' \subseteq Q$, we also have 
\[ \dist(\mbf{x}, Q) \leq \dist(\mbf{x}, Q') < 100\rho \le 10^{-3}\diam(Q), \]
so $\mbf{x} \in 500Q$. (In fact, \eqref{rhosand.eq} gives a better bound than when $\rho \ge 10^{-10}\diam(Q_0)$ and $Q = Q_0$.)
\end{proof}

Next, with regards to later putting together Lemmas~\ref{fndndpt.lem} and \ref{cubeid.lem}, we choose $\delta \in (0, 1/2)$ with
\begin{equation}\label{deltakchoice.eq}
	\big(10^{-10}\diam(Q), 10^{10}C_n\diam(Q)\big) \; \subseteq \; \big(\delta\ell(Q), \delta^{-1}\ell(Q)\big).
\end{equation}
And then, for this $\delta$, we fix $N = N(\delta)$ as given by Lemma~\ref{fndndpt.lem}. This is our {\bf final choice of $N$, which defines $K= 2^N$}. Clearly, $N$ and $K$ only depend on the structural constants; therefore we will not make dependencies on $N$ or $K$ explicit anymore.

Having fixed $K$, we collect some the information we have produced in a more concise format. First we have the following lemma.

\begin{lemma}[Definition of $\psi(r; \cdot)$]\label{1t1etc.lem}
Define the projections $\pi^{\perp}: \mathbb{R}^{n+1} \longrightarrow \mathbb{R}$ and $\pi: \mathbb{R}^{n+1} \longrightarrow \mathbb{R}^n$ by
\[\pi^{\perp}(x_0, \mbf{x}) := x_0, \qquad \pi(x_0, \mbf{x}) = \mbf{x}.\]
Then, for any fixed $\mbf{x} \in 10^5Q(\sbf)$, 
\[\text{the set } \;\; \pi^{-1}(\mbf{x}) \cap \Omega_{\sbf}
\;\; \text{ is a line segment, } \quad \text{ and } \quad
\mathcal{I}_{\mbf{x}} := \overline{\pi^{\perp}(\pi^{-1}(\mbf{x}) \cap \Omega_{\sbf})} \;\; \text{is an interval.}\]
Moreover, the function (of $x_0$)  $u(x_0, \mbf{x}) : \mathcal{I}_{\mbf{x}} \to \mathbb{R}$ is strictly increasing.
Therefore, we define, for $r \in u(\mathcal{I}_{\mbf{x}}, \mbf{x})$, the function $\psi(r;\mbf{x})$ implicitly as the unique element in $\mathcal{I}_{\mbf{x}}$ such that
\begin{equation*}
u(\psi(r; \mbf{x}),\mbf{x}) = r.
\end{equation*}
For brevity in the notation, we also define
\[\mbf{\Psi}(r;\mbf{x}) := (\psi(r; \mbf{x}),\mbf{x}).\]
\end{lemma}
\begin{proof}	
First, recall that $\Omega_{\sbf}$ is the union of $U_Q$ with $Q \in \sbf$, and $\sbf \subseteq \mathcal{G}_{\sbf^*}$ (see Lemma~\ref{graphregimes.lem}). Therefore, Lemma \ref{fkwhsa.lem} (specifically \eqref{nondegwhluq.eq}) yields that $\partial_{y_0} u(y_0, \mbf{y}) > 0$ for any $(y_0, \mbf{y}) \in \Omega_{\sbf}$. 

Now fix $\mbf{x} \in 10^5Q(\sbf)$. We are only left with showing that $\pi^{-1}(\mbf{x}) \cap \Omega_{\sbf}$ is a line segment or, equivalently, that $\mathcal{I}_{\mbf{x}}$ is an interval. Fix $Q \in \sbf$. First note that we have (see \eqref{whitdef.eq})
\[ \pi^\perp (\pi^{-1}(\mbf{x}) \cap U_Q) = (\psi(\mbf{x}) + K^{-1} \ell(Q), \psi(\mbf{x}) + K \ell(Q)), \]
which is indeed an interval. Next we note that if $Q^*$ is the parent of $Q$ in $\dd$, then their associated intervals (as in the last display) intersect because $\big{(} K^{-1}\ell(Q), K\ell(Q) \big{)} \cap \big{(}K^{-1}\ell(Q^*), K\ell(Q^*)\big{)} \neq \emptyset$
if $K$ is large enough (depending on $n$). So going all the way up from $Q$ to $Q(\sbf)$ (by semi-coherency):
\[ \bigcup_{Q' : Q \subseteq Q' \subseteq Q(\sbf)} \pi^\perp (\pi^{-1}(\mbf{x}) \cap U_{Q'}) 
=
(\psi(\mbf{x}) + K^{-1}\ell(Q), \psi(\mbf{x}) + K\ell(Q(\sbf))),  \]
which is an interval. Since the right endpoint of these intervals does not depend on $Q$, the arbitrariness of $Q \in \sbf$ in the last display implies that $\mathcal{I}_{\mbf{x}}$ is an interval, hence finishing the proof. 
\end{proof}

Without loss of generality (by a translation) we assume that, if $\mbf{x}_{Q(\sbf)}$ is the center of $Q(\sbf)$,
\begin{equation}\label{psicentervalzero.eq}
 \psi(h(\mbf{x}_{Q(\sbf)}); \mbf{x}_{Q(\sbf)}) = 0.
\end{equation}
With that, we finally define the (candidate to) \textbf{approximating graph} using $\psi(r; \cdot)$ from Lemma~\ref{1t1etc.lem}:
\begin{equation}\label{psistrSdef.eq}
	\psi_{\sbf}(\mbf{x}) := \psi^*_{\sbf}(\mbf{x}) \, \phi(\mbf{x}),
	\qquad \text{where} \qquad 
\psi^*_{\sbf}(\mbf{x}) := \psi(h(\mbf{x}); \mbf{x}),
\end{equation}
where $\phi \in C^\infty_c(\RR^n)$ is a cut-off function adapted to $Q(\sbf)$, in the sense that
\begin{equation}\label{cutoffqsest1.eq}
\mathbbm{1}_{2Q(\sbf)} \le \phi \le \mathbbm{1}_{4Q(\sbf)}, \quad \|\phi\|_{\Lip(1,1/2)} \lesssim \ell(Q(\sbf))^{-1}.
\end{equation}
Actually, we may choose $\phi$ to satisfy 
\begin{equation}\label{cutoffqsest2.eq}
|\nabla^k_x \phi| \, \ell(Q(\sbf))^k + |\partial_t^k \phi|\, \ell(Q(\sbf))^{2k} \lesssim_k 1
\qquad 
\text{for } k = 1,2,\dots
\end{equation}

Now that we have defined $\psi_\sbf$, we need to show that it is regular $\Lip(1,1/2)$, and that it approximates the original graph $\Sigma$ in the sense of \eqref{eq:approx_graph2}. As anticipated at the beginning of the section, it will be key to exploit the connection between $\psi_\sbf$ and $u$ later. But since $u$ only behaves well on sawtooth/Whitney regions, we need to check that our construction also works on those.

Indeed, for any given $Q_0 \in \sbf$, we define a localized stopping time by
\[\sbf(Q_0) := \{Q \in \sbf: Q \subseteq Q_0\}\] 
Then we associate two sawtooths to $\sbf(Q_0)$ (see Figure~\ref{fig:sawtooths}): the `usual' (dyadic) one, adapted to $\Omega$,
\[\Omega_{\sbf(Q_0)} := \bigcup_{Q \in \sbf(Q_0)} U_Q \quad 
\big( \! \subseteq \Omega \big);\]
and the other is a continuous version, in the upper half space
\begin{equation}\label{fncySdef.eq}
\mathcal{S}(Q_0) := \big{\{}(r, \mbf{x}): \; x \in 500Q_0, \;\; r \in \big(h(\mbf{x})/10, 10^{10}C_n\diam(Q_0)\big)\big{\}}
\quad \big(\! \subseteq \ree_+\big).
\end{equation}
In the case that $Q_0 = Q(\sbf)$ we omit $Q_0$ from our notation, e.g. $\mathcal{S} = \mathcal{S}(Q(\sbf))$.
The reason for introducing these objects (and also for our choice of $N$ and $K$ in \eqref{deltakchoice.eq}) is apparent in the following lemma: $\psi(r; \cdot)$ ``maps'' the sawtooth $\mathcal{S}(Q_0)$ to the other sawtooth $\Omega_{\sbf(Q_0)}$.

\begin{figure}
	\centering 
	\includegraphics[width=.95\textwidth]{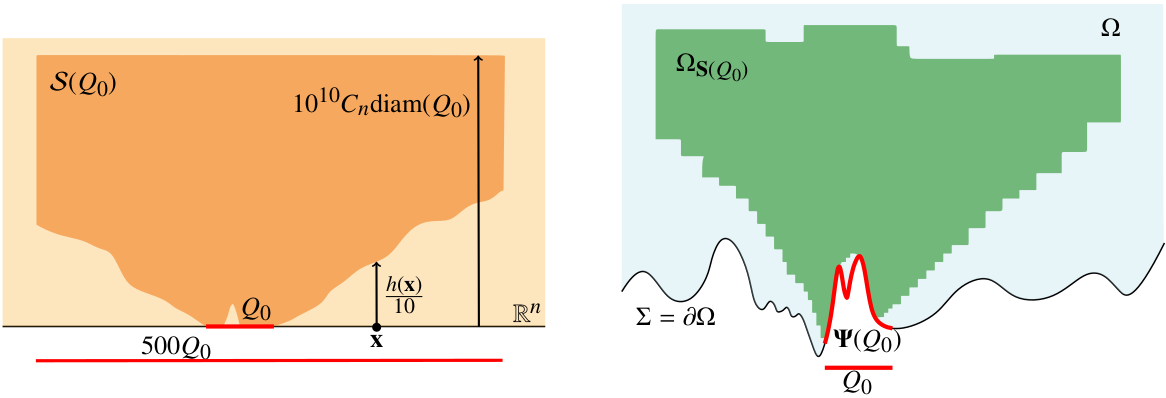}
	\caption{
	A (not-to-scale) depiction of a ``continuous'' sawtooth $\mathcal{S}(Q_0)$ in $\RR^{n+1}$ (the left figure) associated with the dyadic sawtooth $\Omega_{\sbf(Q_0)}$ inside $\Omega$ (the right figure). The map $\psi(r; \cdot)$ takes $\mathcal{S}(Q_0)$ inside $\Omega_{\sbf(Q_0)}$, twisting it to adapt to the shape of $\Omega$.
	}
	\label{fig:sawtooths}
\end{figure}

\begin{lemma} \label{map_sawtooths.lem}
The function $\psi(r;\mbf{x})$ from Lemma~\ref{1t1etc.lem} is well-defined for $(r, \mbf{x}) \in \mathcal{S}$. Moreover,
for any $Q_0 \in \sbf$, it holds
\[(r, \mbf{x}) \in \mathcal{S}(Q_0) \quad \implies \quad (\psi(r; \mbf{x}), \mbf{x}) \in \Omega_{\sbf(Q_0)}.\]
\end{lemma}
\begin{proof}
Let $(r,\mbf{x}) \in \mathcal{S}(Q_0)$. By definition of $\mathcal{S}(Q_0)$ and Lemma \ref{cubeid.lem} it holds that there exists $Q \subseteq Q_0$ with $Q \in \sbf$ such that
\[\mbf{x} \in 500Q \quad \text{ and } \quad 10^{-10}\diam(Q) \le r \le 10^{10}C_n\diam(Q).\]
By Lemma \ref{fndndpt.lem} and the fact we have chosen $K$ so that \eqref{deltakchoice.eq} holds, there exists $(x_0, \mbf{x})$ such that
\[u(x_0,\mbf{x}) =r, \quad \text{and} \quad  (x_0,\mbf{x}) \in U_Q.\]
This of course gives that $x_0 \in \mathcal{I}_{\mbf{x}}$, as defined in Lemma \ref{1t1etc.lem}. By the same Lemma \ref{1t1etc.lem}, $\psi(r;\mbf{x})$ is well-defined and it must be 
$\psi(r;\mbf{x}) = x_0$. Hence $(\psi(r;\mbf{x}),\mbf{x}) = (x_0,\mbf{x}) \in U_Q \subseteq \Omega_{\sbf(Q_0)}$, as desired.
\end{proof}

\subsection{\for{toc}{\small}Pointwise estimates at level sets, and the approximating graph is Lip(1, 1/2)} \label{sec:lip}

Next we want to look closer at the properties of $\psi(r,\mbf{x})$ over $\mathcal{S}$. In fact, it will be crucial to exploit that $\psi(r; \cdot)$ has been defined in terms of level sets of the Green function, for which we have nice estimates in sawtooth/Whitney regions. Let us confirm that the estimates also apply to $\psi(r; \cdot)$.

\begin{lemma}[Transference of estimates from $u$ to $\psi(r; \cdot)$]\label{lemma:estimates:G-star}
Let $(r,y,s)\in\mathcal{S}$. Then,
\begin{equation}\label{eqn:estimates:G-star:1}
|\partial_s \psi(r;y,s)|\lesssim |(\partial_s u)(\mbf{\Psi}(r;y,s))|,
\quad 
|\nabla_{y,r}^2 \psi(r;y,s)|\lesssim |(\nabla_Y^2 u)(\mbf{\Psi}(r;y,s))|,
\end{equation}	
\begin{equation}\label{eqn:estimates:G-star:2}
|\nabla_{y,r}\partial_s \psi(r;y,s)|
\,\lesssim\,
|(\nabla_Y\partial_s u)(\mbf{\Psi}(r;y,s))|\,+\, |(\partial_s u)(\mbf{\Psi}(r;y,s))|\,|(\nabla_Y^2 u)(\mbf{\Psi}(r;y,s))|,
\end{equation}	
\begin{equation}\label{eqn:estimates:G-star:3}
|\nabla_{y,r} \psi(r;y,s)| \,+\,
r\,|\partial_s \psi(r;y,s)|\,+\,r\,|\nabla_{y,r}^2 \psi(r;y,s)| 
\,+\,r^2\, |\nabla_{y,r}\partial_s \psi(r;y,s)| \lesssim 1.
\end{equation}	
\end{lemma}
\begin{proof}
	Let us include the proof from \cite[Lemma 5.18]{BHMN1}. By Lemma~\ref{map_sawtooths.lem}, it holds 
	\[
	u(\psi(r; y, s), y, s) = r, 
	\qquad (r, y, s) \in \mathcal{S}.
	\]
	Differentiating implicitly this equation with respect to $r$ yields 
	\begin{equation} \label{partialr.eq}
	(\partial_{y_0} u) (\mbf{\Psi}(r; y, s))\, \partial_r \psi(r; y, s) = 1, 
	\quad \text{i.e.} \quad 
	\partial_r \psi(r; y, s) = \frac{1}{(\partial_{y_0} u) (\mbf{\Psi}(r; y, s))}, 
	\end{equation}
	and similarly, differentiating with respect to $s$ and $y_i$ ($i = 1, \ldots, n$), we obtain 
	\[
	\partial_s \psi(r; y, s) = - \frac{(\partial_s u)(\mbf{\Psi}(r; y, s))}{(\partial_{y_0} u)(\mbf{\Psi}(r; y, s))}
	\quad \text{ and } \quad 
	\partial_{y_i} \psi(r; y, s) = - \frac{(\partial_{y_i} u)(\mbf{\Psi}(r; y, s))}{(\partial_{y_0} u)(\mbf{\Psi}(r; y, s))}.
	\]
	Recalling that $\partial_{y_0} u(\mbf{\Psi}(r; y, s)) \approx 1$ for $(r, y, s) \in \mathcal{S}$ (a consequence of Lemma \ref{fkwhsa.lem}, which can be applied for points in $\mathcal{S}$ because of Lemma~\ref{map_sawtooths.lem}--, the estimates in \eqref{eqn:estimates:G-star:1} and \eqref{eqn:estimates:G-star:2} follow from the equations in the previous displays, differentiating them again when needed, and using the estimates in \eqref{ptwisederest.eq}. 
	The estimates in \eqref{eqn:estimates:G-star:1} follow similarly, using Lemma~\ref{ptwiseestinsaw4der.lem} and the fact that $\dist(\mbf{\Psi}(r; y, s)) \approx u(\mbf{\Psi}(r; y, s)) = r$ by \eqref{M2defa.eq}. 
\end{proof}

These pointwise estimates will be the key to confirm that $\psi_\sbf$ defined in \eqref{psistrSdef.eq} is $\Lip(1,1/2)$ and has approximates the original graph $\Sigma$ in the required way \eqref{eq:approx_graph2}. Let us start by the latter.

\begin{lemma}[$\psi_\sbf$ approximates $\Sigma$]\label{psiSclose.lem}
	For every $\mbf{x} \in 4Q(\sbf)$ it holds that
	\begin{equation}\label{psiclshest.eq}
		0 \le \psi_\sbf(\mbf{x}) - \psi(\mbf{x}) \le M_0M_2h(\mbf{x}),
	\end{equation}
	where $M_2$ is from Lemma \ref{gfcomplemt1.eq} (or \ref{gfcomplemt2.eq}).
	In particular, \eqref{eq:approx_graph2} holds with $C_1 := 2C_nM_0M_2$.
\end{lemma}
\begin{proof}
	The fact that \eqref{eq:approx_graph2} follows from \eqref{psiclshest.eq} is because if $Q \in \sbf$ and $\mbf{x} \in 2Q$ (using Lemma~\ref{regdistlem.lem})
	\[(C_n)^{-1}h(\mbf{x}) \le d(\mbf{x}) \le \dist(\mbf{x}, Q) + \diam(Q) \le 2\diam(Q).\]
	
	Next, let us show \eqref{psiclshest.eq}. 
	Given $\mbf{x} \in 4Q(\sbf)$, Lemma~\ref{map_sawtooths.lem} yields 
	\[ u(\psi^*_{\sbf}(\mbf{x}),\mbf{x}) = h(\mbf{x}),
	\quad \text{and} \quad
	(\psi^*_{\sbf}(\mbf{x}),\mbf{x}) \in \Omega_{\sbf}.
	\]
	Then, by Lemma \ref{distgraphlem.lem} and \eqref{M2def.eq} (and also $\psi_\sbf(\mbf{x}) \leq \psi_\sbf^*(\mbf{x})$, see \eqref{psistrSdef.eq}), we get 
	\[
	\psi_{\sbf}(\mbf{x}) - \psi(\mbf{x}) 
	\leq
	\psi^*_{\sbf}(\mbf{x}) - \psi(\mbf{x}) 
	\le 
	M_0 \dist((\psi^*_{\sbf}(\mbf{x}),\mbf{x}), \Sigma)
	\leq 
	M_0 M_2 u(\psi^*_{\sbf}(\mbf{x}),\mbf{x})
	=
	M_0 M_2 h(\mbf{x}),\]
	which is the upper bound in \eqref{psiclshest.eq}; and the lower bound by $0$ comes from $(\psi_\sbf(\mbf{x}),\mbf{x}) \in \Omega_{\sbf} \subseteq \Omega$.
\end{proof}

Now, turning to show $\psi_\sbf$ is $\Lip(1, 1/2)$, we make an observation following from Lemma \ref{lemma:estimates:G-star}.
\begin{lemma}\label{estwhenhgtrzero.lem}
	If $(x,t) \in \mathbb{R}^n$ and $h(x,t) > 0$, then 
	\[|\partial_t \psi_{\sbf}(x,t)| \lesssim h(x,t)^{-1}.\]
\end{lemma}
\begin{proof}
	From the definition of $\psi_{\sbf}$ in \eqref{psistrSdef.eq}, we see that the estimate is trivial if $(x,t) \not \in 4Q(\sbf)$. Again recalling the definition of $\psi_{\sbf}$ and $\psi_{\sbf}^*$ (see \eqref{psistrSdef.eq}), Lemma \ref{lemma:estimates:G-star} gives the estimate if $(x,t) \in 2Q(\sbf)$. 
	
	Lastly, fix $\mbf{x} = (x, t) \in 4Q(\sbf) \setminus 2Q(\sbf)$. Then, it is easy to check that (having in mind Lemma~\ref{regdistlem.lem} and the definition of $d$) $h(\mbf{x}) \approx d(\mbf{x}) \approx \dist(\mbf{x}, Q(\sbf)) \approx \ell(Q(\sbf))$. Also, Lemma~\ref{regdistlem.lem} yields
	\begin{equation}\label{dtofhremind.eq}
		|\partial_t h(x,t)| \lesssim d(x,t)^{-1} \approx h(x,t)^{-1}.
	\end{equation}
	Then, by Lemma \ref{lemma:estimates:G-star}, the properties of $\phi$, and that $h$ satisfies \eqref{dtofhremind.eq}, it holds
	\begin{align*}
		|\partial_t \psi_{\sbf}(x,t)| &\le |\partial_t \psi(h(x,t); x,t)| + |\psi(h(x,t); x,t)\, \partial_t \phi(x,t)| 
		\\ & \lesssim |(\partial_{x_0} \psi)(h(x,t); x,t) \, \partial_t h(x,t)| + |(\partial_t \psi)(h(x,t);x,t)|  + |\psi(h(x,t); x,t)\, \partial_t \phi(x,t)| 
		\\ & \lesssim \ell(Q(\sbf))^{-1} + \ell(Q(\sbf))^{-1} + \ell(Q(\sbf)) \ell(Q(\sbf))^{-2} \lesssim \ell(Q(\sbf)) \approx h(x,t)^{-1},
	\end{align*}
	where we used our assumption that $\psi^*_\sbf(\mbf{x}_{Q(\sbf)}) = 0$ in \eqref{psicentervalzero.eq} (along with the computations in Lemma~\ref{psiSclose.lem}, and also $h(\mbf{y}) \lesssim d(\mbf{y}) \lesssim \ell(Q(\sbf))$ for $\mbf{y} \in 4Q(\sbf)$) to obtain
	\[
		|\psi(h(\mbf{x}); \mbf{x})|
		\leq 
		|\psi^*_\sbf(\mbf{x}) - \psi(\mbf{x})| 
		+ |\psi(\mbf{x}) - \psi(\mbf{x}_{Q(\sbf)})| 
		+ |\psi(\mbf{x}_{Q(\sbf)}) - \psi^*_\sbf(\mbf{x}_{Q(\sbf)})| 
		\lesssim 
		\ell(Q(\sbf)).
	\]
\end{proof}

\begin{lemma}[$\psi_\sbf$ is Lipschitz]\label{psiSlip.lem}
The function $\psi_\sbf$, as defined in \eqref{psistrSdef.eq}, is $\Lip(1,1/2)$. 
\end{lemma}
\begin{proof}
We first show that $\psi_{\sbf}^*$ is Lip(1,1/2) when restricted to $4Q(\sbf)$. To this end, fix $\eps' \in (0,1)$ be small, to be chosen later only depending on $n$. Fix also $\mbf{x}, \mbf{y} \in 4Q(\sbf)$.

First, if $\|\mbf{x} - \mbf{y}\| \geq \eps' \max\{h(\mbf{x}),h(\mbf{y})\}$ then from \eqref{psiclshest.eq} and the fact that $\psi$ is Lip(1,1/2) we have
\[|\psi_\sbf^*(\mbf{x}) - \psi_\sbf^*(\mbf{y})| \lesssim \|\mbf{x} - \mbf{y}\| + h(\mbf{x}) + h(\mbf{y}) \lesssim_{\eps'} \|\mbf{x} - \mbf{y}\|.\]
Therefore it suffices to treat the case 
$\|\mbf{x} - \mbf{y}\| < \eps' \max\{h(\mbf{x}),h(\mbf{y})\}$.
Without loss of generality we assume that $\max\{h(\mbf{x}),h(\mbf{y})\} = h(\mbf{x})$. 

We observe that since $h(\mbf{x}) > 0$ and $h$ is Lip(1,1/2), it is trivial to obtain
\begin{equation} \label{comparableh.eq}
	h(\mbf{z}) \ge h(\mbf{x})/2, \quad \forall \mbf{z}: \|\mbf{z} - \mbf{x}\| \le 10 \, \eps'h(\mbf{x})
\end{equation}
if $\eps'$ is sufficiently small, merely depending on the $\Lip(1, 1/2)$ constant of $h$, which in turn only depends on $n$ by Lemma~\ref{regdistlem.lem}. Actually, note that $\mbf{y}$ lies well inside this set.

Now, using our estimates on $\psi(r;\mbf{z})$ from Lemma \ref{lemma:estimates:G-star}, and our estimates on $h$ from Lemma \ref{regdistlem.lem} it holds for all
$\mbf{z} = (z,\tau)$ with $\|\mbf{z} - \mbf{x}\| \le 10 \eps'h(x)$ that
\[|\nabla_z \psi^*_{\sbf}(z,\tau)| = |\nabla_z \psi(h(z,\tau); z, \tau)| \le |(\partial_r \psi)(h(z,\tau); z, \tau)| \, |\nabla_z h(z,\tau)| + |(\partial_z \psi)(h(z,\tau); z, \tau)| \lesssim 1,\]
and similarly (see Lemma \ref{estwhenhgtrzero.lem}), recalling also \eqref{comparableh.eq},
\[|\partial_\tau\psi^*_{\sbf}(z,\tau)| \lesssim h(z,\tau)^{-1} \lesssim h(\mbf{x})^{-1}.\]
Then connecting $\mbf{x}$ to $\mbf{y}$ via a polygonal path made up of purely spatial and purely temporal line segments in the set $\{\mbf{z}: \|\mbf{z} - \mbf{x}\| \le 10 \eps'h(x)\}$ and using the previous two estimates gives that 
\[|\psi_\sbf^*(\mbf{x}) - \psi_\sbf^*(\mbf{y})| \lesssim \|\mbf{x} - \mbf{y}\| + \|\mbf{x} - \mbf{y}\|^2h(\mbf{x})^{-1} \lesssim_{\eps'} \|\mbf{x} - \mbf{y}\|\]
since $h(\mbf{x})^{-1} \lesssim_{\eps'} \|\mbf{x} - \mbf{y}\|^{-1}$. So we have checked that $\psi_{\sbf}^*$ is Lip(1,1/2) when restricted to $4Q(\sbf)$.

To continue, recall that we assumed  $\psi(h(\mbf{x}_{Q(\sbf)}); \mbf{x}_{Q(\sbf)}) = 0$ in \eqref{psicentervalzero.eq}. Therefore, since we just showed that $\psi^*_\sbf$ is Lip(1,1/2) in $4Q(\sbf)$ (actually, the computations in Lemma~\ref{psiSclose.lem} suffice), we have
\begin{equation}\label{psistarcrdbd.eq}
|\psi^*_{\sbf}(\mbf{x})|\lesssim \diam(Q(\sbf)), \quad \forall \mbf{x} \in 4Q(\sbf).
\end{equation}
Now we need to show that $\psi_\sbf$ is Lip(1,1/2), that is,
\begin{equation}\label{psiindeedlip.eq}
|\psi_\sbf(\mbf{x}) - \psi_\sbf(\mbf{y})| \lesssim \|\mbf{x} - \mbf{y}\|.
\end{equation}
By definition, if $\psi_\sbf \equiv 0$ in $(4Q(\sbf))^c$, so that \eqref{psiindeedlip.eq} holds whenever $\mbf{x}, \mbf{y} \in (4Q(\sbf))^c$. When $\mbf{x},\mbf{y} \in 4Q(\sbf)$, we use that $\psi^*_\sbf$ is Lip(1,1/2), \eqref{psistarcrdbd.eq} and \eqref{cutoffqsest1.eq} to also obtain \eqref{psiindeedlip.eq}
\begin{align*}
|\psi_\sbf(\mbf{x}) - \psi_\sbf(\mbf{y})|  &= |\psi^*_\sbf(\mbf{x})\phi(\mbf{x}) - \psi^*_\sbf(\mbf{y})\phi(\mbf{y})|
\\ & \le |\psi^*_\sbf(\mbf{x})\phi(\mbf{y}) - \psi^*_\sbf(\mbf{y})\phi(\mbf{y})| + |\psi^*_\sbf(\mbf{x})\phi(\mbf{x}) - \psi^*_\sbf(\mbf{x})\phi(\mbf{y})|
\\ & \lesssim \|\mbf{x} - \mbf{y}\| + \diam(Q(\sbf)) \diam(Q(\sbf))^{-1} \|\mbf{x} - \mbf{y}\| \lesssim \|\mbf{x} - \mbf{y}\|.
\end{align*}
Finally, we are left treating when $\mbf{x} \in 4Q(\sbf)$ and $\mbf{y} \in (4Q(\sbf))^c$, so $\psi_\sbf(\mbf{y}) = 0$. In this case, find $\hat{\mbf{y}}$ to be the point in the line segment joining $\mbf{x}$ to $\mbf{y}$ which is in $\partial (4Q(\sbf))$. Therefore, by the estimates above and the continuity of $\psi_\sbf$, we are also able to verify \eqref{psiindeedlip.eq} in this case simply by computing
\[
|\psi_\sbf(\mbf{x}) - \psi_\sbf(\mbf{y})|
=
|\psi_\sbf(\mbf{x})|
=
|\psi_\sbf(\mbf{x}) - \psi_\sbf(\hat{\mbf{y}})| \lesssim \|\mbf{x} - \hat{\mbf{y}}\| \le \|\mbf{x} - \mbf{y}\|.\]
\end{proof}

\section{Step 4: Regularity of the approximating graphs via square function estimates for the modified level sets}

The goal of this section is to show the following result.
\begin{proposition}[$\psi_{\sbf}$ is regular] \label{regular.prop}
	Let $\psi_\sbf$ be constructed by \eqref{psistrSdef.eq}. Then $\mathcal{D}_t \psi_\sbf$ is in $\BMO$, that is, $\psi_{\sbf}$ is a regular $\Lip(1, 1/2)$ function. 
\end{proposition}
Indeed, once we show this, the proof of Proposition~\ref{bigprop.prop} will be finished (because we already established Lemmas \ref{psiSclose.lem} and \ref{psiSlip.lem}), and hence Theorems~\ref{main1.thrm} and \ref{main2.thrm} will have been proved. 
For that purpose, it will be crucial to use square function estimates developed in the previous sections. 

We will break Step 4 up into smaller steps, mostly using the ideas of \cite{BHMN1} with one large difference. In \cite{BHMN1} only a $\BMO$ estimate on the `big cube' $Q(\sbf)$ is obtained (and then passed on to the original graph function). It will turn out that if we are trying to get an estimate on the $\BMO$ norm of $\mathcal{D}_t \psi_\sbf$ 
when the cube $Q$ (in the definition of $\BMO$) is in the stopping time $\sbf$, then the proof will be almost identical to 
that of \cite{BHMN1}. In fact, the same analysis extends to the case
that $Q \subseteq 50Q_0$, for $Q_0 \in \sbf$ with similar size to $Q$. On the other hand, if this is not the case (so deviating from \cite{BHMN1}), we will see that $\psi_\sbf$ is even more regular over $Q$ and we can still establish the $\BMO$ estimates.
In any case, the $\BMO$ norm of $\mathcal{D}_t \psi_\sbf$ will be under control within $50Q(\sbf)$. Finally, to handle the part outside $50Q(\sbf)$ we use cut-off functions and the John-Str\"omberg inequality. 

Getting to the details, as in \cite{BHMN1}, using the John-Str\"omberg inequality \cite{JohnS,StromJ}, to show that $\mathcal{D}_t\psi_\sbf(\mbf{y}) \in \BMO$ it suffices that there exists a (large) constant $M_\star$ such that
\begin{equation}\label{JShyp.eq}
\inf_{C_Q \in \RR} \big|\big\{
	\mbf{y}\in Q: |\mathcal{D}_t\psi_\sbf(\mbf{y})-C_Q|>M_\star
	\big\}\big|
	\le
	\frac13 \,|Q|, \qquad \forall Q \in \mathbb{D}.
\end{equation}
Of course, by Chebyshev's inequality, if there exists $C_\star \geq 0$ (independent of $Q$) such that
\begin{equation}\label{BmoimpJS.eq}
\inf_{C_Q \in \RR} \frac{1}{|Q|} \iint_{Q} |\mathcal{D}_t\psi_\sbf(\mbf{y})-C_Q| \, \d\mbf{y} \le C_\star, 
\qquad \forall Q \in \mathbb{D},
\end{equation}
then \eqref{JShyp.eq} holds with $M_\star = 3C_\star$. Moreover, the following stronger estimate implies \eqref{BmoimpJS.eq} and \eqref{JShyp.eq}:
\begin{equation}\label{LinftyimpJS.eq}
\inf_{C_Q \in \RR} |\mathcal{D}_t\psi_\sbf(\mbf{y})-C_Q| \le C_\star, \qquad \forall \, \mbf{y} \in Q, \; Q \in \dd.
\end{equation}

We are going to break proving \eqref{JShyp.eq} into a few cases. In the ``interesting" cases (Case 1, 2 and 3), we will actually show one of the stronger estimates \eqref{BmoimpJS.eq} or \eqref{LinftyimpJS.eq}, and we will later use these cases to prove \eqref{JShyp.eq} in the remaining case (Case 4). We organize these cases now.
\begin{itemize}
\item {\bf Case 1}: $Q \subseteq 50Q(\sbf)$ and there exists $Q_0 \in \sbf$ satisfying
\[Q \subseteq 50Q_0, \qquad \diam(Q) \ge \eps_0 \diam(Q_0),\]
where $\eps_0 > 0$ is to be chosen below.
\item {\bf Case 2}: $Q \subseteq 50Q(\sbf)$, and whenever $Q_0 \in \sbf$ is such that $Q \subseteq 50 Q_0$, then it holds
\[\diam(Q) < \eps_0 \diam(Q_0).\]
\item {\bf Case 3}: $Q \not\subseteq 50Q(\sbf)$ and 
\[\dist(Q,4Q(\sbf)) \ge 6\diam(Q).\]
\item {\bf Case 4:} $Q$ is not in Cases 1-3.
\end{itemize}

For a comparison, in \cite{BHMN1} the authors only had to deal with $Q(\mbf{S})$ (we need to consider all the other cases by the complexity of our construction). Actually, we will follow their strategy closely to deal with not only $Q(\sbf)$, but all cubes in Case 1 (and we postpone the proof for that reason, to focus in the more novel cases). Furthermore, it will turn out that the localization technique already introduced in \cite{BHMN1} helps when dealing with Case 2 and Case 3 cubes. Lastly, Case 4 cubes must be treated in a new different way, which we address right now.

\subsection{\for{toc}{\small}Handling Case 4 cubes, assuming Cases 2 and 3}
As mentioned above, we will (in the next subsections) prove \eqref{BmoimpJS.eq} or \eqref{LinftyimpJS.eq} for cubes in Cases 1-3. Assuming this, let us handle Case 4 cubes.

\begin{lemma}
Suppose that \eqref{BmoimpJS.eq} holds for cubes in Cases 1-3. Then \eqref{JShyp.eq} holds for Case 4 cubes. 
\end{lemma}
\begin{proof}
Fix $Q \in \dd$ a cube in Case 4. This means that
\begin{equation*}
	Q \not \subseteq 50Q(\sbf), 
	\quad \text{ and } \quad 
	\dist(Q,4Q(\sbf)) < 6\diam(Q).
\end{equation*}

The idea now is to cover $Q$ by subcubes which are in Cases 1-3. For that purpose, it will be notationally convenient to ``shift'' the generations in our dyadic structure $\dd$ to a new dyadic grid $\widetilde{\dd}$, with exactly the same cubes, but where the generations are relabeled so that $\widetilde{\dd}_k$ consists of cubes with side length $2^{-k}\ell(Q)$ (so concretely $\widetilde{\dd}_0$ contains $Q$). Writing $k_0 := 10^5$, we set 
\begin{equation*}
	\mathcal{F}_1 := \big\{Q' \in \widetilde{\dd}_{k_0}: \dist(10Q', 4Q(\sbf)) < 60\diam(Q')\big\}.
\end{equation*}

The point of this definition is that we should focus mostly on $\mathcal{F}_1$ cubes, 
because it holds that
\begin{equation} \label{notF1.eq}
	Q' \in \widetilde{\dd}_{k_0} \setminus \mathcal{F}_1 
	\quad \implies \quad 
	10Q' \text{ is in Cases 1-3}. 
\end{equation}
Indeed, if $10Q' \subseteq 50Q(\sbf)$, then $10Q'$ is in Cases 1-2. Otherwise, if $10Q'$ did not belong to Case 3, we would have $\dist(10Q', 4Q(\sbf)) < 6 \diam(10Q') = 60 \diam(Q')$, which would contradict $Q' \notin \mathcal{F}_1$.

Let us consider two cases depending on the size of $Q$:

{\bf Case I}: $Q$ is ``not too big'', say $\diam(Q) \le 2^{200} \diam(Q(\sbf))$. Then, for $Q' \in \mathcal{F}_1$ and $\mbf{x} \in 10Q'$:
\begin{multline*}
\| \mbf{x}_{Q(\sbf)} - \mbf{x}\| 
\leq 
\diam(10Q') + \dist(10Q', Q(\sbf)) + \diam(Q(\sbf))
\le 
70\diam(Q') + \diam(Q(\sbf))
\\ \le 
70\cdot 2^{200} \cdot 2^{-10^5} \diam(Q(\sbf)) + \diam(Q(\sbf)) 
\le 
2\diam(Q(\sbf)).
\end{multline*}
Therefore, $10Q' \subseteq 50Q(\sbf)$, so $10Q'$ falls in Cases 1-2. And if $Q' \in \widetilde{\dd}_{k_0} \setminus \mathcal{F}_1$, we already showed in \eqref{notF1.eq} that $10Q'$ is in Cases 1-3. Putting these together and using the hypothesis of the lemma, we have shown that for every $Q' \in \widetilde{\dd}_{k_0}$ there exists some $C_{Q'} \in \RR$ such that 
\[\frac{1}{|10Q'|} \iint_{10Q'} |\mathcal{D}_t\psi_\sbf(\mbf{y})-C_{Q'}| \, \d\mbf{y} \le C_\star,\]
which of course this implies that
\begin{equation}\label{case4case1est.eq}
	\frac{1}{|Q'|} \iint_{Q'} |\mathcal{D}_t\psi_\sbf(\mbf{y})-C_{Q'}| \, \d\mbf{y} \le 10^{n+1}C_\star.
\end{equation}

To finish the proof in Case I, it remains to obtain a common constant $C_Q$ that works well for all $Q' \in \widetilde{\dd}_{k_0}(Q) := \{Q' \in \widetilde{\dd}_{k_0} : Q' \subseteq Q\}$. For that purpose, first note that for adjacent cubes $Q', Q'' \in \widetilde{\dd}_{k_0}(Q)$ (so that $Q' \subseteq 10Q''$) we have, using the previous estimates, that
\begin{multline} \label{adjacentCQ.eq} 
	|C_{Q'} - C_{Q''}| \le \bariint_{Q'} |\mathcal{D}_t\psi_\sbf(\mbf{y}) - C_{Q''}| \d\mbf{y} +\bariint_{Q'} |\mathcal{D}_t\psi_\sbf(\mbf{y}) - C_{Q'}| \d\mbf{y}
	\\ \le 10^{n+1} \bariint_{10Q''} |\mathcal{D}_t\psi_\sbf(\mbf{y}) - C_{Q''}|\d\mbf{y} + \bariint_{Q'} |\mathcal{D}_t\psi_\sbf(\mbf{y}) - C_{Q'}|\d\mbf{y}
	\le 2\cdot 10^{n+1}C_\star.
\end{multline} 
Hence, if we pick any $\widetilde{Q} \in \widetilde{\dd}_{k_0}(Q)$ and set $C_Q := C_{\widetilde{Q}}$, we can ``connect" $\widetilde{Q}$ to any other cube $Q' \in \widetilde{\dd}_{k_0}(Q)$  by a chain of at most $L_n \lesssim (2^{10^5})^{n+1}$ adjacent cubes (the amount of cubes in $\widetilde{\dd}_{k_0}(Q)$, see Figure~\ref{fig:chains}) so that iterating the estimate in the previous display all along the chain yields
\begin{equation*}
	|C_Q - C_{Q'}| \le 2\cdot 10^{n+1}L_nC_\star, 
	\qquad \forall Q' \in \widetilde{\dd}_{k_0}(Q).
\end{equation*}
Therefore, if we define $M_\star := (2\cdot 10^{n+1}L_n + 100^{n+1})C_\star$, we obtain that, for any $Q' \in \widetilde{\dd}_{k_0}(Q)$,
\[\big \{\mbf{y} \in Q':  | \mathcal{D}_t\psi_\sbf(\mbf{y}) - C_Q| > M_\star\big \} 
\; \subseteq \; 
\big\{\mbf{y} \in Q':  | \mathcal{D}_t\psi_\sbf(\mbf{y}) - C_{Q'}| > 100^{n+1} C_\star\big\},\]
so that using Chebyshev and \eqref{case4case1est.eq} it holds
\begin{equation} \label{chebyshevNotF1.eq}
	\big| \big\{\mbf{y} \in Q':  | \mathcal{D}_t\psi_\sbf(\mbf{y}) - C_Q| > M_\star \big\} \big| \le 10^{-(n+1)} \, |Q'| \le (1/10) \, |Q'|.
\end{equation} 
Summing over $Q' \in \dd_{k_0}(Q)$ (they are disjoint and cover $Q$) we finally get \eqref{JShyp.eq} in Case I. 

{\bf Case II}: $Q$ is very big, say $\diam(Q) > 2^{200} \diam(Q(\sbf))$. In this case we are going to be able to ignore the cubes in $\mathcal{F}_1$. Indeed, if $Q' \in \mathcal{F}_1$ and $\mbf{x} \in 10Q'$, we may compute, similarly as in Case I,
\[ \| \mbf{x}_{Q(\sbf)} - \mbf{x} \| 
\leq 
70 \diam(Q') + \diam(Q(\sbf))
\leq 
70 \cdot 2^{10^{-5}} \diam(Q) + 2^{-200} \diam(Q) 
\leq 
2^{-100} \diam(Q),\]
from which we readily infer the smallness of $\mathcal{F}_1$, in the form of 
\begin{equation}\label{F1small.eq}
	\bigg{|}\bigcup_{Q' \in \mathcal{F}_1} Q' \bigg{|} \le \frac{1}{100} \, |Q|.
\end{equation}

On the other hand, for each $Q' \in \widetilde{\dd}_{k_0} \setminus \mathcal{F}_1$, \eqref{notF1.eq} implies that \eqref{case4case1est.eq} holds, as in Case I. Again, the problem resides on using a common $C_Q$. For that, we argue as in Case I: define $C_Q := C_{\widetilde{Q}}$, where $\widetilde{Q}$ is any cube in $\widetilde{\dd}_{k_0}(Q) \setminus \mathcal{F}_1$. Then, the point is to connect any given $Q' \in \widetilde{\dd}_{k_0}(Q) \setminus \mathcal{F}_1$ to $\widetilde{Q}$ using a chain of at most $L_n$ cubes, but this time all of the cubes in the chain belong to $\widetilde{\dd}_{k_0} \setminus \mathcal{F}_1$ (so that we may apply \eqref{adjacentCQ.eq} at each step). To construct such chain from $Q'$ to $\widetilde{Q}$, we may have problems avoiding cubes which are simultaneously in $\mathcal{F}_1$ and below $Q$ (see Figure~\ref{fig:chains}). However, if we allow ourselves to use all the cubes in the larger region $\{Q'' \in \widetilde{\dd}_{k_0} \setminus \mathcal{F}_1 : Q'' \cap 10Q \neq \emptyset\}$, we will have enough space (and this dilation preserves $L_n \lesssim (2^{10^5})^{n+1}$): indeed, if $\mbf{x} \in Q'' \in \mathcal{F}_1$, then 
\[
\| \mbf{x} - \mbf{x}_{Q(\sbf)} \|
\leq 
\diam(Q'') + \dist(Q'', Q(\sbf)) + \diam(Q(\sbf))
\leq 
(61 \cdot 10^{-5} + 2^{-200}) \diam(Q)
\leq 
\frac{1}{100} \diam(Q),
\]
which says that the ``diameter'' of $\mathcal{F}_1$ is much smaller than that of $10Q$, so that if connecting $Q'$ to $\widetilde{Q}$ was impossible within $Q$ (avoiding $\mathcal{F}_1$), there will always be room to first escape to $10Q \setminus Q$, and then find an alternative path towards $\widetilde{Q}$ (because $\mathcal{F}_1$ is too small to block any section of $10Q \setminus Q$). Summing up, we also obtain \eqref{chebyshevNotF1.eq} for $Q' \in \widetilde{\dd}_{k_0}(Q) \setminus \mathcal{F}_1$ in Case II.

\begin{figure}
	\centering 
	\includegraphics[width=.9\textwidth]{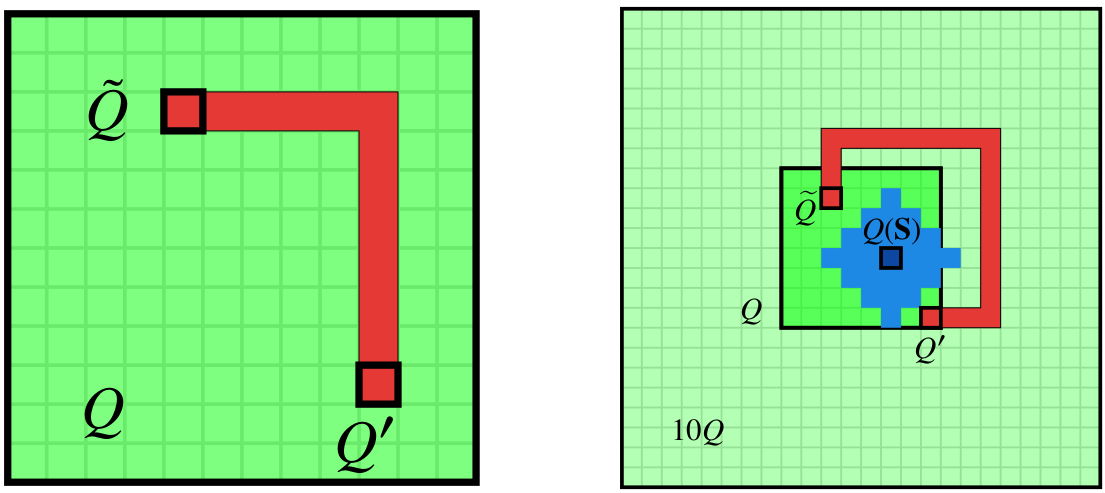}
	\caption{
		The picture in the left represents Case I: we show (in red) a possible chain of cubes in $\widetilde{D}_{k_0}(Q)$ connecting the fixed $\widetilde{Q}$ with some other $Q' \in \widetilde{D}_{k_0}(Q)$. The picture in the right represents Case II: to connect $Q'$ to $\widetilde{Q}$, we need to avoid the cubes in $\mathcal{F}_1$ (in blue), which surround $Q(\sbf)$, and may ``disconnect'' $Q'$ and $\widetilde{Q}$ within $Q$; but there will always be space for a chain within $10Q$.
		}
	\label{fig:chains}
\end{figure}

Therefore, separating the cubes in $\mathcal{F}_1$ (for which we use \eqref{F1small.eq}) and in $\widetilde{\dd}_{k_0}(Q) \setminus \mathcal{F}_1$ (for which we just argued that we can obtain the same estimates as in Case I), we obtain \eqref{JShyp.eq} also in Case II (and with the same $M_\star$ as in Case I):
\begin{equation*}
	\big| \big\{\mbf{y} \in Q:  | \mathcal{D}_t\psi_\sbf(\mbf{y}) - C_Q| > M_\star \big\}\big|
	\leq 
	\frac{1}{100} \, |Q| + \!\!\! \sum_{Q' \in \widetilde{\dd}_{k_0}(Q) \setminus \mathcal{F}_1} \frac{1}{10} \, |Q'|
	\leq 
	\frac13 \, |Q|.
\end{equation*} 
\end{proof}

\subsection{\for{toc}{\small}Localization of $\mathcal{D}_t\psi_\sbf(\mbf{y})$}\label{localizeandcase3.sect}

In this (sub)section we localize $\mathcal{D}_t\psi_\sbf(\mbf{y})$, as was done in \cite{BHMN1}. Heuristically (and realistically) we can only expect to gain information about $\psi_\sbf$ from our hypotheses on the parabolic measure for cubes in or near the stopping time $\sbf$; therefore the BMO estimates on $\mathcal{D}_t\psi_\sbf(\mbf{y})$ for cubes that are in Case 2 and 3 should come from somewhere else. This localization will not only help us in Case 1; but it will trivialize Case 3, and make Case 2 easier to handle.

Let $\varphi \in C_c^\infty(\re)$ be
an even function with $\mathbbm{1}_{(-1,1)}\le \varphi \le \mathbbm{1}_{(-2,2)}$ and set 
\[\Phi(x,t) := \varphi(x_1) \, \varphi(x_2)\dots \varphi(x_{n-1})
\, \varphi(t/2), \qquad  \mbf{x}= (x,t) = (x_1,\dots x_{n-1}, t) \in\re^n.\] 
We set $\Phi_R(\mbf{x}):= \Phi(x/R,t/R^2)$ for $\mbf{x}= (x,t) \in\re^n$.
Recalling the definition of $\mathrm{I_P}$ in \eqref{def-IP}, we write $V_R:=\Phi_R\,V$ for the locally truncated kernel, and we consider the localized 
parabolic fractional integral
\begin{equation}\label{IPRdef.eq}
\mathrm{I}_\mathrm{P}^R \, h(\mbf{x})
:=
\iint_{\re^n} V_R(\mbf{x}-\mbf{y})\,h(\mbf{y})\,\d\mbf{y}
=
\iint_{\re^n} \Phi_R(\mbf{x}-\mbf{y})\,V(\mbf{x}-\mbf{y})\,h(\mbf{y})\,\d\mbf{y}.
\end{equation}
We can then define the localized half-order time derivative
\begin{multline}\label{localdtdef}
	\mathcal{D}_t^R \psi_{\mbf{S}}(\mbf{x})
:=
\partial_t \circ\mathrm{I}_\mathrm{P}^R  \psi_{\mbf{S}}(\mbf{x})
=
\iint_{\re^n} K^R (\mbf{x}-\mbf{y})\,\psi_{\mbf{S}}(\mbf{y})\,\d\mbf{y}
: =
\iint_{\re^n} \partial_t\big(V_R(\mbf{x}-\mbf{y})\big)\,\psi_{\mbf{S}}(\mbf{y})\,\d\mbf{y}
\\
=
\iint_{\re^n} \partial_t\big(V(\mbf{x}-\mbf{y})\,\Phi_R(\mbf{x}-\mbf{y})\big)\,\psi_{\mbf{S}}(\mbf{y})\,\d\mbf{y},
\qquad
\mbf{x}\in\re^n.
\end{multline}
This operator should be viewed as a principal value operator, or $\partial_t$ should be considered in the weak sense. Let $\mathcal{E}^R:=\mathcal{D}_t-\mathcal{D}_t^R$ and set $K_R := \partial_t V - K^R$. Thus
\[
\mathcal{E}^R \psi_{\mbf{S}}(\mbf{x})
=
\iint_{\re^n} \partial_t\big(V(\mbf{x}-\mbf{y})\,(1-\Phi_R(\mbf{x}-\mbf{y}))\big)\,\psi_{\mbf{S}}(\mbf{y})\,\d\mbf{y}
=
\iint_{\re^n} K_R (\mbf{x}-\mbf{y})\,\psi_{\mbf{S}}(\mbf{y})\,\d\mbf{y},
\quad 
\mbf{x} \in \re^n.
\]
Recalling that $d=n+1$, and the (parabolic) scaling of our equations, we observe that
\begin{align*}
	\begin{split}
	& |K_R(\mbf{x})|\lesssim \|\mbf{x}\|^{-d-1}\,\mathbbm{1}_{(Q_{R}(0))^c}(\mbf{x}),
	\\[2pt]
	&|K_R(\mbf{x})-K_R(\mbf{x}')|\lesssim \frac{R}{(\|\mbf{x}\|+R)^{d+2}}, \qquad\text{if }\|\mbf{x}-\mbf{x}'\|\lesssim R,
	\\[2pt]
	&\iint_{\re^n} K_R(\mbf{x})\,\d\mbf{x}=0, \qquad\text{(i.e., $\mathcal{E}^R \mathbbm{1}´=0$).}
	\end{split}
\end{align*}
These estimates imply that if $\|\mbf{x}-\mbf{x'}\|\lesssim R$, using that $\psi_\sbf$ is $\Lip(1, 1/2)$ (see Lemma~\ref{psiSlip.lem}), then
\begin{multline*}
	\big|\mathcal{E}^R \psi_{\mbf{S}}(\mbf{x})-\mathcal{E}^R \psi_{\mbf{S}}(\mbf{x}')\big|
=
\Big|
\iint_{\re^n} \big(K_R(\mbf{x}-\mbf{y})-K_R(\mbf{x}'-\mbf{y})\big)\,\big(\psi_{\mbf{S}}(\mbf{y})-\psi_{\mbf{S}}(\mbf{x})\big)\,\d\mbf{y}
\Big|
\\
\lesssim\,
\iint_{\re^n} \frac{R}{(\|\mbf{x}-\mbf{y}\|+R)^{d+2}}\,\|\mbf{y}-\mbf{x}\|\,\d\mbf{y}
\, \lesssim\,
\iint_{\re^n} \frac{R}{(\|\mbf{x}-\mbf{y}\|+R)^{d+1}}\,\d\mbf{y}
\lesssim 1.
\end{multline*}
Thus, if $Q = Q_R(\mbf{x}_Q)$, we define $C_Q:=\mathcal{E}^R \psi_{\mbf{S}}(\mbf{x}_Q)$ and $\mathcal{D}_t^Q \psi_{\mbf{S}}(\mbf{y}): = 
\mathcal{D}_t^R \psi_{\mbf{S}}(\mbf{y})$, we have, for $ \mbf{y} \in Q$,
\begin{equation*}
	 \big|\mathcal{D}_t \psi_{\mbf{S}}(\mbf{y})-C_Q\big|  =
\big| \mathcal{D}_t^R \psi_{\mbf{S}}(\mbf{y})
+ \mathcal{E}^R \psi_{\mbf{S}}(\mbf{y})
	 -\mathcal{E}^R \psi_{\mbf{S}}(\mbf{x}_Q)\big| 
	 \lesssim 
	\big|\mathcal{D}_t^R \psi_{\mbf{S}}(\mbf{y}) \big| + 1
	=
	\big|\mathcal{D}_t^Q \psi_{\mbf{S}}(\mbf{y}) \big| + 1.
\end{equation*}
Therefore, to show \eqref{JShyp.eq} for some $Q \in \mathbb{D}$, it suffices to show that there exists $M_\star > 0$ such that
\begin{equation*}
 \big|\big\{
	\mbf{y}\in Q: |\mathcal{D}^Q_t\psi_\sbf(\mbf{y})|>M_\star
	\big\}\big|
	\le
	(1/3)\,|Q|,
\end{equation*}
In fact, for Case 1 cubes we will show, for some $C_\star \le M_\star/3$, the stronger condition (by Chebyshev)
\begin{equation}\label{BmoimpJSlocalized.eq}
 \frac{1}{|Q|} \iint_{Q} |\mathcal{D}^Q_t\psi_\sbf(\mbf{y})| \, \d\mbf{y} \le C_\star,
\end{equation}
and for Case 2 and 3 cubes, we will show the even stronger condition
\begin{equation}\label{LinftimpJSlocalized.eq}
\esssup_{\mbf{y} \in Q} \, |\mathcal{D}^Q_t\psi_\sbf(\mbf{y})| \le C_\star.
\end{equation}
.

\subsection{\for{toc}{\small}Case 3 cubes}

Recall that $Q$ falls in Case 3 when  $Q \not\subseteq 50Q(\sbf)$ and 
\begin{equation}\label{cs3remind.eq}
\dist(Q,4Q(\sbf)) \ge 6\diam(Q).
\end{equation}
Also recall that
\[\psi_{\sbf}(\mbf{x}) = \psi^*_{\sbf}(\mbf{x}) \, \phi(\mbf{x}), 
\quad \text{where} \quad 
|\phi| \le \mathbbm{1}_{4Q(\sbf)}.\]
Moreover, for $Q = Q_R(\mbf{x}_Q)$ the object $\mathcal{D}^Q_t\psi_\sbf$ has kernel $V_R$ supported in $Q_{2R}(0)$. Therefore, by \eqref{cs3remind.eq} the integrand in the definition of $\mathcal{D}^Q_t\psi_\sbf$ is zero in $Q$ and \eqref{LinftimpJSlocalized.eq} holds with $C_\star = 0$. 

\subsection{\for{toc}{\small}Proving the localized estimate for Case 2 cubes}
Here we prove that \eqref{LinftimpJSlocalized.eq} holds in Case 2. Recall that Case 2 cubes satisfy $Q \subseteq 50Q(\sbf)$; and whenever $Q_0 \in \sbf$ is such that $Q \subseteq 50Q_0$, it holds $\diam(Q) < \eps_0 \diam(Q_0)$.
Here $\eps_0$ is a small parameter at our disposal. We shall use it now to conclude that in Case 2, for $\mbf{y} \in Q$ the function $h$ is much larger than the diameter of the cube $Q$ and therefore we can conclude that the portion of $\psi_{\sbf}$ above $Q$ has an estimate on its (classical, full order) time-derivative that is sufficient for proving \eqref{LinftimpJSlocalized.eq}. The first step is the following lemma.

\begin{lemma}\label{tderbdscs2.lem}
If $\eps_0$ is sufficiently small (depending only on $n$), and $Q$ is a Case 2 cube, then 
\[h(\mbf{x}) \gtrsim \ell(Q), \quad \forall \mbf{x} \in 10Q.\]
\end{lemma}
\begin{proof}
Let $\mbf{x} \in 10Q$ and suppose, for the sake of contradiction (using Lemma~\ref{regdistlem.lem}), that $d(\mbf{x}) < \diam(Q)$.
Then, by definition of $d$, there exists $Q' \in \sbf$ such that
\[\dist(\mbf{x},Q') + \diam(Q') \le \diam(Q)\]
Since $\diam(Q) \le \eps_0 \diam(Q(\sbf))$ (just use the definition of Case 2 with $Q_0 = Q(\sbf)$), with $\eps_0 \ll 1$, we can find an ancestor $\widetilde{Q}$ of $Q'$ with $\widetilde{Q} \in \sbf$ such that 
\[\diam(Q) \le \diam(\widetilde{Q}) \lesssim \diam(Q),\]
where the implicit constant in the last inequality depends only on dimension (to find such $\widetilde{Q}$, just traverse one by one the ancestors of $Q'$ until the first time that $\ell(\widetilde{Q}) \geq \ell(Q)$). Note that since $Q' \subseteq \widetilde{Q}$, we also have $\dist(\mbf{x},\widetilde{Q} ) \leq \dist(\mbf{x}, Q') \le \diam(Q)$. 
Then it is clear that $Q \subseteq 50\widetilde{Q}$, and if we choose $\eps_0$ sufficiently small, $\diam(Q) \ge \eps_0 \diam(\widetilde{Q})$ (see the last display), which contradicts the fact that $Q$ is in Case 2 (using $Q_0 = \widetilde{Q} \in \sbf$). 
\end{proof}

Therefore, if we are given a Case 2 cube $Q = Q_R(\mbf{x}_Q)$, we can successfully control the localized half-order time derivative $\mathcal{D}_t^Q \psi_\sbf= \mathcal{D}_t^R \psi_{\sbf}$ in \eqref{localdtdef}. Indeed, using Lemmas \ref{estwhenhgtrzero.lem} and \ref{tderbdscs2.lem}, and the support properties of $\Phi_R$, we may integrate by parts in \eqref{localdtdef} to obtain for $\mbf{x} \in Q$
\begin{align*}
|\mathcal{D}_t^R \psi_{\sbf}(\mbf{x})| &= \left|\iint_{\re^n} \big(V(\mbf{x}-\mbf{y})\,\Phi_R(\mbf{x}-\mbf{y})\big)\,\partial_t\psi(\mbf{y})\,\d\mbf{y}\right|
\\ &\lesssim \frac{1}{\diam(Q)} \iint_{Q_{2R}(\mbf{x})} |V(\mbf{x}-\mbf{y})| \, \d\mbf{y} \lesssim  \frac{1}{\diam(Q)} \iint_{Q_{2R}(0)} \|\mbf{x}\|^{-n}\, \d\mbf{x} \lesssim 1,
\end{align*}
where we used \eqref{vstdbd.eq}) and the fact that $\|\cdot\|$ is the parabolic metric. This shows \eqref{LinftimpJSlocalized.eq} in Case 2.

\subsection{\for{toc}{\small}Modified level sets and square function estimates} \label{subsec:sqfunest}

After all the previous (sub)sections, to finish the proof of Proposition~\ref{regular.prop}, we are only left with obtaining \eqref{BmoimpJSlocalized.eq} for Case 1 cubes. In fact, this is the most challenging case, but luckily it is also the closest to \cite{BHMN1}. So before obtaining \eqref{BmoimpJSlocalized.eq}, let us put together some square function estimates that we essentially obtained in the previous sections.

\begin{proposition}\label{psinoheart.prop}
	Let $Q_0 \in \sbf$. Then it holds
	\begin{equation*}
		\iiint_{{\mathcal{S}(Q_0)}}\big(
		|r\,\partial_s \psi(r;y,s)|^2+|r\,\nabla_{y,r}^2 \psi(r;y,s)|^2+ |r^2 \,\nabla_{y,r}\partial_s \psi(r;y,s)|^2\big)\,\frac{\d r}{r}{\d y\d s}
		\lesssim |Q_0|.
	\end{equation*}
\end{proposition}
\begin{proof}
	With all the work that we have already done in the previous sections, the proof of this result is just a tiny fraction of the proof of \cite[Proposition 5.1]{BHMN1}. Let us include the argument for completeness. For the time being, consider the first term. Using Lemma~\ref{lemma:estimates:G-star}, we obtain 
	\[ 
	\iiint_{{\mathcal{S}(Q_0)}}
	|r\,\partial_s \psi(r;y,s)|^2
	\,\frac{\d r}{r}{\d y\d s}
	\lesssim
	\iiint_{{\mathcal{S}(Q_0)}}
	|r\,(\partial_s u)(\mbf{\Psi}(r;y,s))|^2
	\,\frac{\d r}{r}{\d y\d s}.
	 \]
	Now perform the change of variables $(y_0, y, s) = \mbf{Y} = \mbf{\Psi}(r;y,s) = (\psi(r; y, s), y, s)$. The domain of integration changes to $\Omega_{\sbf(Q_0)}$ by Lemma~\ref{map_sawtooths.lem}. The Jacobian of the transformation, which is $|\partial_r \psi(r; y, s)|^{-1}$, which by \eqref{partialr.eq} equals $|(\partial_{y_0} u)(\mbf{\Psi}(r; y, s))| \lesssim 1$ after using Lemma~\ref{ptwiseestinsaw4der.lem}. Moreover, $r = u(\mbf{\Psi}(r; y, s)) = u(\mbf{Y}) \approx_{M_2} \dist(\mbf{Y}, \pom)$ because $\mbf{Y} \in \Omega_{\sbf(Q_0)} \subset \Omega_{\sbf}$, so we can apply \eqref{M2defa.eq} since $\mbf{S}$ only contains cubes with side-length smaller than $2^{-10N}\ell(Q(\sbf^*)) = K^{-10}\ell(Q(\sbf^*))$ (see Lemma~\ref{fkwhsa.lem} or the introduction of Section~\ref{construction.sec}). All these considerations yield 
	\[ 
	\iiint_{{\mathcal{S}(Q_0)}}
	|r\,\partial_s \psi(r;y,s)|^2
	\,\frac{\d r}{r}{\d y \d s}
	\lesssim 
	\iiint_{{\Omega_{\sbf(Q_0)}}}
	|\dist(\mbf{Y}, \pom)\,\partial_s u(\mbf{Y})|^2
	\,\frac{\d \mbf{Y}}{\dist(\mbf{Y}, \pom)}
	\lesssim 
	|Q(\sbf)|,
	\]
	where in the last estimate we have used Lemma~\ref{impsfnest.lem} (letting $N_3 \to \infty$ using the monotone convergence theorem). This ends the proof for the first term of the statement, and the same procedure and all the estimates in Lemma~\ref{impsfnest.lem} yield the result for the remaining terms. 
\end{proof}

Note here a difference with \cite{BHMN1}: we need to obtain Proposition~\ref{psinoheart.prop} in ${\mathcal{S}(Q_0)}$ for every $Q_0 \in \sbf$ since our construction of the approximating graphs $\psi_\sbf$ is more intricate, and we need to consider many cubes in Case 1, not only $Q(\sbf)$ as in \cite{BHMN1}. Hence the importance of working in Section~\ref{sec:step2} with arbitrary subregimes (which we can now localize to be below any given $Q_0$) like in \eqref{locmnsqfnest.eq}, as opposed to only dealing with the original scale like in \eqref{mnsqfnest.eq}.

Actually, we will use a slightly modified version of the previous proposition. First we define
\begin{equation}\label{psiheart}
	\psi^{\heart}(r;x,t) := \psi(r;x,t) \, \phi(x,t),
\end{equation}
where $\phi \in C^\infty_c(\RR^n)$ is a cut-off adapted to $Q(\sbf)$ satisfying \eqref{cutoffqsest1.eq} and \eqref{cutoffqsest2.eq}. Note that
\begin{equation*}
	\psi^{\heart}(h(x,t);x,t) = \psi_\sbf(x,t),
\end{equation*}
so $\psi^\heart$ somehow ``interpolates'' between $\psi$ (with $r=0$) and $\psi_\sbf$ (with $r = h(x, t)$).

\begin{proposition}\label{psiheartbd.prop}
	Let $Q_0 \in \sbf$. Then it holds
	\begin{equation}\label{psiSboundheart.eq} 
		\iiint_{{\mathcal{S}(Q_0)}}\big(
		|r\,\partial_s \psi^{\heart}(r;y,s)|^2+|r\,\nabla_{y,r}^2 \psi^{\heart}(r;y, s)|^2+ |r^2 \,\nabla_{y,r}\partial_s \psi^{\heart}(r;y, s)|^2\big)\,\frac{\d r}{r}{\d y\d s}
		\lesssim |Q_0|.
	\end{equation}
\end{proposition}
\begin{proof} 
	We are going to leverage Proposition \ref{psinoheart.prop} heavily. Let us note that since $0 \le \phi \le 1$, when we expand the terms in \eqref{psiSboundheart.eq} out using the product rule, the only terms not handled by Proposition \ref{psinoheart.prop} are those where at least one derivative lands on $\phi$. Because $\mathbbm{1}_{2Q(\sbf)} \le \phi \le \mathbbm{1}_{4Q(\sbf)}$ these additional terms are zero unless $(y,s) \in 4Q(\sbf) \setminus 2Q(\sbf)$. In such case, by definition of $d$,
	\begin{equation}\label{hxtlwbdheart.eq}
		h(y,s) \approx d(y,s) \approx \diam(Q(\sbf)), \quad \forall (y,s) \in 4Q(\sbf) \setminus 2Q(\sbf).
	\end{equation}
	Therefore, recalling the definition of $\mathcal{S}(Q_0)$ in \eqref{fncySdef.eq}, to control the terms not handled by Proposition \ref{psinoheart.prop}, it is enough to obtain a bound of the form
	\begin{equation*}
		\iint_{500Q_0} \int_{c \diam(Q(\sbf))}^{10^{10}C_n\diam(Q_0)}
		\!\!\!\! 
		\big(
		|r\,\partial_s \psi^{\heart}(r; y,s)|^2+|r\,\nabla_{y,r}^2 \psi^{\heart}(r;y,s)|^2+ |r^2 \,\nabla_{y,r}\partial_s \psi^{\heart}(r;y,s)|^2\big)\,\frac{\d r}{r}{\d y\d s}
		\lesssim |Q_0|,
	\end{equation*}
	where $c$ is implicitly defined by \eqref{hxtlwbdheart.eq}.
	Now, using Lemma \ref{lemma:estimates:G-star}, the estimates on the derivatives of $\phi$ (see \eqref{cutoffqsest2.eq}), and the fact that $|\psi(y, s)| \lesssim \diam(Q(\sbf))$ for $(y, s) \in 500Q(\sbf)$ (recall the normalization in \eqref{psistrSdef.eq}, and the bounds for our regions in Lemma~\ref{map_sawtooths.lem} and \eqref{eq:U_Q_in_balls}), we get
	\[|r\,\partial_s \psi^{\heart}(r;y,s)|^2+|r\,\nabla_{y,r}^2 \psi^{\heart}(r;y,s)|^2+ |r^2 \,\nabla_{y,r}\partial_s \psi^{\heart}(r;y,s)|^2 \lesssim 1
	\]
	for every $(r, y, s)$ in the domain of integration in the previous display, which readily gives the desired estimate (recalling that $\diam(Q_0) \leq \diam(Q(\sbf))$ because $Q_0 \in \sbf$).
\end{proof}

\subsection{\for{toc}{\small}Treating the Case 1 cubes with methods from \cite{BHMN1}}\label{Case1.ssect}

To finish handling the Case 1 cubes, we will largely follow the ideas in \cite{BHMN1}. In fact, in some sense we have an easier task, because in \cite{BHMN1} the authors needed to more closely relate the object $\mathcal{D}_t^R\psi_\sbf$ to $\mathcal{D}_t^R\psi$ (see e.g. \cite[``Step 2"]{BHMN1}), whereas we are proving a Coronization and using Lemma \ref{coronaenough.lem}.

So fix a Case 1 cube $Q = Q_R(\mbf{x}_Q)$. Our goal is to show \eqref{BmoimpJSlocalized.eq}. Since $Q \subseteq 50Q_0$ for some $Q_0 \in \sbf$ with $\diam(Q) \approx_{\eps_0} \diam(Q_0)$ we may (and do) assume that $Q = 50Q_0$ for $Q_0 \in \sbf$. We are going to work with a smoothed, cut-off version of $\psi(r;x,t)$. To begin, we let $\zeta \in C_0^\infty([-1,1])$ be an even function with $\mathbbm{1}_{[-1/2,1/2]} \le \zeta \le \mathbbm{1}_{[-1,1]}$, and set
\begin{equation}\label{prdef1}
	p(x,t) := c_n \, \zeta(x_1)\, \zeta(x_2)\dots \zeta(x_{n-1}) \, \zeta(t), \quad \forall \, \mbf{x} = (x,t) = (x_1,\dots,x_{n-1},t) \in \rn,
\end{equation}
where $c_n$ is chosen so that $\iint_{\rn} p(x,t) \, \d x \, \d t = 1$.  We define, for $(x,t) \in \rn$,
\begin{equation}\label{prdef2}
	P_r f(x,t) := (p_r \ast f)(x,t),
\end{equation}
where $p_r(x,t) := r^{-n -1}p(x/r,t/r^2)$,
Thus, $P_r$ is a nice parabolic approximation to the identity on $\rn$.

Since $h$ is Lip(1,1/2) with $\|h\|_{\Lip{(1,1/2)}} \lesssim_n 1$ (Lemma~\ref{regdistlem.lem}), there exists $\gamma \ll_n 1/100$ such that
\begin{equation} \label{eq:approx_id}
	|h(\mbf{x}) - P_{\gamma r}h(\mbf{x})| \le r/4,\quad \; \forall r > 0, \; \mbf{x} \in \rn.
\end{equation}
The estimate \eqref{eq:approx_id} implies 
\begin{equation}\label{eqprh1}
	r + P_{\gamma r}h(\mbf{x}) \ge 3r/4 + h(\mbf{x}) > h(\mbf{x}), \quad \forall \, r > 0, \; \mbf{x} \in \rn.
\end{equation}

Now we define the smoothed cut-off version of $\psi(r;x,t)$ (with $\psi^\heart$ defined in \eqref{psiheart})
\begin{equation}\label{psitildedef.eq}
	\widetilde{\psi}_r(\mbf{x}) 
	:=
	\widetilde{\psi}(r; \mbf{x}) 
	:=
	\psi^\heart(r+ P_{\gamma r} h(\mbf{x}); \mbf{x}) 
	,
	\qquad 
	\mbf{x} \in 500Q_0, 0 \leq r \leq 10\diam(Q).
\end{equation}
We will leave $\gamma = \gamma(n)$ as a parameter we can adjust.
Let us make two observations. We have that
\begin{equation}\label{smthagreatzero.eq}
\widetilde{\psi}(0;\mbf{x})  = \psi_{\sbf}(\mbf{x}).
\end{equation}
We also want to observe that in the domain where we consider $\widetilde{\psi}_r(\mbf{x})$, the function only pulls information for $\psi^\heart$ in the region $\mathcal{S}(Q_0)$, where we have established a local square function estimate. In particular, we want to confirm that (recalling that $Q = 50Q_0$)
\begin{equation}\label{styinSQ.eq}
\mbf{x} \in 500Q_0, \;\; 0 \leq r \leq 10\diam(Q)
\quad \implies \quad 
(r+ P_{\gamma r} h(\mbf{x}), \mbf{x}) \in \mathcal{S}(Q_0).
\end{equation}
Indeed, inspecting the definition of $\mathcal{S}(Q_0)$ (see \eqref{fncySdef.eq}), we would want to confirm that
\[h(\mbf{x})/10 < r+ P_{\gamma r} h(\mbf{x}) < 10^{10}C_n\diam(Q_0).\]
The first inequality follows from \eqref{eqprh1}.  For the second, the definition of $d$ gives that $d(\mbf{x}) \le 501\diam(Q_0)$ for $\mbf{x} \in 500Q_0$, so by \eqref{eq:approx_id} and the properties of $h$ (see \eqref{regular-dist}) it holds
\[r+ P_{\gamma r} h(\mbf{x}) \le \frac{5}{4}r + h(x) \le \frac{50}{4} \diam(Q) + 501C_n \diam(Q_0) < 10^{10}C_n\diam(Q_0),\]
(where we again recall that $Q = 50Q_0$) which finishes the proof of \eqref{styinSQ.eq}.

Recall we are trying to prove \eqref{BmoimpJSlocalized.eq}. By H\"older's inequality it suffices to show
\begin{equation*}
\iint_{Q} |\mathcal{D}^Q_t\psi_\sbf(\mbf{y})|^2 \, \d\mbf{y} \le C_\star|Q|^{1/2}.
\end{equation*}
By duality and density, to prove this, it suffices to show, for $f \in C_c^\infty(Q)$ with $\|f\|_{L^2(Q)} \le 1$, that
\begin{equation}\label{BMO2cont.eq}
	\biggl |\iint \mathcal{D}_t^Q \psi^{\mbf{S}}(\mbf{x}) f(\mbf{x}) \, \d\mbf{x}\biggr | 
	\,=:\, \biggl |\iint \mathrm{I}_\mathrm{P}^Q \psi^{\mbf{S}}(\mbf{x})\, \partial_t f(\mbf{x}) \, \d\mbf{x}\biggr | \,\lesssim\, |Q|^{1/2}\,,
\end{equation}
where we have interpreted the $\partial_t$ in the definition of $\mathcal{D}_t^Q := \mathcal{D}_t^R$ (see \eqref{localdtdef}) in the weak sense.
Now, following \cite{BHMN1}, we use \eqref{smthagreatzero.eq}, introduce an approximate identity on $f$ and 
integrate several times by parts (we will explain this below)
to obtain
\begin{multline}\label{IBPfordthalf.eq}
	- \iint \mathrm{I}_\mathrm{P}^Q \psi^{\mbf{S}}(\mbf{x})\, \partial_t f(\mbf{x}) 
	\d\mbf{x}  \,=\, - \iint
	\mathrm{I}_\mathrm{P}^Q \tilde{\psi}(0; \mbf{x})\, \partial_t f(\mbf{x})
	\d\mbf{x}
	\\ 
	=- \int_0^{\diam(Q)} \iint_{2Q} 
	\partial_r[\mathcal{D}_t^Q \tilde{\psi}(r; \mbf{x}) 
	P_{\gamma r} f(\mbf{x})]\, \d\mbf{x} \, \d r \,+\, b_1
	\\  = \, \int_0^{\diam(Q)} \iint_{2Q} 
	\partial_r^2[\mathcal{D}_t^Q \tilde{\psi}(r; \mbf{x}) 
	P_{\gamma r} f(\mbf{x})]  \, \d\mbf{x}\, r \d r \,+\, b_1 - b_2
	\, =:\, I + b_1 - b_2,
\end{multline}
where the boundary terms $b_1$ and $b_2$ are defined by
\begin{equation}\label{b1err.eq}
	b_1 := 
	- \! \iint_{2Q} \!\! \mathrm{I}_\mathrm{P}^Q \tilde{\psi}(\diam(Q); \mbf{x}) \partial_t P_{\gamma \diam(Q)} f(\mbf{x}) \d\mbf{x}
	=
	\iint_{2Q} \!\! \mathcal{D}_t^Q \tilde{\psi}( \diam(Q), \mbf{x}) P_{\gamma \diam(Q)} f(\mbf{x}) \d\mbf{x},
\end{equation}
where we have integrated by parts in $t$ recalling that $\mathcal{D}_t^Q = \partial_t \circ \mathrm{I}_\mathrm{P}^Q$ (see \eqref{localdtdef}), and in turn
\begin{equation}\label{b2err.eq}
	b_2 :=  \iint_{2Q} \partial_r[\mathcal{D}_t^Q \tilde{\psi}(r; \mbf{x}) P_{\gamma r} f(\mbf{x})]\big|_{r = \diam(Q)}\, \diam(Q) \, \d\mbf{x}.
\end{equation}

Let us explain \eqref{IBPfordthalf.eq} more carefully. First, note the fact that all the integrals are taken over $2Q = 100Q_0$ because $\gamma \ll_n 1/100$ and $f$ being supported in $Q$ imply that $P_{\gamma r} f$ is supported in $2Q$ for $r < \diam(Q)$. Then, the first equality in \eqref{IBPfordthalf.eq} is just \eqref{smthagreatzero.eq}. In turn, in the second inequality, we first use the fundamental theorem of calculus in the variable $r$ (so that we obtain the boundary term $b_1$), and then we integrate by parts in $t$ in the main term, passing the $\partial_t$ from $f$ to $\mathrm{I}_\mathrm{P}^Q$ (recall $\mathcal{D}_t^Q = \partial_t \circ \mathrm{I}_\mathrm{P}^Q$). This actually makes sense because we will show in \eqref{L3est2.eq} below that $\mathcal{D}_t^Q \tilde{\psi}(r; \mbf{x})$ is well defined for $r>0$. Lastly, the third equality is an elementary 1D integration by parts in $r$.

To continue, we claim, and will prove in 
(sub)section \ref{tech} below, that
\begin{equation}\label{b12errest.eq}
	|b_1| + |b_2| \lesssim |Q|^{1/2}.
\end{equation}
This leaves the contribution of the main term
$I$, in which we distribute the $r$-derivatives:
\begin{align}\label{Iintothree.eq}
	I =  \int_0^{\diam(Q)}\! \iint_{2Q} \mathcal{D}_t^Q  \partial_r^2\tilde{\psi}(r; \mbf{x}) \, P_{\gamma r} f(\mbf{x}) \d\mbf{x}\,r  \d r 
	\,+\, 2 \int_0^{\diam(Q)}\! \iint_{2Q}  \mathcal{D}_t^Q \partial_r \tilde{\psi}(r; \mbf{x}) \, \partial_r P_{\gamma r} f(\mbf{x}) \d\mbf{x}\, r \d r &
	\nonumber
	\\
	+\,\, \int_0^{\diam(Q)} \!\iint_{2Q}  \mathcal{D}_t^Q \tilde{\psi}(r; \mbf{x}) \, \partial_r^2 P_{\gamma r} f(\mbf{x})   \d\mbf{x}\, r  \d r
	\,=:\, I_1 + I_2 + I_3. & 
\end{align}

To deal with these terms, we will use Littlewood-Paley theory. So let us define $\mathcal{Q}^{(j)}_r = 
\mathcal{Q}^{(j,Q)}_r $ by 
\[
\mathcal{Q}^{(1)}_r \,:=\,  r \,\mathcal{D}_t^Q P_{\gamma r}, 
\quad \mathcal{Q}^{(2)}_r \,:=\,  \mathrm{I}_P^Q \,\partial_r P_{\gamma r} , 
\quad \mathcal{Q}^{(3)}_r \,:= \,r \,\mathrm{I}_P^Q \,\partial_r^2 P_{\gamma r}\,.
\]
Then, recall that $\mathrm{I}_P^Q$ is the smoothly truncated fractional integral operator of order $1$ (see \eqref{IPRdef.eq}), and 
is self-adjoint. This implies (by an integration by parts in $t$, at least when we test against functions with compact suppport, as is our case) that 
$\mathcal{D}_t^Q = \partial_t \circ \mathrm{I}_P^Q$ is skew-adjoint.
This implies that we can continue estimating $I_1$ in \eqref{Iintothree.eq} as follows:
\begin{multline}\label{estI1.eq}
	|I_1| = \left|\int_0^{\diam(Q)} \!\iint_{2Q} r \partial_r^2\tilde{\psi}(r; \mbf{x}) \mathcal{Q}^{(1)}_r f(\mbf{x})  \, \d\mbf{x}\,  \frac{\d r}{r}\right|
	\\  
	\le \left( \int_0^{\diam(Q)}\! \iint_{2Q} 
	|r \partial_r^2\tilde{\psi}(r; \mbf{x})|^2 \, \d\mbf{x}\,  
	\frac{\d r}{r}\right)^{1/2} \left(\int_0^{\diam(Q)} \!
	\iint_{2Q} |\mathcal{Q}^{(1)}_r f(\mbf{x})|^2  \, \d\mbf{x}\,  \frac{\d r}{r}\right)^{1/2},
\end{multline}

To get a bound for the last term, note that by parabolic Littlewood-Paley theory\footnote{Estimate
	\eqref{IRPLPTest.eq} may be proved via 
	Plancherel's theorem.  We omit the standard argument.}
we have
\begin{equation}\label{IRPLPTest.eq}
	\int_0^\infty\! \iint_{\rn} |\mathcal{Q}^{(1)}_rf(\mbf{x})|^2 + |\mathcal{Q}^{(2)}_rf(\mbf{x})|^2 + |\mathcal{Q}^{(3)}_rf(\mbf{x})|^2 \, \d\mbf{x} \, \frac{\d r}{r} \,\lesssim\,  \|f\|_{L^2}^2 \,\le\, 1,
\end{equation}
where the implicit constant is independent of $Q$.
Note that the cancellation for $\mathcal{Q}^{(1)}_r $ comes from this $t$-derivative.

Therefore, to complete the estimate for $I_1$ in \eqref{estI1.eq}, we only lack the following lemma.
\begin{lemma}\label{atildeprop.lem} Let $\tilde \psi$ be defined as 
	in \eqref{psitildedef.eq} and $Q_0 \in \sbf$. Then, if $Q = 50Q_0$, it holds
	\begin{equation}\label{L3est2.eq}
		|r\partial_r^2 \tilde{\psi}(r;\mbf{x})| + |r\partial_t \tilde{\psi}(r; \mbf{x})|  + |r^2 \partial_t \partial_r \tilde{\psi}(r;\mbf{x})| \lesssim 1,
		\quad \forall \, r\in (0, 10\diam(Q)),\, \mbf{x}\in 10Q\,,
	\end{equation}
	and
	\begin{equation}\label{L3est1.eq}
		\int_{0}^{\diam(Q)}\! \iint_{2Q} 
		\left(|r\partial_r^2 \tilde{\psi}(r;\mbf{x})|^2 
		+ |r\partial_t \tilde{\psi}(r; \mbf{x})|^2  +
		|r^2 \partial_t \partial_r \tilde{\psi}(r;\mbf{x})|^2 \right) 
		\d\mbf{x}  \frac{\d r}{r} \,\lesssim\, |Q|.
	\end{equation}
	\end{lemma}

We postpone the proof of the lemma to (sub)section \ref{tech}: indeed, it is a consequence of Proposition \ref{psiheartbd.prop} (and Lemma \ref{PgrhLPprop.lem} in the next subsection). Taking it for granted, we can quickly finish the proof of \eqref{BMO2cont.eq}. Indeed, just put together \eqref{estI1.eq}, \eqref{IRPLPTest.eq} and \eqref{L3est1.eq} to get the desired estimate for $I_1$ (and similarly for $I_2$ and $I_3$): 
\[|I_1| + |I_2| + |I_3| \lesssim  |Q|^{1/2}.\]
Therefore, combining our estimates for $I_1,I_2,I_3$ here, and $b_1,b_2$ in \eqref{b12errest.eq}, we conclude that \eqref{BMO2cont.eq} holds, which finally shows \eqref{BmoimpJSlocalized.eq} for Case 1 cubes, modulo the proofs in the next subsection.

\subsection{\for{toc}{\small}Estimating the boundary terms $b_1,b_2$ and the proof of Lemma \ref{atildeprop.lem}.}\label{tech}

To finish the proof of \eqref{BmoimpJSlocalized.eq} (and therefore of Theorems~\ref{main1.thrm} and \ref{main2.thrm}) we are only left with showing \eqref{b12errest.eq} and Lemma~\ref{atildeprop.lem}. For that, we will also follow \cite{BHMN1} very closely in this subsection. We start with the following result,
which is based on Lemma \ref{regdistlem.lem}. We refer to 
\cite[Lemma 2.8]{HL96} 
for its proof.

\begin{lemma}\label{PgrhLPprop.lem} Define $P_r$ as in 
	\eqref{prdef1}-\eqref{prdef2}. Then we have
	\begin{align*}\bigl| \frac{1}{r} (P_r - I) h(\mbf{x})  \bigr| + \bigl|\partial_r P_r \, h(\mbf{x}) \bigr| + \bigl|r^{2j + m - 1}\nabla_{x,r}^m \partial_t^j P_{r} h(\mbf{x})\bigr| \lesssim 1,
	\end{align*}
	where the implicit constant depends at most on $(n,m,j)$. Furthermore, for all $Q \subseteq \rn$,
	\begin{align*}
		\mathrm{(i)}&\quad  \int_{0}^{50\ell(Q)} \iint_{50Q}\bigl|\frac{1}{r} (P_r - I)\,  h(\mbf{x})  \bigr|^2\, \d\mbf{x} \, \frac{\d r}{r} \lesssim |Q|,\\
		\mathrm{(ii)}&\quad \int_{0}^{50\ell(Q)} \iint_{50Q} \left|\partial_r P_r h(\mbf{x})  \right|^2\, \d\mbf{x} \, \frac{\d r}{r}\lesssim |Q|,\\
		\mathrm{(iii)}&\quad \int_{0}^{50\ell(Q)} \iint_{50Q}  \big|r^{2j + m - 1}\nabla_{x,r}^m \partial_t^j P_{r} h(\mbf{x})  \big|^2 \, \d\mbf{x} \, \frac{\d r}{r} \lesssim |Q|,
	\end{align*}
	where in $\mathrm{(iii)}$ we require that $j,m\geq 0$, and that either 
	$j \ge 1$,  or that $ m \ge 2$. Again the implicit constants depend at most on $(n,m,j)$.
\end{lemma}

With this lemma in hand, we prove Lemma \ref{atildeprop.lem}.
\begin{proof}[Proof of Lemma \ref{atildeprop.lem}]
We start by proving \eqref{L3est2.eq}. We will only handle the first term to the left in \eqref{L3est2.eq}, as the rest of the terms can be handled analogously. We have
\[
\partial_r \tilde{\psi}(r;\mbf{x})
= 
\partial_r \left[
\psi(r+ P_{\gamma r}h(\mbf{x});\mbf{x})\phi(\mbf{x})\right] 
= 
(\partial_r \psi)(r+ P_{\gamma r}h(\mbf{x}),\mbf{x})
(1 + \partial_r P_{\gamma r}h(\mbf{x}))\phi(\mbf{x}),
\]
and hence
\begin{multline}\label{dr2tpsi}
	r\partial_r^2 \tilde{\psi}(r;\mbf{x})  
	= r(\partial_r^2 \psi)(r+ P_{\gamma r}h(\mbf{x});\mbf{x}) \,
	\left(1 + \partial_r P_{\gamma r}h(\mbf{x})\right)^2 
	\\ \, +\, (\partial_r \psi)(r+ P_{\gamma r}h(\mbf{x});\mbf{x}) 
	\, r\left(\partial_r^2 P_{\gamma r}h(\mbf{x})\right)\phi(\mbf{x}).
\end{multline}
The bound for $|r\partial_r^2 \tilde{\psi}(r;\mbf{x})|$ now 
follows from Lemma \ref{lemma:estimates:G-star}, specifically 
\eqref{eqn:estimates:G-star:1} and 
\eqref{eqn:estimates:G-star:3},  and 
Lemma \ref{PgrhLPprop.lem}.  In particular, \eqref{styinSQ.eq} justifies the appropriate use of  
Lemma \ref{lemma:estimates:G-star}. Note that for the remaining terms in \eqref{L3est2.eq}, there are derivatives landing on $\phi$, but these contribute terms that act well since $h(\mbf{x}) \approx \ell(Q(\sbf))$ where $\nabla_{x,t} \phi \neq 0$, as observed in the proof of Proposition \ref{psiheartbd.prop}.

Next we turn our attention to \eqref{L3est1.eq}, which is a little more delicate. Again, we will only handle the first term (in the integral) as the others terms can be handled analogously. First, we control a 
closely related expression, without the smoothing factor. 
Indeed, we observe that
\begin{multline*}
	\int_{0}^{\diam(Q)}\iint_{2Q} \!
	|r\partial_r^2 \psi^\heart(r +h(\mbf{x});\mbf{x})|^2 \, \d\mbf{x}  
	\frac{\d r}{r}
	= 
	\iint_{2Q}  \int_{h(\mbf{x})}^{h(\mbf{x})+ 
		\diam(Q)}
	|\partial_s^2 \psi^\heart( s;\mbf{x})|^2 \,\big(s - h(\mbf{x})\big)  \d s \d\mbf{x} 
	\\ \le \iiint_{\mathcal{S}_Q} 
	|\partial_s^2 \psi^\heart(s;\mbf{x})|^2 s \d\mbf{x} \d s \lesssim  |Q|\,,
\end{multline*}
where in the last step we used Proposition \ref{psiheartbd.prop}, 
and in the previous inequality, we used \eqref{styinSQ.eq} and $h \geq 0$. 
Now, we need to relate this to our smoothed version of $\psi$. With the preceding estimate in hand, we see that to obtain the bound
\[\int_{0}^{\diam(Q)}\iint_{2Q} |r\partial_r^2 \tilde{\psi}(r;\mbf{x})|^2\, \d\mbf{x} \, \frac{\d r}{r} \lesssim |Q|\]
in \eqref{L3est1.eq}, it is enough to prove
\begin{equation}\label{roughhsmoothhbd.eq}
	\int_{0}^{\diam(Q)}\iint_{2Q} 
	|r\partial_r^2[\tilde{\psi}(r;\mbf{x}) - \psi^\heart(r +h(\mbf{x});\mbf{x})]|^2 \, \d\mbf{x} \, \frac{\d r}{r} \lesssim |Q|.
\end{equation}

To do this, we first note that by \eqref{dr2tpsi}, Lemma 
\ref{lemma:estimates:G-star}, and Lemma \ref{PgrhLPprop.lem}, 
\[
r\partial_r^2\tilde{\psi}(r;\mbf{x}) =
r(\partial_r^2 \psi^\heart)\big(r + P_{\gamma r}h(\mbf{x});\mbf{x}\big)
\,+\, O\left(
\left|\partial_r P_{\gamma r}h(\mbf{x})\right|\,+\,r\left|\partial_r^2 P_{\gamma r}h(\mbf{x})\right|\right)\,,
\]
and observe further that the ``big-$O$" term 
may be handled via Lemma \ref{PgrhLPprop.lem}.
To treat the contribution to \eqref{roughhsmoothhbd.eq}
of the remaining term, we 
use the mean value theorem:  
for $(r,\mbf{x}) \in (0,\diam(Q)) \times {2Q}$, 
there exists $\tilde{r}$ between  $r + h(\mbf{x})$ and $r + P_{\gamma r}h(\mbf{x})$, such that
\begin{multline}\label{eesta}
	\big|(\partial_r^2 \psi^\heart)(r + P_{\gamma r}h(\mbf{x});\mbf{x}) 
	\,-\,\partial_r^2  \psi^\heart(r +h(\mbf{x});\mbf{x})\big| 
	\\ = |(P_{\gamma r} - I) \, h(\mbf{x})| \,
	|(\partial_r^3\psi)^\heart(\tilde{r};\mbf{x})|
	\lesssim \tilde{r}^{-2}|(P_{\gamma r} - I) \, h(\mbf{x})|
	\leq 
	r^{-2}|(P_{\gamma r} - I) \, h(\mbf{x})|.
\end{multline}
The use of \eqref{eqn:estimates:G-star:3} to 
derive the second-to-last inequality 
in this display may be justified 
by the fact that $\tilde{r}$ is between $r + h(\mbf{x})$ and 
$r + P_{\gamma r}h(\mbf{x})$; in particular,
$( \tilde{r},\mbf{x})$ lies in $\mathcal{S}(Q_0) \subseteq \mathcal{S}$. (We also remark that the cut-off function $\phi$ is independent of $r$ and $|\phi| \le 1$.)
Also, in the last inequality we have used the fact that $\tilde{r} \geq \min \{ r + h(\mbf{x}), r + P_{\gamma r}h(\mbf{x}) \} \geq r$ because $h \geq 0$. 
Multiplying \eqref{eesta} by $r$,
and using Lemma \ref{PgrhLPprop.lem} (i), we obtain \eqref{roughhsmoothhbd.eq}, thus completing the proof of the lemma. Note that in the (omitted) proofs of bounds for the second third terms in \eqref{L3est1.eq} that when derivatives land on the cut-off function these terms can be handled as in Proposition \ref{psiheartbd.prop}. 
\end{proof}

We are left with handling the ``boundary terms" $b_1$ and $b_2$, defined in \eqref{b1err.eq} and \eqref{b2err.eq}.
Let us estimate $|b_1|$ in \eqref{b1err.eq}. 
Recall that $\|f\|_{L^2} \leq 1$, and that, for $r>0$, the operator $P_{r}$ is bounded on $L^2$, uniformly in $r$. Hence
\begin{multline}\label{b1bdst.eq}
	|b_1|\,\le  \,\left( \iint_{2Q} |\mathcal{D}_t^Q 
	\tilde{\psi}(\diam(Q); \mbf{x})|^2 \, \d\mbf{x}\right)^{1/2} \left(\iint_{2Q} |P_{\gamma \diam(Q)} f(\mbf{x})|^2\, \d\mbf{x}\right)^{1/2}
	\\
	\lesssim \, \left( \iint_{2Q} |\mathcal{D}_t^Q 
	\tilde{\psi}(\diam(Q); \mbf{x})|^2 \, \d\mbf{x}\right)^{1/2} 
	\,  \lesssim \,\|\mathcal{D}_t^Q \tilde{\psi}(\diam(Q); \mbf{x})\|_{L^\infty} 
	\, |Q|^{1/2}\,.
\end{multline}
Next, recall that $\mathcal{D}_t^Q$ is a convolution operator with kernel $K^Q = \partial_t V_Q$, where $V_Q$ is the truncated version of the kernel of
$\mathrm{I}_P$, localized at scale $\diam(Q)$ 
(see section \ref{localizeandcase3.sect}). 
Hence,  passing the $t$-derivative onto $\tilde{\psi}(\diam(Q); \mbf{x})$ because $V_Q$ acts by convolution, 
and then using \eqref{L3est2.eq}, \eqref{vstdbd.eq} (and properties of $\phi$) we have
\begin{align*}
	\big|\mathcal{D}_t^Q \tilde{\psi}(\diam(Q); \mbf{x})\big| 
	&=
	 \big|V_Q \ast \partial_t \tilde{\psi} \, (\diam(Q); \mbf{x})\big|
	 \lesssim \frac{1}{\ell(Q)}\iint_{C_{10\diam(Q)}(\mbf{x})}\|\mbf{x} - \mbf{y}\|^{1-d} \, \d \mbf{y} \lesssim 1\,.
\end{align*}
Plugging the latter bound into \eqref{b1bdst.eq} gives 
$|b_1| \lesssim |Q|^{1/2}$, as desired. 
The proof in the case of $b_2$ in \eqref{b2err.eq} proceeds analogously 
and we omit the details. Having proved Lemma \ref{atildeprop.lem} and handling the boundary terms we have tied up all of the loose ends and proved Theorem \ref{main1.thrm} and Theorem \ref{main2.thrm}.


\appendix
\section{Cut-off function construction, Lemma \ref{cutofffnlem.lem}}\label{cutoff.sect}
In this section we construct the cut-off function in Lemma \ref{cutofffnlem.lem}. Here we use ideas in a forthcoming joint work of the first and last author, with J.M. Martell and K. Nystr\"om \cite{BHMN2}\footnote{The general scheme here was shown to the first author by J.M. Martell.}.

First, consider $P_r$ a parabolic approximate identity on $\mathbb{R}^{n+1}$, that is, we take
$\zeta \in C_0^\infty([-1,1])$ be an even function with $\mathbf{1}_{[-1/2,1/2]} \le \zeta \le \mathbf{1}_{[-1,1]}$ and set
\begin{equation}\label{prdef1'}
	p(x_0,x,t) := c_n\, \zeta(x_0)\, \zeta(x_1)\dots \zeta(x_{n-1})\, \zeta(t), \quad \forall \, \mbf{X} = (x_0,x,t) = (x_0, x_1,\dots,x_{n-1},t) \in \mathbb{R}^{n+1},
\end{equation}
where $c_n$ is chosen so that $\iint_{\rn} p(x_0,x,t) \, \d x_0 \, \d x \, \d t = 1$
and we define $P_r f(\mbf{X}) := p_r \ast f (\mbf{X})$, where $p_r(x_0,x,t) := r^{-n-2}p(x_0/r,x/r,t/r^2)$.

Now, for each $Q \in \dd$, we set, for $\eps = \eps(K) > 0$ sufficiently small depending on $K$
\[\tilde{\eta}_Q(x) := P_{\eps \ell(Q)} \mathbbm{1}_{U_Q^*}(x).\]
Then it is easy to verify (simply use elementary properties of convolutions and rescalings) that
\begin{equation}\label{tetasupp.eq}
	\mathbbm{1}_{U_Q^*} \lesssim \tilde\eta_Q \le  \mathbbm{1}_{U_Q^{**}},
	\quad \text{and} \quad 
	|\nabla \tilde{\eta}_Q| \, \ell(Q) + |\partial_t\tilde{\eta}_Q|\, \ell(Q)^{2} \lesssim 1,
\end{equation}
with implicit constants depending on $n$, $K$ and $\|\psi\|_{\Lip(1,1/2)}$. Let us omit these dependencies from now on.

Then, we set
\begin{equation*}
	\eta_Q := \frac{\tilde \eta_Q}{\sum_{Q' \in \dd} \tilde \eta_{Q'}}.
\end{equation*}
Note now that whenever $U_Q^{**} \cap U_{Q'}^{**} \neq \emptyset$, it holds $\ell(Q) \approx \ell(Q')$. Actually, using this and \eqref{eq:U_Q_in_balls}, it is easy to deduce that the $U_Q^{**}$ regions have bounded overlap (i.e., given a cube $Q \in \dd$, there is only a --uniformly-- bounded number of $Q' \in \dd$ with $U_Q^{**} \cap U_{Q'}^{**} \neq \emptyset$). Therefore, from \eqref{tetasupp.eq} it follows 
\begin{equation}\label{etaQest.eq} 
	\mathbbm{1}_{U_Q^*} \lesssim {\eta}_Q \le  \mathbbm{1}_{U_Q^{**}}, \quad \text { and } \quad 
	|\nabla{\eta}_Q|\, \ell(Q) + |\partial_t{\eta}_Q| \, \ell(Q)^{2} \lesssim 1.
\end{equation}
 Moreover, by construction,
\[\sum_{Q \in \dd} \eta_Q \equiv 1 \quad \text{ in } \Omega.\]

Now we define the cut-off function $\eta$.
For simplicity of notation, set $\sbf' := \sbf_{N_1,N_2}$ and define
\begin{equation}\label{cutoffdefine.eq}
\eta := \sum_{Q \in \widetilde{\sbf} } \eta_Q, 
\quad 
\text{where} 
\quad 
\widetilde{\sbf} := \big\{Q' \in \dd: \exists Q \in \sbf' \text{ such that } U_Q^{**} \cap U_{Q'}^{**} \neq \emptyset \big\}.
\end{equation}
Then it is easy to verify \eqref{iscutoff.eq}: just use \eqref{etaQest.eq} and the definitions in Subsection~\ref{whit.sec}. It is even easier to obtain \eqref{cutoffbds1.eq}: simply use \eqref{etaQest.eq} and the fact that at a given point, only finitely many $\eta_Q$ can be non-zero, which follows from the bounded overlap of the $U_Q^{**}$ regions.

Therefore, it remains to show \eqref{cutoffbds2.eq} and \eqref{cutoffbds3.eq}. By definition
\[\big\{\nabla_{\mbf{X}} \eta \neq 0 \big\} \subseteq \bigcup_{Q \in  \F_b} U_{Q}^{***}.\]
Using that $|U_Q^{***}| \approx \ell(Q)^{n+2}$ and \eqref{cutoffbds1.eq}, we easily check that \eqref{cutoffbds2.eq} implies \eqref{cutoffbds3.eq}. 

In order to prove \eqref{cutoffbds2.eq} we make an observation: 
\begin{equation} \label{boundarycube.eq}
Q \in \F_b 
\quad \implies \quad 
\exists \; \Qin \in \sbf', \; \Qout \in \dd \setminus \sbf'
\; \text{ such that } \;
U_Q^{***} \cap U_\Qin^{***} \neq \emptyset \neq U_Q^{***} \cap U_\Qout^{***}.
\end{equation}
Indeed, given $Q \in \F_b$, pick $\mbf{X} \in U_Q^{***}$ with $\nabla_{\mbf{X}} \eta(\mbf{X}) \neq 0$. Then, by \eqref{cutoffdefine.eq} and \eqref{etaQest.eq}, there exists $\Qout \in \widetilde{S}$ with $\mbf{X} \in U_\Qout^{**}$. Then, using the definition of $\widetilde{S}$, we may find $\Qin \in \sbf'$ with $U_\Qout^{**} \subseteq U_\Qin^{***}$. This gives the result for $\Qout$ because $\mbf{X} \in U_Q^{***} \cap U_\Qout^{***}$. After that, let us note that if $\Qout$ belonged to $\sbf'$, then $\mbf{X}$ would belong to $\Omega_{\sbf'}^{**}$, which is an open set where $\eta \equiv 1$ (this is \eqref{iscutoff.eq}), which would lead to a contradiction because then we would have $\nabla_{\mbf{X}} \eta(\mbf{X}) = 0$. This gives the result for $\Qin$.

In this case, let us note that since $U_Q^{***} \cap U_\Qin^{***} \neq \emptyset \neq U_Q^{***} \cap U_\Qout^{***}$, it holds that
\begin{equation}\label{QQpprop.eq}
	\ell(\Qin) \approx \ell(Q) \approx \ell(\Qout) \quad \text{ and } \quad \dist(Q, \Qin), \,\dist(Q, \Qout) \lesssim \ell(Q).
\end{equation}
Also, this implies that, given $Q' \in \sbf$, the amount of cubes for which $Q'$ is their $\Qin$ is bounded:
\begin{equation} \label{associationQQin.eq}
	\# \big \{ Q \in \F_b : \Qin = Q' \big \} \lesssim 1.
\end{equation}

We are going to break up into cases, depending on the properties of the $\Qout$ from \eqref{boundarycube.eq}. 
\begin{itemize}
\item $\F_1$ is the collection of cubes $Q \in \F_b$ such that $\ell(\Qout) > 2^{-N_1} \ell(Q(\sbf))$,
\item $\F_2$ is the collection of cubes $Q \in \F_b$ such that $\ell(\Qout) \le 2^{-N_1} \ell(Q(\sbf))$ and $\Qout \cap Q(\sbf) \neq \emptyset$, 
\item $\F_3$ is the collection of cubes $Q \in \F_b$ such that $\ell(\Qout) \le 2^{-N_1} \ell(Q(\sbf))$ and $\Qout \cap Q(\sbf) = \emptyset$.
\end{itemize} 

\textbf{Case $\F_1$.} If we have $Q \in \F_1$, then, after recalling that $\Qin \in \sbf' = \sbf_{N_1, N_2}$, we quickly realize that $\ell(\Qin) \leq 2^{-N_1}\ell(Q(\sbf))$. Then, since $\ell(\Qin) \approx \ell(\Qout)$ (see \eqref{QQpprop.eq}) and $\ell(\Qout) > 2^{-N_1}\ell(Q(\sbf))$ (for $Q \in \F_1$), we obtain that 
\[\ell(\Qin) \approx 2^{-N_1}\ell(Q(\sbf)).\]
Using this, \eqref{associationQQin.eq}, $\ell(\Qin) \approx \ell(Q)$ from \eqref{QQpprop.eq}, and the fact that $\Qin \subseteq Q(\sbf)$, we get
\[
\sum_{Q\in \F_1} |Q|
\lesssim 
\!\!\!\! \sum_{\substack{\Qin \in \sbf' \\ \ell(\Qin) \approx 2^{-N_1}\ell(Q(\sbf))}} \!\!\!\!|\Qin| 
\lesssim 
|Q(\sbf)|.
\]

Next, to treat the cubes in $\F_2$ and $\F_3$, let us fix some notation. Let $\mbf{x}_{\Qout}$ be the center of $\Qout$ and let $\Qout^\star \in \dd(\Qout)$ be such that $\mbf{x}_{\Qout} \in \Qout^\star$ and $\ell(\Qout^\star) = 2^{-10}\ell(\Qout)$. Then 
\begin{equation}\label{QstarwhitQp.eq}
	\dist(\Qout^\star, (\Qout)^c) \approx \ell(\Qout), 
	\qquad 
	\ell(\Qout^\star) \approx \ell(\Qout).
\end{equation}

\textbf{Case $\F_2$.} Let us start with $\F_2$.
Let $\mathcal{B}'$ denote the minimal cubes in $\sbf'$. If $Q \in \F_2$, then $\Qout \subseteq \widetilde{Q}$ for some $\widetilde{Q} \in \mathcal{B}'$ (just develop the definition of $Q \in \F_2$), and hence $\Qout^\star \subseteq \widetilde{Q}$. We claim that 
\begin{equation}\label{QstarwhitQtil.eq}
	\ell(\Qout^\star) \lesssim \dist(\Qout^\star, (\widetilde{Q})^c) \lesssim \ell(\Qout^\star).
\end{equation}
Taking this claim momentarily for granted, let us continue. Fix $\widetilde{Q} \in \mathcal{B}'$. Let us consider the family of cubes $Q_k \in \F_2$ for which $(Q_k)_{\text{out}} \subseteq \widetilde{Q}$. Then \eqref{QstarwhitQtil.eq} implies that the associated $\{ (Q_k)_{\text{out}}^\star \}_k$ have bounded overlap. Indeed, if $(Q_k)_{\text{out}}^\star \cap (Q_{k'})_{\text{out}}^\star \neq \emptyset$, then \eqref{QstarwhitQtil.eq} (and \eqref{QstarwhitQp.eq}) would imply that $\ell(Q_k) \approx \ell(Q_k^\star) \approx \ell(Q_{k'}^\star) \approx \ell(Q_{k'})$, which would easily imply that $\dist((Q_k)_{\text{out}}^\star, (Q_{k'})_{\text{out}}^\star) \lesssim \ell((Q_k)_{\text{out}})$. Therefore, for a fixed $k$, the bounded overlap follows by the above facts:
\[\#\big\{Q_{k'}^\star: Q_{k'}^\star \cap Q_{k}^\star \neq \emptyset\big\} \le \#\big\{Q_{k'} : \ell(Q_{k'}) \approx \ell(Q_k), \; \dist(Q_k, Q_{k'}) \lesssim \ell(Q_k)\big\} \lesssim 1.\]

These facts allow us to finish this case by computing 
\[
\sum_{Q \in \F_2} |Q| 
=
\sum_{\widetilde{Q} \in \mathcal{B}'} \sum_{\substack{Q_k \in \F_2 \\ (Q_k)_{\text{out}} \subseteq \widetilde{Q}}} |Q_k|
\approx 
\sum_{\widetilde{Q} \in \mathcal{B}'} \sum_{\substack{Q_k \in \F_2 \\ (Q_k)_{\text{out}} \subseteq \widetilde{Q}}} |(Q_k)_{\text{out}}^\star|
\lesssim 
\sum_{\widetilde{Q} \in \mathcal{B}'} |\widetilde{Q}| 
\leq 
|Q(\sbf)|.
\]
In this line, the first equality is an elementary rearrangement since $Q \in \F_2$ implies $\Qout \subseteq \widetilde{Q}$, then we use $\ell((Q_k)_{\text{out}}^\star) \approx \ell((Q_k)_{\text{out}}) \approx \ell(Q)$, later we use the bounded overlap discussed in the last paragraph, and we finish by disjointness of the family $\mathcal{B}'$.

Therefore, we are left to prove the claim \eqref{QstarwhitQtil.eq}. First, the lower bound comes from \eqref{QstarwhitQp.eq}, since $Q' \subseteq \widetilde{Q}$. In turn, for the upper bound, we break into cases:
\begin{itemize}
	\item Case $\widetilde{Q} \cap \Qin = \emptyset$. Then $\Qin \subseteq \widetilde{Q}$, which allows us to compute (using also \eqref{QQpprop.eq} and \eqref{QstarwhitQp.eq})
	\[ \dist(\Qout^\star, (\widetilde{Q})^c) 
	\leq 
	\dist(\Qout^\star, \Qin)
	\leq 
	\diam(\Qout) + \dist(\Qout, \Qin)
	\lesssim 
	\ell(\Qout)
	\approx
	\ell(\Qout^\star).
	 \]
	 \item Case $\widetilde{Q} \subseteq \Qin$. Then we have, using \eqref{QQpprop.eq} and \eqref{QstarwhitQp.eq},
	 \[\Qout^\star \subseteq \widetilde{Q} \subseteq \Qin, 
	 \quad \text{ and } \quad 
	 \ell(\Qin) \lesssim \ell(\Qout) \lesssim \ell(\Qout^\star),
	 \]
	 which clearly implies that $\ell(\widetilde{Q}) \approx \ell(\Qout^\star)$. With this and the fact that $\mbf{x}_{\Qout} \in \Qout^\star \subseteq \widetilde{Q}$, we finish by this simple computation:
	 \[
	 \dist(\Qout^\star, (\widetilde{Q})^c)
	 \leq 
	 \dist(\mbf{x}_{\Qout}, \partial \widetilde{Q})
	 \leq 
	 \diam(\widetilde{Q})
	 \approx 
	 \ell(\widetilde{Q})
	 \approx 
	 \ell(\Qout^\star)
	 \approx 
	 \ell(\Qout).
	 \]
\end{itemize}
Actually, since $\Qin \in \sbf'$ and $\widetilde{Q} \in \mathcal{B}'$ (minimal cube in $\sbf'$), there are no more cases to be considered. Therefore, we are done with claim \eqref{QstarwhitQtil.eq}, and hence with Case $\F_2$.

\textbf{Case $\F_3$.} Now we conclude by handling the collection $\F_3$, but this is done in a similar manner to $\F_2$. In this case, if $Q \in \F_3$, then we claim that
\begin{equation}\label{QstarwhitQS.eq}
	\ell(\Qout^\star) \lesssim \dist(\Qout^\star, Q(\sbf)) \lesssim \ell(\Qout^\star).
\end{equation}
Indeed, the lower bound follows from \eqref{QstarwhitQp.eq} and the fact that $Q(\sbf) \subseteq (\Qout)^c$. And for the upper bound, note that $\Qin \subseteq Q(\sbf)$ because $\Qin \in \sbf' = \sbf_{N_1, N_2}$. Then, using also \eqref{QQpprop.eq} and \eqref{QstarwhitQp.eq},
\[
\dist(\Qout^\star, Q(\sbf))
\leq 
\dist(\Qout^\star, \Qin)
\leq 
\diam(\Qout) + \dist(\Qout, \Qin)
\lesssim 
\ell(\Qout)
\approx 
\ell(\Qout^\star).
\]

Now, from \eqref{QstarwhitQS.eq} and the fact that $\ell(\Qout^\star) \lesssim \ell(Q(\sbf))$ (because $\Qin \in \sbf$, and also \eqref{QQpprop.eq} and \eqref{QstarwhitQp.eq}), it is easy to deduce that there exists a uniform $C \geq 1$ such that $\Qout^\star \subseteq CQ(\sbf)$. This and the fact that $\{ \Qout^\star \}$ have bounded overlap (which can be proved in a similar manner as in Case $\F_2$, this time using \eqref{QstarwhitQS.eq} instead of \eqref{QstarwhitQtil.eq}) readily yield the desired estimate
\[\sum_{Q \in \mathcal{F}_3 } |Q| 
\lesssim 
\sum_{Q \in \mathcal{F}_3 } |\Qout^\star|
\lesssim 
|CQ(\sbf)|
\lesssim
|Q(\sbf)|.\]
This finishes Case $\F_3$, and ultimately the proof of \eqref{cutoffbds2.eq}, and thus of Lemma~\ref{cutofffnlem.lem}.

\section{Comments on smoothness assumptions}\label{extendrmks.sect}

In this last appendix, we describe how one can extend the main results in this paper to rougher matrices than the ones considered in our Definition~\ref{smoothL1carl.def}, because concretely we have assumed them to be smooth. In fact, we will obtain a satisfactory extension of Theorem~\ref{main1.thrm} in the case that the coefficients $A$ are symmetric; and at the end of the section, we discuss how one might extend the results to non-symmetric matrices (for a possible extension of Theorem~\ref{main2.thrm}). 

For brevity in the notation, in the setting of Theorem~\ref{main1.thrm}, we denote
\[\delta(\mbf{X}) := \dist(\mbf{X}, \partial\Omega),
\quad \text{ and } \quad  \osc(A,\mbf{X}) := \sup_{\mbf{Z}, \mbf{Y} \in \mathcal{J}_{\delta(\mbf{X})/(1000\sqrt{n+1})}(\mbf{X})} |A(\mbf{Y}) - A(\mbf{Z})|.\]
Then, the following extension of Theorem \ref{main1.thrm} can be made.
\begin{theorem}\label{main1ext.thrm}
Let $n \geq 2$. Suppose that $\Omega$ is the region above a $\Lip(1,1/2)$ function $\psi$ (see Definition \ref{Lipnot}), and $A$ is a \textbf{symmetric} elliptic matrix satisfying \eqref{ellip.eq} and the following Carleson condition
\begin{equation}\label{weakcarlcond.eq}
\sup_{\mbf{X} \in \partial\Omega} \; \sup_{\rho > 0} \;\; \rho^{-n - 1} \iiint_{\mathcal{J}_{\rho}(\mbf{X}) \cap \Omega} \frac{\osc(A, \mbf{Y})}{\delta(\mbf{Y})} \, \d \mbf{Y} < \infty.
\end{equation}
Assume further that there exists $1 < p < \infty$ such that the $L^p$ Dirichlet problem is solvable for $\mathcal{L}$ (see Definition~\ref{Lpsolv.def}), where
\[\mathcal{L}f := \partial_t f - \div_X(A(X,t) \nabla_X f).\]
Then $D_t^{1/2}\psi \in \BMO$ (the parabolic BMO space, see Section \ref{sec:BMO}). Equivalently, the $L^p$ solvability of the Dirichlet problem implies that the graph of $\psi$ is parabolic uniformly rectifiable. 
\end{theorem}
\begin{remark}
Note that \eqref{weakcarlcond.eq} holds whenever $A$
satisfies the Carleson condition
\begin{equation*}
\sup_{\mbf{X} \in \partial\Omega} \; \sup_{\rho > 0} \;\; \rho^{-n - 1} \iiint_{\mathcal{J}_{\rho}(\mbf{X}) \cap \Omega} \;\; \sup_{\mbf{Y} \in \mathcal{J}_{\delta(\mbf{X})/(1000\sqrt{n+1})}(\mbf{X})} \!\! \big( |\nabla A(\mbf{Y})| + |\partial_t A(\mbf{Y})|\delta(\mbf{Y}) \big) \, \d \mbf{Y} < \infty,
\end{equation*}
which is actually a stronger condition than our Carleson measure assumption in Theorem~\ref{main1.thrm}, namely \eqref{carlesonmeasure.eq} (indeed, the measure in \eqref{eq:carleson_measure} does not have the supremum above). Therefore, Theorem~\ref{main1ext.thrm} extends Theorem~\ref{main1.thrm}.
\end{remark}

\begin{proof}
The idea is to regularize $A$ and use some perturbative argument. For that, first let $P_r$ be the parabolic approximate identity on $\mathbb{R}^{n+1}$ from Appendix \ref{cutoff.sect} (see \eqref{prdef1'}). We then let $\delta'(\cdot)$ be a regularized version of $\delta(\cdot)$ (we can use \cite[Lemma 3.24]{BHHLN-CME}, which is based on \cite[Chapter VI, \textsection 1 \& 2]{stein-SIOs}) so that
\[c_n\delta(\mbf{X}) \le \delta'(\mbf{X}) \le \delta(\mbf{X}), \qquad \delta(\mbf{X})^{k-1}|\nabla_X^k \delta'(\mbf{X})| + \delta(\mbf{X})^{2k -1}|\partial^k_t \delta'(\mbf{X})| \lesssim_{k,n} 1, 
\qquad \mbf{X} \in \Omega.\]
Then for $c := (10^5\sqrt{n+1})^{-1}$ we define
\[\widetilde{A}(\mbf{X}) 
:= 
p_{c\delta'(\mbf{X})} \ast A (\mbf{X})
:=
\iiint p_{c\delta'(\mbf{X})}(\mbf{Y}) A(\mbf{X} - \mbf{Y}) \,\d\mbf{Y}
,\]
where the integral is taken component-wise.
Then it is not hard to show that
\[|\nabla_X \widetilde{A}(\mbf{Y})|  + |\partial_t \widetilde{A}(\mbf{Y})|\delta(\mbf{Y}) \lesssim  \frac{\osc(A, \mbf{Y})}{\delta(\mbf{Y})}, 
\qquad \mbf{Y} \in \Omega,\]
and, moreover,
\[\delta(\mbf{Y})^k|\nabla_X^k \widetilde{A}(\mbf{Y})| + \delta(\mbf{Y})^{2k} |\partial_t^k \widetilde{A}(\mbf{Y})| \lesssim_{n,k} \|A\|_{L^\infty}, \qquad k = 1,2,\dots, \; \mbf{Y} \in \Omega.\]

These (and the hypothesis \eqref{weakcarlcond.eq}) easily imply that $\widetilde{A}$ satisfies the $L^1$-Carleson oscillation condition in Definition \ref{smoothL1carl.def} (and it is elementary to see that it is also elliptic). Therefore, since we have not modified $\Omega$ nor $\Sigma$, to prove Theorem \ref{main1ext.thrm} it suffices to show that there exists $1 < q < \infty$ such that the $L^q$ Dirichlet problem is solvable for $\widetilde{\mathcal{L}}$, where
\[\widetilde{\mathcal{L}}f := \partial_t f - \div_X(\widetilde{A}(X,t) \nabla_X f).\]

To this end, we use the already established perturbative theory in \cite{N97}, which requires us to compare $A$ to $\widetilde{A}$. For that purpose, define
\[a(\mbf{Y}) := \sup_{\mbf{Z} \in  \mathcal{J}_{\delta(\mbf{X})/(10^{20}\sqrt{n+1})}(\mbf{Y}) } |A(\mbf{Z}) - \widetilde{A}(\mbf{Z})|,
\qquad \mbf{Y} \in \Omega.\]
Then, it is easy to show that 
\[ \|a\|_{L^\infty} \lesssim \|A\|_{L^\infty}, 
\quad \text{ and } \quad 
a(\mbf{Y}) \lesssim \osc(A, \mbf{Y}), \quad \mbf{Y} \in \Omega, \]
which imply, jointly with \eqref{weakcarlcond.eq}, the estimate
\[\sup_{\mbf{X} \in \partial\Omega} \; \sup_{\rho > 0} \;\; \rho^{-n - 1} \int_{\mathcal{J}_{\rho}(\mbf{X}) \cap \Omega} \frac{a(\mbf{Y})^2}{\delta(\mbf{Y})} \, \d \mbf{Y} < \infty.\]
With this, by the perturbative theory in \cite[Theorem 6.4]{N97}, there exists $1 < q < \infty$ such that the $L^q$ Dirichlet problem is solvable for $\widetilde{\mathcal{L}}$, which finishes the proof.
\end{proof}

Finally, we close with a remark that while in this section we assumed that $A$ was symmetric, it seems likely that the results in \cite{N97} hold for non-symmetric $A$. Indeed, this is the case in the elliptic setting.

\section{A linear algebra fact about elliptic matrices} \label{appendix:matrices}

In this last appendix, we will show that \eqref{hessianAndMixedIndices.eq} holds because actually the stronger pointwise estimate $|\nabla^2 u|^2 \approx (a_{i, j}u_{x_i, x_k})(a_{k, \ell}u_{x_j, x_\ell})$ is true (along the section, we use the summation convention over repeated indices). In fact, we will prove a more general result for matrices. 

\begin{lemma} \label{matrices.lem}
	Let $A, B$ be $n \times n$ real and symmetric matrices. Suppose further that $A$ is elliptic and bounded, i.e. $\lambda \in [\Lambda^{-1}, \Lambda]$ for every eigenvalue $\lambda$ of $A$, for some $\Lambda > 0$. Then 
	\begin{equation*}
		|B|^2
		\approx 
		a_{i, j} b_{i, k} a_{k, \ell} b_{j, \ell},
	\end{equation*}
	with implicit constant depending only on $\Lambda$. Here and everywhere along the section, $|X|$ denotes the Hilbert-Schmidt norm of the matrix $X$, i.e. $|X|^2 := \sum_{i, j=1}^n |x_{i, j}|^2$. 
\end{lemma}

Indeed, specializing the previous theorem to $B = \nabla^2 u$, which is symmetric by the discussion in Subsection~\ref{sec:green_whitney}, and $A$ is of course our coefficient matrix, we obtain \eqref{hessianAndMixedIndices.eq}. The proof of the lemma will follow from a sequence of easier intermediate results.

\begin{lemma}\label{multiplyMatrices.lem}
	Let $\widetilde{A}$ be as in Lemma~\ref{matrices.lem} and $B$ be any $n\times n$ matrix (not necessarily symmetric). Then (with implicit constants depending only on $\Lambda$)
	\[
	|\widetilde{A}B| \approx |B| \approx |B \widetilde{A}|.
	\]
\end{lemma}
\begin{proof}
	Let us only show the first equivalence, the other one being similar. Since $\widetilde{A}$ is symmetric, we can use the spectral theorem to diagonalize it, i.e., there exists an orthogonal matrix $U$ and a diagonal matrix $D = \text{diag}(\lambda_1, \ldots, \lambda_n)$ such that $\widetilde{A} = U^T D U$. In fact, the assumption on the eigenvalues of $\widetilde{A}$ implies that $\lambda_i \approx_{\Lambda} 1$ for $i = 1, \ldots, n$. This easily implies that for any vector $v \in \RR^n$, $|Dv| \approx_{\Lambda} |v|$. Hence, computing the norm of any matrix $X$ by testing against vectors (i.e. $|X| = \sup_{v \in \RR^n : |v| \leq 1} |Xv|$), we also obtain $|DX| \approx_{\Lambda} |X|$. 	
	Using this and the fact that multiplication by orthogonal matrices preserves the Hilbert-Schmidt norm (recall that they are isometries in the euclidean space), we finish by computing
	\[
	|B| = |U B| \approx_{\Lambda} |D U B| = |U^T D U B| = |\widetilde{A} B|.
	\]
\end{proof}

\begin{corollary}\label{squareRootMatrix.corol}
	Let $A, B$ as in Lemma~\ref{matrices.lem}. Then (with implicit constant depending only on $\Lambda$)
	\[
	|B| \approx |A^{1/2} B A^{1/2}|.
	\]
\end{corollary}
\begin{proof}
	It is an elementary linear algebra fact that symmetric elliptic matrices (i.e. those whose eigenvalues are positive) like $A$ have a (unique) square root matrix $A^{1/2}$ which is also symmetric. It is also elementary that the eigenvalues of $A^{1/2}$ lie in $[\Lambda^{-1/2}, \Lambda^{1/2}]$. Therefore, the result follows by applying Lemma~\ref{multiplyMatrices.lem} twice with $\widetilde{A} = A^{1/2}$. 
\end{proof}

\begin{proof}[Proof of Lemma~\ref{matrices.lem}]
	Let $\tilde{a}_{i, j}$ denote the entries in $A^{1/2}$. Then, clearly the entries in $A^{1/2}BA^{1/2}$ are 
	\[
	(A^{1/2}BA^{1/2})_{\alpha, \beta}
	=
	\tilde{a}_{\alpha, i} b_{i, k} \tilde{a}_{k, \beta}.
	\]
	And using the symmetry of both $B$ and $A^{1/2}$, we similarly obtain 
	\[
	(A^{1/2}BA^{1/2})_{\alpha, \beta}
	=
	\tilde{a}_{\beta, \ell} b_{\ell, j} \tilde{a}_{j, \alpha}.
	\]
	Moreover, note also that since $A^{1/2}A^{1/2} = A$, then we have
	\[
	a_{k, \ell} = \tilde{a}_{k, m} \tilde{a}_{m, \ell}, 
	\qquad 
	\text{for } k, \ell = 1, \ldots, n.
	\]
	Using all the above facts together and Corollary~\ref{multiplyMatrices.lem}, we finish by computing
	\[
		|B|^2 
		\approx_{\Lambda}
		|A^{1/2} B A^{1/2}|^2
		= 
		\sum_{\alpha, \beta=1}^n ((A^{1/2}BA^{1/2})_{\alpha, \beta})^2
		=
		 \tilde{a}_{\alpha, i} b_{i, k} \tilde{a}_{k, \beta} \tilde{a}_{\beta, \ell} b_{\ell, j} \tilde{a}_{j, \alpha}
		\\ =
		a_{j, i} a_{k, \ell} b_{i, k} b_{\ell, j},
	\]
	which finishes the proof accounting for the symmetry of both $A$ and $B$.
\end{proof}

\newcommand{\etalchar}[1]{$^{#1}$}

\end{document}